\documentclass[reqno]{amsart}
\usepackage{amssymb}
\usepackage{graphicx}
\usepackage{amsmath}
\usepackage{mathbbol}

\usepackage[usenames, dvipsnames]{color}
\usepackage{verbatim}
\usepackage{mathrsfs}
\usepackage{bm}
\usepackage{cite}
\usepackage{color}
\allowdisplaybreaks[4]

\numberwithin{equation}{section}

\newtheorem{theorem}{Theorem}[section]

\newtheorem{lemma}[theorem]{Lemma}
\newtheorem{prop}[theorem]{Proposition}

\theoremstyle{definition}
\newtheorem{remark}[theorem]{Remark}

\theoremstyle{definition}

\theoremstyle{definition}

\makeatletter
\def\dashint{\operatorname%
{\,\,\text{\bf-}\kern-.98em\DOTSI\intop\ilimits@\!\!}}
\makeatother

\def\\det{\text{\det}}

\def\.5{\frac{1}{2}}

\newcommand{\RN}[1]{%
  \textup{\uppercase\expandafter{\romannumeral#1}}%
}

\newcommand{\dist}{\text{dist}}

\renewcommand{\epsilon}{\varepsilon}

\newcounter{marnote}

\begin{document}

\title[Estimates for stress concentration of the  Stokes flow ]{Estimates for Stress Concentration between Two Adjacent Rigid Inclusions in Stokes Flow}

\author[H.G. Li]{Haigang Li}
\address[H.G. Li]{School of Mathematical Sciences, Beijing Normal University, Laboratory of MathematiCs and Complex Systems, Ministry of Education, Beijing 100875, China.}
\email{hgli@bnu.edu.cn}

\author[L.J. Xu]{Longjuan Xu}
\address[L.J. Xu]{Academy for Multidisciplinary Studies, Capital Normal University, Beijing 100048, China.}
\email{ljxu311@163.com}

%\footnote{}

\date{\today} % delete this line to display the current date

%%% BEGIN DOCUMENT
\begin{abstract}
In this paper, we establish the estimates for the gradient and the  second-order partial derivatives for the Stokes flow in the presence of two closely located strictly convex inclusions in dimension three. Moreover, the blow-up rate of the gradient is showed to be optimal by a pointwise upper bound and a lower bound in the narrowest region. We also show the optimal blow-up rate of Cauchy stress tensor. In dimensions greater than three, the upper bounds of the gradient are established. These results answer the questions raised in \cite{K}.
\end{abstract}

\maketitle
%\tableofcontents

\section{Introduction}
Let $D$ be a bounded open set in $\mathbb{R}^{d}$, $d\geq2$, that contains two adjacent rigid particles $D_{1}$ and $D_{2}$, with $\varepsilon$ apart and far away from $\partial D$, that is,
\begin{equation*}%\label{assump D1 D2}
\begin{split}
\overline{D}_{1},\overline{D}_{2}\subset D,\quad
\varepsilon:=\mbox{dist}(D_{1},D_{2})>0,\quad\mbox{dist}(D_{1}\cup D_{2},\partial D)>\kappa_{0}>0,
\end{split}
\end{equation*}
where $\kappa_{0}$ is a constant independent of $\varepsilon$. We assume that $D_{1}$ and $D_{2}$ are of class $C^{3}$, and their $C^{3}$ norms are bounded by another positive constant $\kappa_{1}$, independent of $\varepsilon$. Denote the linear space of rigid displacements in $\mathbb{R}^{d}$:
$$\Psi:=\Big\{{\boldsymbol\psi}\in C^{1}(\mathbb{R}^{d};\mathbb{R}^{d})~|~e({\boldsymbol\psi}):=\frac{1}{2}(\nabla{\boldsymbol\psi}+(\nabla{\boldsymbol\psi})^{\mathrm{T}})=0\Big\},$$
with a basis $\{\boldsymbol{e}_{i},~x_{k}\boldsymbol{e}_{j}-x_{j}\boldsymbol{e}_{k}~|~1\leq\,i\leq\,d,~1\leq\,j<k\leq\,d\}$, where $e_{1},\dots,e_{d}$ is the standard basis of $\mathbb{R}^{d}$. Denote this basis of $\Psi$ as $\{\boldsymbol\psi_{\alpha}\}$, $\alpha=1,2,\dots,\frac{d(d+1)}{2}$. Let us consider the following Stokes flow:
\begin{align}\label{sto}
\begin{cases}
\mu \Delta {\bf u}=\nabla p,\quad\quad~~\nabla\cdot {\bf u}=0,&\hbox{in}~~\Omega:=D\setminus\overline{D_1\cup D_2},\\
{\bf u}|_{+}={\bf u}|_{-},&\hbox{on}~~\partial{D_{i}},~i=1,2,\\
e({\bf u})=0, &\hbox{in}~~D_{i},~i=1,2,\\
\int_{\partial{D}_{i}}\frac{\partial {\bf u}}{\partial \nu}\Big|_{+}\cdot{\boldsymbol\psi}_{\alpha}-\int_{\partial{D}_{i}}p\,{\boldsymbol\psi}_{\alpha}\cdot
\nu=0
,&i=1,2,\\
{\bf u}={\boldsymbol\varphi}, &\hbox{on}~~\partial{D},
\end{cases}
\end{align}
where $\mu>0$, $\alpha=1,2,\dots,\frac{d(d+1)}{2}$, 
$\frac{\partial {\bf u}}{\partial \nu}\big|_{+}:=\mu(\nabla {\bf u}+(\nabla {\bf u})^{\mathrm{T}})\nu,$
and $\nu$ is the  unit outer normal vector of $D_{i},~i=1,2$. Here and throughout this paper the subscript $\pm$ indicates the limit from outside and inside the domain, respectively. Moreover, it follows from $\nabla\cdot{\bf u}=0$ and Gauss theorem  that the prescribed velocity field ${\boldsymbol\varphi}$ satisfy the compatibility condition:
\begin{equation}\label{compatibility}
\int_{\partial{D}}{\boldsymbol\varphi}\cdot \nu\,=0.
\end{equation}
With \eqref{compatibility}, the existence and uniqueness of the solutions for the Stokes flow in a bounded domain follows from the argument in \cite{Lady1959} by an integral variational formulation and Riesz representation theorem. The regularity can be obtained directly from the general theory of Agmon, Douglis and Nirenberg \cite{ADN1964} and Solonnikov \cite{Solonni1966} since the Stokes flow is elliptic in the sense of Douglis-Nirenberg \cite{T1984}.

The main purpose of this paper is to establish the estimates of stress concentration for the Stokes flow \eqref{sto} with  \eqref{compatibility} in three dimensions and higher dimensions, as the distance between inclusions tends to zero, which answers the questions raised by Kang \cite{K} where he mentioned that the extensions to general shape of inclusions and higher dimensions $d\geq3$ are quite challenging. The corresponding two-dimensional problem was studied by Ammari, Kang, Kim, and Yu \cite{AKKY} for cylinder inclusions, and by Li and Xu \cite{LX2} for general convex inclusions. Compared to the two-dimensional case \cite{LX2}, more new difficulties are needed to overcome and the analysis is more involved in three dimensions and higher dimensions.

\subsection{Background}\label{subsec1.1}
It is vital important to study the field enhancements in the narrow region between two rigid particles in material sciences and fluid mechanics. In practice, complex fluids including particle suspensions usually result in complicated flow behavior and characteristic rheological properties. Many applied mathematicians and physicists made important progress in this field. By the beginning of the twentieth century, hydrodynamic lubrication theory was well advanced, based on the work of Navier and Stokes \cite{H1994}. The Reynolds equation \cite{Reynolds1886} had been widely validated for continuous liquid films acting under restricted conditions \cite{Oscar1987}. It is proved that the fluid film is thick enough to prevent physical or direct contact between the friction-coupled elements. Elastohydrodynamic lubrication was further developed by Dowson and Higginson \cite{DH1966} and by Winer and Cheng \cite{Cheng1967} in the latter half of the twentieth century, which takes into account the elastic deflection of solid contacting surfaces. Due to very high pressures at the contact interface, the increase in viscosity ensures continuity of the fluid film in a bearing contact.

In this paper we assume that the Reynolds number is low to ignore the nonlinear inertial force, whereas at high Reynolds number the boundary layers problem of  is another very important topic. Beginning with Reynolds famous paper in 1883, the stability and transition to turbulence of laminar flows at high Reynolds number has been an active field in fluid mechanics \cite{Yag2012}. In fluid lubrication theory, thin layers of fluid can prevent solid bodies from contact, and so a solid body close to a solid plane or two nearly parallel surface are also often considered \cite{Cope1949}. For instance, by using spherical polars, one can calculate the force exerted by a sphere immersed in unbounded fluid which is at rest at infinity \cite{Lady,Lamb}. For the simply-connected exterior domain, Odqvist \cite{Odqvist} proved that a solution can be written as suitable potentials of double-layer.
However, when two spheres are very closely spaced and may even touch, it is a question of particular interest  whether the stresses remain uniformly bounded. Actually, it was an analogous question concerning linear elasticity problem posed by Ivo Babuska \cite{Bab} that initially piqued our interest in this problem area.

In material sciences, we would like to point out the connection to the field enhancement in composites in the electro-static case and in the elasto-static case. In \cite{Bab}, Babu\u{s}ka et al.  observed that if the distance $\varepsilon$ between inclusions tends to zero, the stress (the gradient of solutions) of the systems of linear elasticity 
$$\mu\Delta{\bf u}+(\lambda+\mu)\nabla(\nabla\cdot{\bf u})=0$$ remains bounded. By developing rigorous analysis, Bonnetier and Vogelius \cite{BV}, Li and Vogelius \cite{LV}, Li and Nirenberg \cite{LN} proved the numerical observation in \cite{Bab}. The bounded estimates in \cite{BV,LV,LN}  depend on the ellipticity of the coefficients. In order to figure out the enhancement caused by the smallness of the interparticles, the coefficients occupying in the inclusions are mathematically assumed to degenerate to infinity. In  the context of the Lam\'e systems with partially infinite coeffcients,  the blow-up rate of the stress is $\varepsilon^{-1/2}$ in dimension two \cite{BLL,KY}, $(\varepsilon|\ln\varepsilon|)^{-1}$ in dimension three \cite{BLL2,Li2021,LX}.  More results can be found in \cite{CL}. The corresponding results for the scalar case, describing the electrical field enhancement by the perfect conductivity problem in the electro-static case, see, for example, \cite{AKL,AKL3,AKKY,BC,BLY,BT,KLeY,KLiY,KLiY2,KangYu,Li,Yun,Yun2}. We also refer the reader to  \cite{CS1,CS2} about cases of $p$-Laplacian and Finsler Laplacian. 

The main purpose of this paper is to investigate the interaction between two adjacent particles in a viscous incompressible fluid, modelled by \eqref{sto}, when the distance $\varepsilon$ tends to zero.  Indeed, even though in numerical simulation the hydrodynamic interactions among particles is not easy to treat. So in the previous studies the interparticle distance is usually assumed not relatively small to avoid the singularity caused by the interaction \cite{AGKL,AGGJS,BMT,GH}. However, in densely packed suspensions such singularity between multiple bodies play a significant role in liquid-solid model where the small interparticle distance occurs.  Therefore, the investigation of the suspension problem is an essentially important issue in the analysis of complex fluids. There are also wide applications in engineering, fluid mechanics, and material sciences, see e.g. \cite{G1994,S1989,SB2001,SB2002}. Here we would like to mention a very recent interesting result on this topic by Ammari et al \cite{AKKY} where $D_{1}$ and $D_{2}$ are assumed to be two adjacent disks in the two-dimensional steady Stokes system, they first derived an asymptotic representation formula for the stress and completely captured the singular behavior of the stress by using the bipolar coordinates and showed the blow-up rate is $\varepsilon^{-1/2}$ in dimension two, the same as the linear elasticity case. However, as far as we know it is difficult to extend to study the 3D problem. We fix it in this paper by proving the optimal blow-up rate of Cauchy stress tensor, as well as the upper bounds of the first and second order derivatives of the solutions  in three dimension and higher dimensions. We would like to point out that our method works well for the convex inclusions with arbitrary shape. 

\subsection{Assumptions on $\partial D_1$ and $\partial D_2$}\label{subsec1.2}
Let $D_{1}^{0}$ and $D_{2}^{0}$ be a pair of (touching at the origin) $C^3$-convex subdomains of $D$, far away from $\partial D$, and satisfy
\begin{equation}\label{def_D0}
D_{1}^{0}\subset\{(x', x_{d})\in\mathbb R^{d}~:~ x_{d}>0\},\quad D_{2}^{0}\subset\{(x', x_{d})\in\mathbb R^{d}~:~ x_{d}<0\},
\end{equation}
with $\{x_d=0\}$ being their common tangent plane, after a rotation of coordinates if necessary. This means that the axes $x_{1},\dots,x_{d-1}$ are in the tangent plane. Here and throughout this paper, we use superscripts prime to denote the ($d-1$)-dimensional variables and domains, such as $x'$ and $B'$. 
Translate $D_{i}^{0}$ ($i=1,\, 2$) by $\pm\frac{\varepsilon}{2}$ along $x_{d}$-axis in the following way
\begin{equation*}%\label{def-D1D2-1}
D_{1}^{\varepsilon}:=D_{1}^{0}+(0',\frac{\varepsilon}{2}),\quad \text{and}\quad D_{2}^{\varepsilon}:=D_{2}^{0}+(0',-\frac{\varepsilon}{2}).
\end{equation*}
Then we have $\varepsilon=\mbox{dist}(\partial D_1,\partial D_2)$. For simplicity of notation, we drop the superscript $\varepsilon$ and denote
\begin{equation*}%\label{def-D1D2-2}
D_{i}:=D_{i}^{\varepsilon}\, (i=1, \, 2), \quad  \Omega:=D\setminus\overline{D_1\cup D_2}.
\end{equation*}
Denote $P_1:= (0',\frac{\varepsilon}{2})$ and $P_2:=(0',-\frac{\varepsilon}{2})$ by the two nearest points between $\partial D_1$ and $\partial D_2$ such that
$\varepsilon=\text{dist}(P_1, P_2)=\text{dist}(\partial D_1, \partial D_2)$. 

Since we assume that $\partial D_1$ and $\partial D_2$ are of $C^{3}$, then near the origin there exists a constant $R$, independent of $\varepsilon$, such that the portions of $\partial D_1$ and $\partial D_2$ are represented, respectively, 
by graphs 
\begin{equation*}%\label{h1h2'}
x_d=\frac{\varepsilon}{2}+h_1(x')\quad\text{and}\quad x_d=-\frac{\varepsilon}{2}-h_2(x'),\quad \text{for}~ |x'|\leq 2R,
\end{equation*}
where $h_1$, $h_2\in C^{3}(B'_{2R}(0'))$ and satisfy 
\begin{align}
&-\frac{\varepsilon}{2}-h_{2}(x') <\frac{\varepsilon}{2}+h_{1}(x'),\quad\mbox{for}~~ |x'|\leq 2R,\label{h1-h2}\\
&h_{1}(0')=h_2(0')=0,\quad \nabla_{x'} h_{1}(0')=\nabla_{x'}h_2(0')=0,\label{h1h1}\\
&h_{1}(x')=h_2(x')=\frac{\kappa_{2}}{2}|x'|^{2}+O(|x'|^{3}),\quad\mbox{for}~~|x'|<2R,\label{h1h14}
\end{align}
where the constant $\kappa_{2}>0$. For simplicity of computation, here we assumed that near the origin $\partial D_1$ and $\partial D_2$ are $2$-convex, and symmetric with $x_{3}$-axis. Our method can be applied to deal with the general convex inclusion cases, like $h_{1}(x')=h_2(x')=\frac{\kappa_{2}}{2}x_1^{2}+\frac{\kappa'_{2}}{2}x_2^{2}+O(|x'|^{3})$, and $h_{1}(x')=h_2(x')=\frac{\kappa_{m}}{2}|x'|^{m}+O(|x'|^{m+1})$, $m\geq2$. For $0\leq r\leq 2R$, let us define the neck region between $D_{1}$ and $D_{1}$ by
\begin{equation*}%\label{narrow-region}
\Omega_r:=\left\{(x',x_{d})\in \Omega~:~ -\frac{\varepsilon}{2}-h_2(x')<x_{d}<\frac{\varepsilon}{2}+h_1(x'),~|x'|<r\right\}.
\end{equation*}

Because the difficulties mainly focus on the neck region $\Omega_{2R}$, without loss of generality, we may assume that ${\boldsymbol\varphi}\in C^{2,\alpha}(\partial D;\mathbb R^3)$ for some $0<\alpha<1$, and $\partial D$ is of class $C^{3}$.  Indeed, the $H^1$ norm of the solution ${\bf u}$ in $\Omega$ is bounded by a universal constant. Then standard elliptic estimates \cite{ADN1964,Solonni1966} give a universal bound of ${\bf u}$ in $C^2$ norm in $\{x\in D~|~ \frac{\kappa_{0}}{10}<\dist(x,\partial D)<\frac{\kappa_{0}}{2}\}$. We deal with the problem instead in $D':=\{x\in D~ |~ dist(x,\partial D)>\frac{\kappa_{0}}{8}\}$ with ${\boldsymbol\varphi}:={\bf u}|_{\partial D'}$. Throughout this paper, we say a constant is {\em{universal}} if it depends only on $d,\mu,\kappa_{0},\kappa_{1},\kappa_{2}$, and the upper bounds of the $C^{3}$ norms of $\partial{D}$, $\partial{D}_{1}$ and $\partial{D}_{2}$, but independent of $\varepsilon$.  

\subsection{Upper Bounds of $|\nabla{\bf u}|$ and $|p|$ in 3D}\label{subsection-upp}

Our first main result in dimension three is as follows, including estimates for the gradient and estimates for the second order derivatives.

\begin{theorem}\label{mainthm}(Upper Bounds in 3D)
Assume that $D_1,D_2,D,\Omega$ and $\varepsilon$ are defined as above, and ${\boldsymbol\varphi}\in C^{2,\alpha}(\partial D;\mathbb R^3)$ for some $0<\alpha<1$. Let ${\bf u}\in H^1(D;\mathbb R^3)\cap C^2(\bar{\Omega};\mathbb R^3)$ and $p\in L^2(D)\cap C^1(\bar{\Omega})$ be the solution to \eqref{sto} and \eqref{compatibility}. Then for sufficiently small $0<\varepsilon<1/2$, we have the following assertions.

(i) For the gradient, 
\begin{align}
&|\nabla{{\bf u}}(x)|\leq
\frac{C (1+|\ln\varepsilon||x'|)}{|\ln\varepsilon|(\varepsilon+|x'|^{2})}\|{\boldsymbol\varphi}\|_{C^{2,\alpha}(\partial D;\mathbb R^3)},\quad~x\in\Omega_{R},\nonumber\\%\label{thm_u3D}\\
&\inf_{c\in\mathbb{R}}\|p+c\|_{C^0(\bar{\Omega}_{R})}\leq \frac{C}{\varepsilon^{3/2}|\ln\varepsilon|}\|{\boldsymbol\varphi}\|_{C^{2,\alpha}(\partial D;\mathbb R^3)},\label{thm_p3D}
\end{align}
and
$$\|\nabla{\bf u}\|_{L^{\infty}(\Omega\setminus\Omega_{R})}+\|p\|_{L^{\infty}(\Omega\setminus\Omega_{R})}\leq\,C\|{\boldsymbol\varphi}\|_{C^{2,\alpha}(\partial D;\mathbb R^3)},$$
where $C$ is a universal constant, independent of $\varepsilon$. In particular, 
\begin{equation*}
\|\nabla{{\bf u}}\|_{L^{\infty}(\Omega)}\leq \frac{C}{\varepsilon|\ln\varepsilon|}\|{\boldsymbol\varphi}\|_{C^{2,\alpha}(\partial D;\mathbb R^3)}.
\end{equation*}

(ii) For the second order derivatives,
\begin{align*}
|\nabla^2{\bf u}(x)|+|\nabla p(x)|\leq
\frac{C(1+|\ln\varepsilon||x'|)}{|\ln\varepsilon|(\varepsilon+|x'|^2)^{2}}\|{\boldsymbol\varphi}\|_{C^{2,\alpha}(\partial D;\mathbb R^3)},\quad~x\in\Omega_{R},
\end{align*}
and 
$$\|\nabla^2{\bf u}\|_{L^{\infty}(\Omega\setminus\Omega_{R})}+\|\nabla p\|_{L^{\infty}(\Omega\setminus\Omega_{R})}\leq\,C\|{\boldsymbol\varphi}\|_{C^{2,\alpha}(\partial D;\mathbb R^3)}.$$
\end{theorem}

\begin{remark}
Our strategy to establish the gradient estimates for the solution of Stokes system \eqref{Stokessys} in spirit follows \cite{BLL2}. However, here is another difficulty to obtain an appropriate estimate for $p$. To this end, we have to derive the estimates for the gradient of $p$ and the second-order derivatives of ${\bf u}$. The proof is more technical and involved in the constructions of auxiliary functions and during adopting the iteration approach. It is unfortunately rather long. We carry it out through two stages: $(i)$ Section \ref{sec2outline3D} for the framework, and $(ii)$ Section \ref{sec_estimate3D} for detailed procedure. 

We would like to emphasize that here the blowup rate of $|\nabla^2{\bf u}|$ is $(\varepsilon^2|\ln\varepsilon|)^{-1}$, which is larger than that for Lam\'e system. The reason is that the incompressible condition makes the constructions of auxiliary functions ${\bf v}_i^\alpha$, especially for $\alpha=5,6$, more complicated, which leads to $|\sum_{\alpha=5}^{6}(C_{1}^{\alpha}-C_{2}^{\alpha})\nabla^2{\bf v}_1^\alpha|$ be the main singular term of $|\nabla^2{\bf u}|$. More details can be found in Subsection \ref{subsec-constuction} and Subsection \ref{subsec-mainthm}. Besides, these second-order derivatives of ${\bf u}$ established here may be a key step to design a proper finite elements in numerical simulations, because there $H^2$ estimate usually is required.
\end{remark}

\begin{remark}
If ${\boldsymbol\varphi}=0$, then the solution to \eqref{sto} is ${\bf u}\equiv0$ and  Theorem \ref{mainthm} is trivial. So we only need to prove them for $\|{\boldsymbol\varphi}\|_{C^{2,\alpha}(\partial D;\mathbb R^3)}=1$, by considering ${\bf u}/\|{\boldsymbol\varphi}\|_{C^{2,\alpha}(\partial D;\mathbb R^3)}$.
\end{remark}

\subsection{Lower Bounds of $|\nabla{\bf u}|$ in 3D}\label{subsection-lower}

We are going to prove the lower bound of $|\nabla{\bf u}(x)|$ on the segment $\overline{P_1 P_2}$ which is the shortest line between $\partial D_1$ and $\partial D_2$, with the same order as showed in Theorem \ref{mainthm}, under some additional symmetric assumptions on the domain and the given boundary data, for simplicity. Suppose that 
\begin{align*}
({\rm S_{H}}): ~&~D_{1}\cup D_{2}~\mbox{ and~} D ~\mbox{are~ symmetric~ with~ repect~ to~ each~ coordinate~axis~}x_{i}, \\&\mbox{~and~}\mbox{the coordinate plane~} \{x_{3}=0\};
\end{align*}
and
\begin{align*}
({\rm S_{{\boldsymbol\varphi}}}):~~ {\boldsymbol\varphi}^{i}(x)=-{\boldsymbol\varphi}^{i}(-x),\quad\,i=1,2,3.\hspace{5.6cm}
\end{align*}

We introduce the Cauchy stress tensor
\begin{equation*}
\sigma[{\bf u},p]=2\mu e({\bf u})-p\mathbb{I},
\end{equation*}
where $\mathbb{I}$ is the identity matrix, and $({\bf u},p)$ is a pair of solution to \eqref{sto}. Then we reformulate \eqref{sto} as
\begin{align}\label{Stokessys}
\begin{cases}
\nabla\cdot\sigma[{\bf u},p]=0,~~\nabla\cdot {\bf u}=0,&\hbox{in}~~\Omega,\\
{\bf u}|_{+}={\bf u}|_{-},&\hbox{on}~~\partial{D_{i}},~i=1,2,\\
e({\bf u})=0, &\hbox{in}~~D_{i},~i=1,2,\\
\int_{\partial{D}_{i}}\sigma[{\bf u},p]\cdot{\boldsymbol\psi}_{\alpha}
\nu=0
,&i=1,2,\alpha=1,\dots,6,\\
{\bf u}={\boldsymbol\varphi}, &\hbox{on}~~\partial{D}.
\end{cases}
\end{align}
Let $\varepsilon=0$, and recall \eqref{def_D0}, 
$$D_{1}^{0}:=\{x\in\mathbb R^{3}~\big|~ x+P_{1}\in D_{1}\},\quad~ D_{2}^{0}:=\{x\in\mathbb R^{3}~\big|~ x-P_{2}\in D_{2}\},$$ 
and set $\Omega^{0}:=D\setminus\overline{D_{1}^{0}\cup D_{2}^{0}}$. Let us define a linear and continuous functional of ${\boldsymbol\varphi}$:
\begin{align}\label{blowupfactor}
\tilde b_{j}^{*\alpha}[{\boldsymbol\varphi}]:=
\int_{\partial D_j^0}{\boldsymbol\psi}_\alpha\cdot\sigma[{\bf u}^*,p^*]\nu,\quad\alpha=1,\dots,6,
\end{align}
where $({\bf u}^*,p^*)$ verifies
\begin{align}\label{maineqn touch}
\begin{cases}
\nabla\cdot\sigma[{\bf u}^{*},p^{*}]=0,\quad\nabla\cdot {\bf u}^{*}=0,\quad&\hbox{in}\ \Omega^{0},\\
{\bf u}^{*}=\sum_{\alpha=1}^{6}C_{*}^{\alpha}{\boldsymbol\psi}_{\alpha},&\hbox{on}\ \partial D_{1}^{0}\cup\partial D_{2}^{0},\\
\int_{\partial{D}_{1}^{0}}{\boldsymbol\psi}_\alpha\cdot\sigma[{\bf u}^*,p^{*}]\nu+\int_{\partial{D}_{2}^{0}}{\boldsymbol\psi}_\alpha\cdot\sigma[{\bf u}^*,p^{*}]\nu=0,&\alpha=1,\dots,6,\\
{\bf u}^{*}={\boldsymbol\varphi},&\hbox{on}\ \partial{D},
\end{cases}
\end{align}
where the constants $C_{*}^{\alpha}$, $\alpha=1,\dots,6$, are uniquely determined by the solution $({\bf u}^*,p^*)$. We obtain a lower bound of $|\nabla{{\bf u}}|$ on the segment $\overline{P_{1}P_{2}}$.

\begin{theorem}\label{mainthm2}(Lower Bounds in 3D)
Assume that $D_1,D_2,D,\Omega$, and $\varepsilon$ are defined as above, and ${\boldsymbol\varphi}\in C^{2,\alpha}(\partial D;\mathbb R^3)$ for some $0<\alpha<1$. Let ${\bf u}\in H^1(D;\mathbb R^3)\cap C^1(\bar{\Omega};\mathbb R^3)$ and $p\in L^2(D)\cap C^0(\bar{\Omega})$ be a solution to \eqref{Stokessys} and \eqref{compatibility}. Then if $\tilde b_{1}^{*\alpha_0}[{\boldsymbol\varphi}]\neq0$ for some $\alpha_0<3$, then there exists  a sufficiently small $\varepsilon_0>0$, such that for $0<\varepsilon<\varepsilon_0$, 
\begin{equation*}
|\nabla{{\bf u}}(0',x_{3})|\geq \frac{\Big|\tilde b_1^{*\alpha_0}[{\boldsymbol\varphi}]\Big|}{C\varepsilon|\ln\varepsilon|}\|{\boldsymbol\varphi}\|_{C^{2,\alpha}(\partial D;\mathbb R^3)},\quad|x_{3}|\leq\varepsilon.
\end{equation*}
\end{theorem}

As a consequence, we show the point-wise upper bounds and lower bounds for  the Cauchy stress tensor $\sigma[{\bf u},p]$, see  Theorem \ref{mainthmsigma} in Section \ref{sec5}.

\subsection{Upper Bounds of $|\nabla{\bf u}|$ and $|p|$ in dimensions $d\geq4$}\label{subsection-upphigher}

\begin{theorem}\label{mainthmD4}(Upper Bounds in Higher Dimensions)
Assume that $D_1,D_2,D,\Omega$, and $\varepsilon$ are defined as in Subsection \ref{subsec1.1}, and ${\boldsymbol\varphi}\in C^{2,\alpha}(\partial D;\mathbb R^d)$ for some $0<\alpha<1$, $d\geq4$. Let ${\bf u}\in H^1(D;\mathbb R^d)\cap C^2(\bar{\Omega};\mathbb R^d)$ and $p\in L^2(D)\cap C^1(\bar{\Omega})$ be the solution to \eqref{sto} and \eqref{compatibility}. Then for sufficiently small $0<\varepsilon<1$, 
\begin{align*}
&|\nabla{{\bf u}}(x)|\leq C\frac{\varepsilon+|x'|}{(\varepsilon+|x'|^{2})^2}\|{\boldsymbol\varphi}\|_{C^{2,\alpha}(\partial D;\mathbb R^d)},\quad x\in \Omega_{R},\\
&\inf_{c\in\mathbb{R}}\|p+c\|_{C^0(\bar{\Omega}_{R})}\leq \frac{C}{\varepsilon^2}\|{\boldsymbol\varphi}\|_{C^{2,\alpha}(\partial D;\mathbb R^d)},
\end{align*}
\begin{align*}
|\nabla^2{{\bf u}}(x)|+|\nabla p(x)|\leq\,C\frac{\varepsilon+|x'|}{(\varepsilon+|x'|^{2})^{3}}\|{\boldsymbol\varphi}\|_{C^{2,\alpha}(\partial D;\mathbb R^d)},\quad~x\in\Omega_{R},
\end{align*}
and
$$\|{\bf u}\|_{C^{2}(\Omega\setminus\Omega_{R})}+\|p\|_{C^{1}(\Omega\setminus\Omega_{R})}\leq\,C\|{\boldsymbol\varphi}\|_{C^{2,\alpha}(\partial D;\mathbb R^d)},$$
where $C$ is a universal constant independent of $\varepsilon$. 
\end{theorem}

\begin{remark}
Here we would like to point out the presence of the pressure term $p$ makes the construction of auxiliary functions constructed in Section \ref{sec2outline3D} below more involved. This is an essential difficulty compared to the case in Lam\'{e} system. Moreover, the maximum of the upper bound of $|\nabla{\bf u}|$  in Theorem \ref{mainthmD4} is of order $\varepsilon^{-3/2}$, which is larger than $\varepsilon^{-1}$ obtained in \cite[Theorem 1.2]{BLL2}. 
\end{remark}

\begin{remark}
The estimates for $C_{1}^{\alpha}-C_{2}^{\alpha}$, \eqref{estC1121} below,  play an important role in proving the lower bounds of the gradient; see the proof of Theorem \ref{mainthm2} in Section \ref{sec5} for the details. However, how to prove a similar smallness estimate as in \eqref{estC1121} is quite challenging in dimensions at least four. We will consider this problem in the near future.  
\end{remark}

The rest of this paper is organized as follows: In Section \ref{sec2outline3D}, we present our main idea to establish the gradient estimates for the solution of Stokes system \eqref{Stokessys} in dimension three. We explain the main difficulties needed to overcome in the course of adapting the iteration approach and list the ingredients to prove Theorem \ref{mainthm}. A sketched proof is given in the end of Section \ref{sec2outline3D}.  In Section \ref{sec_estimate3D}, for each auxiliary function ${\bf v}_{i}^{\alpha}$, we calculate the concrete estimates required for the iteration process, then prove the estimates for ${\bf u}_{i}^{\alpha}$ listed in Section \ref{sec2outline3D}.  Following the method outlined in Section \ref{sec2outline3D},  the construction of the auxiliary functions and the main ingredients of the proof are presented there. We prove Theorem \ref{mainthm2} for the lower bounds of $|\nabla{\bf u}|$ in Section \ref{sec5}, where we also show the point-wise upper bounds and lower bounds for  the Cauchy stress tensor $\sigma[{\bf u},p]$ in Theorem \ref{mainthmsigma}. Finally, the proof of Theorem \ref{mainthmD4} is given in Section \ref{prfthmD4}.

\section{Main Ingredients and Main Idea  of the Proof}\label{sec2outline3D}
In this section, we will present the main ingredients and show the main idea of the proof of Theorem \ref{mainthm}, the upper bounds of $|\nabla{\bf u}|$ and $|p|$ in three dimensions. 

\subsection{Construction of Auxiliary Functions and Main Estimates}\label{subsec-constuction}

Recall that a basis of $\Psi$ in dimension three is
$${\boldsymbol\psi}_{1}=\begin{pmatrix}
1 \\
0\\
0
\end{pmatrix},
{\boldsymbol\psi}_{2}=\begin{pmatrix}
0\\
1\\
0
\end{pmatrix},
{\boldsymbol\psi}_{3}=\begin{pmatrix}
0\\
0\\
1
\end{pmatrix},
{\boldsymbol\psi}_{4}=\begin{pmatrix}
x_{2}\\
-x_{1}\\
0
\end{pmatrix},
{\boldsymbol\psi}_{5}=\begin{pmatrix}
x_{3}\\
0\\
-x_{1}
\end{pmatrix},
{\boldsymbol\psi}_{6}=\begin{pmatrix}
0\\
x_{3}\\
-x_{2}
\end{pmatrix}.
$$
First, it is important to note that problem \eqref{Stokessys} has free boundary value feature. Although the third line in \eqref{Stokessys}, $e({\bf u}) =0$ in $D_i$, implies ${\bf u}$ is linear combination of ${\boldsymbol\psi}_{\alpha}$,
\begin{equation}\label{introC}
{\bf u}=\sum_{\alpha=1}^{6}C_{i}^{\alpha}{\boldsymbol\psi}_{\alpha},\quad\mbox{in}~D_i,\quad i=1,2,
\end{equation}
these twelve $C_{i}^{\alpha}$ are free constants, to be dealt with. By the continuity of the transmission condition on $\partial{D}_{i}$, we decompose the solution of \eqref{Stokessys} in $\Omega$ as follows:
\begin{align}\label{udecom}
{\bf u}(x)&=\sum_{i=1}^{2}\sum_{\alpha=1}^{6}C_i^{\alpha}{\bf u}_{i}^{\alpha}(x)+{\bf u}_{0}(x),\quad~ \mbox{and}~
p(x)=\sum_{i=1}^{2}\sum_{\alpha=1}^{6}C_i^{\alpha}p_{i}^{\alpha}(x)+p_{0}(x),
\end{align}
where ${\bf u}_{i}^{\alpha},{\bf u}_{0}\in{C}^{2}(\overline{\Omega};\mathbb R^3),~p_{i}^{\alpha}, p_0\in{C}^{1}(\overline{\Omega})$, respectively, satisfy
\begin{equation}\label{equ_v1}
\begin{cases}
\nabla\cdot\sigma[{\bf u}_{i}^\alpha,p_{i}^{\alpha}]=0,\quad\nabla\cdot {\bf u}_{i}^{\alpha}=0,&\mathrm{in}~\Omega,\\
{\bf u}_{i}^{\alpha}={\boldsymbol\psi}_{\alpha},&\mathrm{on}~\partial{D}_{i},\\
{\bf u}_{i}^{\alpha}=0,&\mathrm{on}~\partial{D_{j}}\cup\partial{D},~j\neq i,
\end{cases}i=1,2,
\end{equation}
and
\begin{equation}\label{equ_v3}
\begin{cases}
\nabla\cdot\sigma[{\bf u}_{0},p_0]=0,\quad\nabla\cdot {\bf u}_{0}=0,&\mathrm{in}~\Omega,\\
{\bf u}_{0}=0,&\mathrm{on}~\partial{D}_{1}\cup\partial{D_{2}},\\
{\bf u}_{0}={\boldsymbol\varphi},&\mathrm{on}~\partial{D}.
\end{cases}
\end{equation}
Note that $|\nabla{\bf u}_0|$ and $|p_0|$ are bounded by a constant independent of $\varepsilon$; see Proposition \ref{prop1.7} below. Therefore, in order to establish the estimates of $|\nabla{\bf u}|$ and $|p|$, it suffices to prove the gradient estimates of these ${\bf u}_{i}^{\alpha}$ and determine  free constants $C_{i}^{\alpha}$ in \eqref{udecom}. This is quite involved since we need to overcome the difficulty caused by the presence of $p$ and the divergence free condition. To express our idea clearly, we assume for simplicity that $h_{1}$ and $h_{2}$ are quadratic and symmetric with respect to the plane $\{x_{3}=0\}$, say, $h_1(x')=h_2(x')=\frac{1}{2}|x'|^2$ for $|x'|\leq 2R$, (see discussions after the proof of Proposition \ref{propu113D} in Subsection \ref{subsec3.1}). It is well-known that the Stokes system is elliptic in the sense of Douglis-Nirenberg, the regularity and estimates can be obtained from \cite{ADN1964,Solonni1966} provided the global energy is bounded. So in the sequel we will only deal with the estimates in the narrow region $\Omega_{R}$, except the boundedness of the global energy in $\Omega$. 

Next we introduce the Keller-type function $k(x)\in C^{3}(\mathbb{R}^{3})$,
satisfying $k(x)=\frac{1}{2}$ on $\partial D_{1}$,  $k(x)=-\frac{1}{2}$ on $\partial D_{2}$, $k(x)=0$ on $\partial D$, especially,
\begin{equation}\label{def_kx}
k(x)=\frac{x_{3}}{\delta(x')},\quad\hbox{in}\ \Omega_{2R},
\end{equation}
and $\|k(x)\|_{C^{3}(\mathbb{R}^{3}\backslash\Omega_{R})}\leq C$, where 
\begin{align}\label{55}
\delta(x'):=\epsilon+h_{1}(x')+h_{2}(x'),\quad\mbox{for}\, |x'|\leq 2R.
\end{align}
Clearly,
\begin{align}\label{partial_kx}
\partial_{x_j}k(x)=-\frac{2x_{j}}{\delta(x')}k(x),\quad\,j=1,2,\quad
\partial_{x_3}k(x)
=\frac{1}{\delta(x')},\quad\hbox{in}\ \Omega_{2R}.
\end{align}

We begin by estimating ~$\nabla{\bf u}_{1}^{\alpha}$, $\alpha=1,2$. We use the Keller-type function \eqref{def_kx} to construct  an auxiliary function ${\bf v}_{1}^{\alpha}\in C^{2}(\Omega;\mathbb R^3)$, such that ${\bf v}_{1}^{\alpha}={\bf u}_{1}^{\alpha}={\boldsymbol\psi}_{\alpha}$ on $\partial{D}_{1}$ and ${\bf v}_{1}^{\alpha}={\bf u}_{1}^{\alpha}=0$ on $\partial{D}_{2}\cup\partial{D}$, and specifically, 
\begin{align}\label{v1alpha}
{\bf v}_{1}^{\alpha}=\boldsymbol\psi_{\alpha}\Big(k(x)+\frac{1}{2}\Big)+\boldsymbol\psi_{3}x_{\alpha}\Big(k^2(x)-\frac{1}{4}\Big),\quad\hbox{in}\ \Omega_{2R},
\end{align}
and 
\begin{equation*}%\label{bddv113D}
\|{\bf v}_{1}^{\alpha}\|_{C^{2}(\Omega\setminus\Omega_{R})}\leq\,C.
\end{equation*}
Clearly,  by a direct calculation, it is easy to check that such ${\bf v}_{1}^{\alpha}$ satisfy 
\begin{equation*}%\label{divfree1}
\nabla\cdot{\bf v}_{1}^{\alpha}=0,\quad\alpha=1,2,\quad\mbox{in}~\Omega_{2R}.
\end{equation*} 
This is one of the key points of our constructions. Compared with \cite{BLL2}, here we use the additional terms $\boldsymbol\psi_{3}x_{\alpha}(k^2(x)-\frac{1}{4})$ to modify $\boldsymbol\psi_{\alpha}(k(x)+\frac{1}{2})$ so that the modified functions ${\bf v}_{1}^{\alpha}$ become divergence free. This turns out to be a new difficulty here and in the  subsequent constructions. However, this is only the first step. By a calculation, the first derivatives of ${\bf v}_{i}^{\alpha}$ show the singularity of order $\frac{1}{\delta(x')}$, that is,
\begin{align}\label{v11-1d}
\frac{1}{C\delta(x')}\leq|\nabla{\bf v}_{i}^{\alpha}|(x)\leq\,\frac{C}{\delta(x')}.
\end{align}
In the process of employing the energy iteration approach developed in \cite{LLBY,BLL,BLL2} to prove \eqref{v11-1d} can capture all the singular terms of $|\nabla{\bf u}_{i}^{\alpha}|$, we have to find an appropriate $\bar{p}_1^{\alpha}$ to make $|\mu\Delta{\bf v}_{1}^{\alpha}-\nabla\bar{p}_1^{\alpha}|$ as small as possible in $\Omega_{2R}$. This is another crucial issue. Indeed, their differences
\begin{align*}%\label{w}
{\bf w}_{1}^{\alpha}:={\bf u}_{1}^{\alpha}-{\bf v}_{1}^{\alpha},\quad\mbox{and}~ q_{1}^{\alpha}:=p_{1}^{\alpha}-\bar{p}_{1}^{\alpha},
\end{align*}
verify the following general boundary value problem
\begin{align}\label{w1}
\begin{cases}
-\mu\Delta{\bf w}+\nabla q={\bf f}:=\mu\Delta{\bf v}-\nabla\bar{p},\quad&\mathrm{in}\,\Omega,\\
\nabla\cdot {\bf w}=0\quad&\mathrm{in}\,\Omega_{2R},\\
\nabla\cdot {\bf w}=-\nabla\cdot {\bf v}\quad&\mathrm{in}\,\Omega\setminus\Omega_R,\\
{\bf w}=0,\quad&\mathrm{on}\,\partial\Omega.
\end{cases}
\end{align}
By an observation, we choose $\bar{p}_1^{\alpha}\in C^{1}(\Omega)$ such that
\begin{equation}\label{defp113D}
\bar{p}_1^{\alpha}=\frac{2\mu x_{\alpha}}{\delta(x')}k(x),\quad\mbox{in}~\Omega_{2R},
\end{equation}
and  $\|\bar{p}_1^{\alpha}\|_{C^{1}(\Omega\setminus\Omega_{R})}\leq C$. It turns out that $|\mu\Delta{\bf v}_{1}^{\alpha}-\nabla\bar{p}_1^{\alpha}|$ is exactly smaller than $|\mu\Delta{\bf v}_{1}^{\alpha}|$ itself. This choice of $\bar{p}_1^{\alpha}$ enables us to adapt the iteration approach in \cite{BLL2} to work for the Stokes flow, and to prove  that such $\nabla {\bf v}_{1}^{\alpha}$ and $\bar{p}_1^{\alpha}$ are the main singular terms of $\nabla {\bf u}_{1}^{\alpha}$ and $p_1^{\alpha}$. This is an essential difference with that in \cite{BLL2}. We would like to remark that here the choice of $\bar{p}_1^{\alpha}$ were inspired by the recent paper \cite{AKKY}, where the authors also established the estimates of $p$ in dimension two for circular inclusions $D_{1}$ and $D_{2}$. 

It is easy to calculate that
$$|\bar{p}_1^{\alpha}|\leq\frac{C|x'|}{\delta(x')},\quad\mbox{and}~\partial_{x_3}({\bf v}_{1}^{\alpha})^{(\alpha)}=\frac{1}{\delta(x')},\quad\mbox{in}~\Omega_{2R},$$
while, the other terms of the first order derivatives of ${\bf v}_{1}^{\alpha}$ can be controlled by $C(1+\frac{|x'|}{\delta(x')})$. By appropriate energy estimates of $|{\bf f}|$ and ${\bf w}$, such as  Lemma \ref{lem3.0} and Lemma \ref{lem3.1}, the former for the boundedness of the global energy of ${\bf w}$ and the latter for the estimates of its local energy in $\Omega_{\delta}(z')$ obtained by using the iteration technique, we can prove that the differences $|\nabla({\bf u}_{i}^{\alpha}-{\bf v}_{i}^{\alpha})|$, $\alpha=1,2$, are of order $O(1)$, see Proposition \ref{propu113D} below.

Similarly, for $i=2$, we only need to replace $k(x)=\frac{x_{3}}{\delta(x')}$ by $\tilde{k}(x)=\frac{-x_{3}}{\delta(x')}$ in $\Omega_{2R}$, which satisfies $\tilde{k}(x)+\frac{1}{2}=1$ on $\partial{D}_{2}$ and $\tilde{k}(x)+\frac{1}{2}=0$ on $\partial{D}_{1}$, then use $\tilde{k}$ to construct the corresponding ${\bf v}_{2}^{\alpha}$ and $\bar{p}_2^{\alpha}$. So, in the sequel, the assertions hold for both $i=1$ and $i=2$, but we only prove the case for $i=1$, and omit the case for $i=2$. 

Let us denote
$$\Omega_{\delta}(x'):=\left\{(y',y_{d})\big| -\frac{\varepsilon}{2}-h_{2}(y')<y_{d}
<\frac{\varepsilon}{2}+h_{1}(y'),\,|y'-x'|<\delta \right\},$$
for $|x'|\leq\,R$, where $\delta:=\delta(x')$ is defined by \eqref{55}. Define 
\begin{equation}\label{defqialpha} (q_{i}^\alpha)_{R}:=\frac{1}{|\Omega_{R}\setminus\Omega_{R/2}|}\int_{\Omega_{R}\setminus\Omega_{R/2}}q_{i}^\alpha(x)\mathrm{d}x,
\end{equation}
which is a constant independent of $\varepsilon$. The following estimates hold:

\begin{prop}\label{propu113D}
Let ${\bf u}_{i}^{\alpha}\in{C}^{2}(\Omega;\mathbb R^3),~p_{i}^{\alpha}\in{C}^{1}(\Omega)$ be the solution to \eqref{equ_v1}, $i,\alpha=1,2$. Then there holds
\begin{equation*}
\|\nabla({\bf u}_{i}^{\alpha}-{\bf v}_{i}^{\alpha})\|_{L^{\infty}(\Omega_{\delta/2}(x'))}\leq C,\quad \,x\in\Omega_{R},
\end{equation*}
and 
\begin{equation*}
\|\nabla^2({\bf u}_{i}^{\alpha}-{\bf v}_{i}^{\alpha})\|_{L^{\infty}(\Omega_{\delta/2}(x'))}+\|\nabla q_i^\alpha\|_{L^{\infty}(\Omega_{\delta/2}(x'))}\leq\frac{C}{\delta(x')},\quad \,x\in\Omega_{R}.
\end{equation*}
Consequently, in view of \eqref{v11-1d} and \eqref{defp113D},
\begin{align*}
\frac{1}{C\delta(x')}\leq|\nabla {\bf u}_{i}^{\alpha}(x)|\leq \frac{C}{\delta(x')},\quad|\nabla^2 {\bf u}_{i}^{\alpha}(x)|\leq C\left(\frac{1}{\delta(x')}+\frac{|x'|}{\delta^2(x')}\right),\quad x\in\Omega_{R},
\end{align*}
and 
\begin{align*}
|p_{i}^{\alpha}(x)-(q_{i}^\alpha)_{R}|\leq\frac{C}{\varepsilon},\quad|\nabla p_{i}^{\alpha}(x)|\leq C\left(\frac{1}{\delta(x')}+\frac{|x'|}{\delta^2(x')}\right),\quad x\in\Omega_{R}.
\end{align*}
\end{prop}

For $\alpha=3$, we seek ${\bf v}_{1}^{3}\in C^{2}(\Omega;\mathbb R^3)$ satisfying, in $\Omega_{2R}$,
\begin{align}\label{v133D}
{\bf v}_{1}^{3}=\boldsymbol\psi_{3}\Big(k(x)+\frac{1}{2}\Big)+\frac{3}{\delta(x')}\left(\sum_{\alpha=1}^{2}x_{\alpha}\boldsymbol\psi_{\alpha}+2x_{3}\boldsymbol\psi_{3}\Big(\frac{|x'|^{2}}{\delta(x')}-\frac{1}{3}\Big)\right)\Big(k^2(x)-\frac{1}{4}\Big),
\end{align}
and $\|{\bf v}_{1}^{3}\|_{C^{2}(\Omega\setminus\Omega_{R})}\leq\,C$. One can check that $\nabla\cdot{\bf v}_{1}^{3}=0$ in $\Omega_{2R}$, and
\begin{equation}\label{v13upper1}
|\nabla{\bf v}_{1}^{3}|\leq\,C\left(\frac{1}{\delta(x')}+\frac{|x'|}{\delta^{2}(x')}\right),\quad\mbox{in}~\Omega_{2R},
\end{equation}
and
\begin{equation}\label{v13lower}
|\nabla{\bf v}_{1}^{3}(0',x_{3})|\geq\,\frac{1}{C}\frac{1}{\delta(x')},\quad\mbox{if}~|x_{3}|\leq\frac{\varepsilon}{2}.
\end{equation}
Choose a $\bar{p}_1^3\in C^{1}(\Omega)$ such that
\begin{equation}\label{p133D}
\bar{p}_1^3=-\frac{3}{2}\frac{\mu}{\delta^2(x')}+\frac{18\mu}{\delta(x')}\left(\frac{|x'|^{2}}{\delta(x')}-\frac{1}{3}\right)k^{2}(x),\quad\mbox{in}~\Omega_{2R},
\end{equation}
and  $\|\bar{p}_1^3\|_{C^{1}(\Omega\setminus\Omega_{R})}\leq C$, to make $|\mu\Delta{\bf v}_{1}^{\alpha}-\nabla\bar{p}_1^{\alpha}|$ be as small as possible. 

\begin{prop}\label{propu133D}
Let ${\bf u}_{i}^{3}\in{C}^{2}(\Omega;\mathbb R^3),~p_{i}^{3}\in{C}^{1}(\Omega)$ be the solution to \eqref{equ_v1}, then we have
\begin{equation*}
\|\nabla({\bf u}_{i}^{3}-{\bf v}_{i}^{3})\|_{L^{\infty}(\Omega_{\delta/2}(x'))}\leq\frac{C}{\sqrt{\delta(x')}},\quad x\in\Omega_{R},
\end{equation*}
and
\begin{equation*}
\|\nabla^2({\bf u}_{i}^{3}-{\bf v}_{i}^{3})\|_{L^{\infty}(\Omega_{\delta/2}(x'))}+\|\nabla q_i^3\|_{L^{\infty}(\Omega_{\delta/2}(x'))}\leq \frac{C}{\delta^{3/2}(x')},\quad x\in\Omega_{R}.
\end{equation*}
Consequently, from \eqref{v13upper1}--\eqref{p133D}, it follows that for $x\in\Omega_R$,
\begin{align*}
\frac{1}{C\delta(x')}\leq|\nabla {\bf u}_{i}^{3}(x)|\leq C\left(\frac{1}{\delta(x')}+\frac{|x'|}{\delta^2(x')}\right),~|\nabla^2 {\bf u}_{i}^{3}(x)|\leq C\left(\frac{1}{\delta^{2}(x')}+\frac{|x'|}{\delta^3(x')}\right),
\end{align*}
and
\begin{align*}
|p_{i}^{3}(x)-(q_{i}^3)_{\Omega_R}|\leq\frac{C}{\varepsilon^2},~|\nabla p_{i}^{3}(x)|\leq C\left(\frac{1}{\delta^{2}(x')}+\frac{|x'|}{\delta^3(x')}\right).
\end{align*}
\end{prop}

Following the above idea, we can establish the estimates of ${\bf u}_{1}^{\alpha}$,  $\alpha=4,5,6$. Here we would like to remark that the construction of ${\bf v}_{1}^{4}$ is easy, while for the cases $\alpha=5,6$, it is even more complicated than that of ${\bf v}_{1}^{3}$, due to the appearance of $x_{3}$ in ${\boldsymbol\psi}^{5}$ and ${\boldsymbol\psi}^{6}$. For example, we here present ${\bf v}_{1}^{5}\in C^{2}(\Omega;\mathbb R^3)$, such that, in $\Omega_{2R}$,
\begin{align}\label{v15}
{\bf v}_{1}^{5}={\boldsymbol\psi}_{5}\left(k(x)+\frac{1}{2}\right)
+
\frac{3}{5}\Big(k^2(x)-\frac{1}{4}\Big)(F_1(x),G(x),H_1(x))^{\mathrm T},
\end{align}
and $\|{\bf v}_{1}^{5}\|_{C^{2}(\Omega\setminus\Omega_{R})}\leq\,C$, where 
\begin{equation}\label{defFGH}
\begin{split}
F_i(x)&=1-\frac{4x_{i}^2}{\delta(x')}-\frac{25}{3}x_{3}k(x),\quad i=1,2,\quad G(x)=-\frac{4x_{1}x_2}{\delta(x')},\\
H_i(x)&=8x_i\left(\frac{2}{3}-\frac{|x'|^{2}}{\delta(x')}\right)k(x)-10x_ix_3k^2(x),\quad i=1,2.
\end{split}
\end{equation}
We choose $\bar{p}_1^5\in C^{1}(\Omega)$ such that
\begin{equation*}
\bar{p}_1^5=\frac{6\mu}{5}\frac{ x_1}{\delta^2(x')}+\frac{72\mu}{5}\frac{ x_1}{\delta(x')}\left(\frac{2}{3}-\frac{|x'|^{2}}{\delta(x')}\right)k^{2}(x),\quad\mbox{in}~\Omega_{2R},
\end{equation*}
and $
\|\bar{p}_1^5\|_{C^{1}(\Omega\setminus\Omega_{R})}\leq C$. By a straightforward calculation, it is not difficult to check the divergence free condition. However, more new techniques are needed to prove the boundedness of the global energy of ${\bf u}_{1}^{5}-{\bf v}_{1}^{5}$.  More details can be found in Section \ref{sec_estimate3D}. The estimates of $|\nabla{\bf u}_{1}^{\alpha}|$,  $\alpha=4,5,6$, are obtained as follows: 

\begin{prop}\label{propu4563D}
Let ${\bf u}_{i}^{\alpha}\in{C}^{2}(\Omega;\mathbb R^3),~p_{i}^{\alpha}\in{C}^{1}(\Omega)$ be the solution to \eqref{equ_v1}, $\alpha=4,5,6$, then we have
\begin{equation*}
\|\nabla({\bf u}_{i}^{\alpha}-{\bf v}_{i}^{\alpha})\|_{L^{\infty}(\Omega_{\delta/2}(x'))}\leq
C,\quad x\in\Omega_{R},
\end{equation*}
\begin{equation*}
\|\nabla^2({\bf u}_{i}^{\alpha}-{\bf v}_{i}^{\alpha})\|_{L^{\infty}(\Omega_{\delta/2}(x'))}+\|\nabla q_i^\alpha\|_{L^{\infty}(\Omega_{\delta/2}(x'))}\leq 
\begin{cases}
\frac{C}{\sqrt{\delta(x')}},&\alpha=4,\\
\frac{C}{\delta(x')},&\alpha=5,6,
\end{cases}\quad x\in\Omega_{R}.
\end{equation*}
Consequently, for $x\in\Omega_{R}$,
\begin{align*}
|\nabla {\bf u}_{i}^{\alpha}(x)|\leq 
\begin{cases}
C\left(\frac{|x'|}{\delta(x')}+1\right),&\alpha=4,\\
\frac{C}{\delta(x')},&\alpha=5,6,
\end{cases}~
|\nabla^2 {\bf u}_{i}^{\alpha}(x)|\leq
\begin{cases}
\frac{C}{\delta(x')},&\alpha=4,\\
\frac{C}{\delta^{2}(x')},&\alpha=5,6,
\end{cases}
\end{align*}
and
\begin{align*}
|p_{i}^{\alpha}(x)-(q_{i}^\alpha)_{R}|\leq
\begin{cases}
\frac{C}{\sqrt{\varepsilon}},&\alpha=4,\\
\frac{C}{\varepsilon^{3/2}},&\alpha=5,6,
\end{cases} ~
|\nabla p_{i}^{\alpha}(x)|\leq
\begin{cases}
\frac{C}{\sqrt{\delta(x')}},&\alpha=4,\\
\frac{C}{\delta^{2}(x')},&\alpha=5,6.
\end{cases} 
\end{align*}
\end{prop}

As explained in \cite[Proposition 2.5]{LX2},  the stress tensor will not blow up when the boundary data
takes the same value on $\partial D_1$ and $\partial D_2$. We also present it here for reader's convenience.

\begin{prop}\label{prop1.7}
Let ${\bf u}_{i}^{\alpha},{\bf u}_0\in{C}^{2}(\Omega;\mathbb R^3),~p_{i}^{\alpha},p_0\in{C}^{1}(\Omega)$ be the solution to \eqref{equ_v1} and \eqref{equ_v3}, respectively, $\alpha=1,\dots,6$. Then we have
\begin{equation*}
\|\nabla^{k_1}{\bf u}_0\|_{L^{\infty}(\Omega)}+\|\nabla^{k_1}({\bf u}_{1}^{\alpha}+{\bf u}_{2}^{\alpha})\|_{L^{\infty}(\Omega)}\leq C,\quad k_1=1,2;
\end{equation*}
and
\begin{equation*}
\|\nabla^{k_2} p_0\|_{L^{\infty}(\Omega)}+\|\nabla^{k_2}(p_{1}^{\alpha}+p_{2}^{\alpha})\|_{L^{\infty}(\Omega)}\leq C,\quad k_2=0,1.
\end{equation*}
\end{prop}

Now let us solve out the twelve free constants $C_{i}^{\alpha}$ introduced in \eqref{introC}. Indeed, the trace theorem can ensure the boundedness of $C_i^{\alpha}$. The main challenge is how to fix the differences of $|C_1^{\alpha}-C_2^{\alpha}|$. For this, it follows from the forth line of \eqref{Stokessys} and the decomposition \eqref{udecom} that  
\begin{equation}\label{equ-decompositon}
\sum_{i=1}^2\sum\limits_{\alpha=1}^{6} C_{i}^{\alpha}
\int_{\partial D_j}{\boldsymbol\psi}_\beta\cdot\sigma[{\bf u}_{i}^\alpha,p_{i}^{\alpha}]\nu
+\int_{\partial D_j}{\boldsymbol\psi}_\beta\cdot\sigma[{\bf u}_{0},p_{0}]\nu=0,~~\beta= 1,\dots,6,
\end{equation}
where $j=1,2$. With the above estimates of $\nabla{\bf u}_{i}^{\alpha}$ and $p_{i}^{\alpha}$ at hand, we can prove the estimates of $|C_1^{\alpha}-C_2^{\alpha}|$ as showed in the lemma below; the proof of which will be given in Subsection \ref{subsecCi}.  

\begin{prop}\label{lemCialpha3D}
Let $C_{i}^{\alpha}$ be defined in \eqref{udecom}. Then
\begin{equation}\label{bddC}
|C_i^{\alpha}|\leq\,C,\quad\,i=1,2,~\alpha=1,2,\dots,6,
\end{equation}
and
\begin{equation}\label{estC1121}
|C_1^{\alpha}-C_2^{\alpha}|\leq \frac{C}{|\ln\varepsilon|}, ~\alpha=1,2,5,6,\quad |C_1^3-C_2^3|\leq C\varepsilon,~\mbox{and}~|C_1^4-C_2^4|\leq C.
\end{equation}
\end{prop}

\subsection{The Completion of the Proof of Theorem \ref{mainthm}}\label{subsec-mainthm}
Now, we are in a position to complete the proof of Theorem \ref{mainthm}.

\begin{proof}[Proof of Theorem \ref{mainthm}]
We  rewrite \eqref{udecom} as
\begin{equation*}%\label{nablau_dec}
\nabla{{\bf u}}=\sum_{\alpha=1}^{6}\left(C_{1}^{\alpha}-C_{2}^{\alpha}\right)\nabla{\bf u}_{1}^{\alpha}
+\nabla {\bf u}_{b}\quad\mbox{and }~
p=\sum_{\alpha=1}^{6}\left(C_{1}^{\alpha}-C_{2}^{\alpha}\right)p_{1}^{\alpha}+p_{b},\quad\mbox{in}~\Omega,
\end{equation*}
where
\begin{equation*}%\label{def u_b}
{\bf u}_{b}:=\sum_{\alpha=1}^{6}C_{2}^{\alpha}({\bf u}_{1}^{\alpha}+{\bf u}_{2}^{\alpha})+{\bf u}_{0},\quad p_{b}:=\sum_{\alpha=1}^{6}C_{2}^{\alpha}(p_{1}^{\alpha}+p_{2}^{\alpha})+p_{0},
\end{equation*}
whose gradients are obviously bounded, due to Propositions \ref{prop1.7}  and \ref{lemCialpha3D}. 

(i) By virtue of the above estimates of $|\nabla{\bf u}_{1}^{\alpha}(x)|$, established in Propositions \ref{propu113D}--\ref{lemCialpha3D}, we obtain
\begin{align}\label{upper-u3D}
|\nabla{\bf u}(x)|&\leq\sum_{\alpha=1}^{6}\left|(C_{1}^{\alpha}-C_{2}^{\alpha}\right)\nabla{\bf u}_{1}^{\alpha}(x)|+C\nonumber\\
&\leq\sum_{\alpha=1,2,5,6}\left|(C_{1}^{\alpha}-C_{2}^{\alpha}\right)\nabla{\bf u}_{1}^{\alpha}(x)|+\left|(C_{1}^{3}-C_{2}^{3}\right)\nabla{\bf u}_{1}^{3}(x)|+C|\nabla{\bf u}_{1}^{4}(x)|+C\nonumber\\
&\leq\frac{C}{|\ln\varepsilon|\delta(x')}+C\varepsilon\left(\frac{1}{\delta(x')}+\frac{|x'|}{\delta^2(x')}\right)+C\left(\frac{|x'|}{\delta(x')}+1\right)\nonumber\\
&\leq \frac{C(1+|\ln\varepsilon||x'|)}{|\ln\varepsilon|(\varepsilon+|x'|^2)}.
\end{align}
Taking into account $q_{1}^{\alpha}=p_{1}^{\alpha}-\bar{p}_{1}^{\alpha}$, let us denote
$$q:=\sum_{\alpha=1}^{6}\left(C_{1}^{\alpha}-C_{2}^{\alpha}\right)q_{1}^{\alpha}+q_{b},\quad\mbox{where}~q_{b}:=\sum_{\alpha=1}^{6}C_{2}^{\alpha}(q_{1}^{\alpha}+q_{2}^{\alpha})+q_{0}.$$
Choose $(q)_{R}:=\sum_{\alpha=1}^{6}(C_{1}^{\alpha}-C_{2}^{\alpha})(q_{1}^\alpha)_{R}$, where $(q_{1}^\alpha)_{R}$ is defined in \eqref{defqialpha}. Then by means of the estimates of $p_{1}^{\alpha}(x)$ in Propositions \ref{propu113D}--\ref{lemCialpha3D}, we have 
\begin{align}\label{upper-p3D}
|p(x)-(q)_{R}|&\leq\sum_{\alpha=1}^{6}|(C_{1}^{\alpha}-C_{2}^{\alpha})(p_{1}^{\alpha}(x)-(q_{1}^\alpha)_{R})|+C\nonumber\\
&\leq\sum_{\alpha=1,2}\left|(C_{1}^{\alpha}-C_{2}^{\alpha}\right)(p_{1}^{\alpha}(x)-(q_{1}^\alpha)_{R})|\nonumber\\
&\quad+\left|(C_{1}^{3}-C_{2}^{3}\right)(p_{1}^{3}(x)-(q_{1}^3)_{R})|+C|(p_{1}^{4}(x)-(q_{1}^4)_{R})|\nonumber\\
&\quad+\sum_{\alpha=5,6}\left|(C_{1}^{\alpha}-C_{2}^{\alpha}\right)(p_{1}^{\alpha}(x)-(q_{1}^\alpha)_{R})|+C\nonumber\\
&\leq\frac{C}{\varepsilon}+\frac{C}{|\ln\varepsilon|\varepsilon^{3/2}}+C\leq \frac{C}{|\ln\varepsilon|\varepsilon^{3/2}},
\end{align}
we conclude that \eqref{thm_p3D} holds. 

(ii) With the use of Propositions \ref{propu113D}--\ref{lemCialpha3D} again, we obtain
\begin{align*}
|\nabla^2{\bf u}(x)|
&\leq\sum_{\alpha=1,2}\left|(C_{1}^{\alpha}-C_{2}^{\alpha}\right)\nabla^2{\bf u}_{1}^{\alpha}(x)|+\left|(C_{1}^{3}-C_{2}^{3}\right)\nabla^2{\bf u}_{1}^{3}(x)|+C|\nabla^2{\bf u}_{1}^{4}(x)|C\nonumber\\
&\quad+\sum_{\alpha=5,6}\left|(C_{1}^{\alpha}-C_{2}^{\alpha}\right)\nabla^2{\bf u}_{1}^{\alpha}(x)|+C\nonumber\\
&\leq\frac{C}{|\ln\varepsilon|}\left(\frac{1}{\delta(x')}+\frac{|x'|}{\delta^2(x')}\right)+C\varepsilon\left(\frac{1}{\delta^{2}(x')}+\frac{|x'|}{\delta^3(x')}\right)+\frac{C}{\delta(x')}+\frac{C}{|\ln\varepsilon|\delta^2(x')}\nonumber\\
&\leq\frac{C(1+|\ln\varepsilon||x'|)}{|\ln\varepsilon|(\varepsilon+|x'|^2)^{2}},
\end{align*}
and 
\begin{align*}
|\nabla p(x)|
&\leq\sum_{\alpha=1,2}\left|(C_{1}^{\alpha}-C_{2}^{\alpha}\right)\nabla p_{1}^{\alpha}(x)|+\left|(C_{1}^{3}-C_{2}^{3}\right)\nabla p_{1}^{3}(x)|+C|\nabla p_{1}^{4}(x)|\nonumber\\
&\quad+\sum_{\alpha=5,6}\left|(C_{1}^{\alpha}-C_{2}^{\alpha}\right)\nabla p_{1}^{\alpha}(x)|+C\nonumber\\
&\leq\frac{C}{|\ln\varepsilon|}\left(\frac{1}{\delta(x')}+\frac{|x'|}{\delta^{2}(x')}\right)+C\varepsilon\left(\frac{1}{\delta^{2}(x')}+\frac{|x'|}{\delta^3(x')}\right)+\frac{C}{\sqrt{\delta(x')}}+\frac{C}{|\ln\varepsilon|\delta^2(x')}\nonumber\\
&\leq \frac{C(1+|\ln\varepsilon||x'|)}{|\ln\varepsilon|(\varepsilon+|x'|^2)^{2}}.
\end{align*}
The proof of Theorem \ref{mainthm} is completed.
\end{proof}

\section{Proofs for the Main Estimates}\label{sec_estimate3D}

This section is devoted to proving the main estimates listed in Section \ref{sec2outline3D}. 

\subsection{Estimates of $|\nabla{\bf u}_1^\alpha|$ and $|p_1^\alpha|$,  $\alpha=1,2$} \label{subsec3.1}

Recalling the definition of ${\bf v}_{1}^{\alpha}$ in \eqref{v1alpha} and using direct calculations, for the first order derivatives,
\begin{align}
\partial_{x_j}({\bf v}_{1}^{\alpha})^{(\alpha)}&=-\frac{2x_j}{\delta(x')}k(x),
\quad\partial_{x_3}({\bf v}_{1}^{{\alpha}})^{(\alpha)}=\frac{1}{\delta(x')},~\alpha,j=1,2;\label{v11-1}\\
 \nabla({\bf v}_{1}^{{\alpha}})^{(\beta)}&=0,\quad\quad~\alpha,\beta=1,2,~\alpha\neq\beta;\label{v11-0}\\
|\partial_{x'}({\bf v}_{1}^{\alpha})^{(3)}|&\leq C,\quad\partial_{x_3}({\bf v}_{1}^{\alpha})^{(3)}=\frac{2x_{\alpha}}{\delta(x')}k(x),~\alpha=1,2.\label{v11-3}
\end{align}
Clearly, we have
\begin{equation*}
\nabla\cdot{\bf v}_{1}^{\alpha}=0,\quad\mbox{for}~\alpha=1,2.
\end{equation*}
Furthermore, concerning the second order derivatives,
\begin{align}
|\partial_{x_1x_1}({\bf v}^{\alpha}_{1})^{(\alpha)}|,|\partial_{x_2x_2}({\bf v}^{\alpha}_{1})^{(\alpha)}|&\leq\frac{C}{\delta(x')},\quad\partial_{x_3x_3}({\bf v}^{\alpha}_{1})^{(\alpha)}=0;\label{estv1alphaxx0}\\
|\partial_{x_1x_1}({\bf v}^{\alpha}_{1})^{(3)}|,|\partial_{x_2x_2}({\bf v}^{\alpha}_{1})^{(3)}|&\leq\frac{C|x_{\alpha}|}{\delta(x')},\quad
\partial_{x_3x_3}({\bf v}^{\alpha}_{1})^{(3)}= \frac{2x_{\alpha}}{\delta^{2}(x')}.\nonumber%\label{estv1alphaxx2}
\end{align}

In order to control the right hand side of \eqref{w1}, $|{\bf f}_{1}^{\alpha}|$, we choose $\bar{p}_1^{\alpha}\in C^{1}(\Omega)$ such that
\begin{equation*}
\bar{p}_1^{\alpha}=\frac{2\mu x_{\alpha}}{\delta(x')}k(x),\quad\mbox{in}~\Omega_{2R},
\end{equation*}
and $
\|\bar{p}_1^{\alpha}\|_{C^{1}(\Omega\setminus\Omega_{R})}\leq C$.
Luckily, it is easy to observe that
\begin{equation}\label{v11-p}
\mu\partial_{x_3}({\bf v}^{\alpha}_{1})^{(3)}-\bar{p}_1^{\alpha}=0.
\end{equation}
We emphasize that this observation is crucial, because it not only does make the possible biggest term  $\partial_{x_3x_3}({\bf v}^{\alpha}_{1})^{(3)}= \frac{2x_{\alpha}}{\delta^{2}(x_{1})}$, among the Hessian matrix of ${\bf v}^{\alpha}_{1}$, not appear in $|{\bf f}_{1}^{\alpha}|$, but also does not cause the other terms in $|{\bf f}_{1}^{\alpha}|$ to become larger. To be specific, since 
\begin{equation*}%\label{p11x}
|\partial_{x'}\bar{p}_1^{\alpha}|\leq\frac{C}{\delta(x')}, 
\end{equation*}
it follows from \eqref{estv1alphaxx0}  and \eqref{v11-0} that
\begin{align*}
|({\bf f}_{1}^{\alpha})^{(\alpha)}|&=\Big|\mu\partial_{x_1x_1}({\bf v}^{\alpha}_{1})^{(\alpha)}+\mu\partial_{x_2x_2}({\bf v}^{\alpha}_{1})^{(\alpha)}-\partial_{x_\alpha}\bar{p}_1^\alpha\Big|\leq \frac{C}{\delta(x')},\quad\alpha=1,2;\\
|({\bf f}_{1}^{\alpha})^{(\beta)}|&=|-\partial_{x_\beta}\bar{p}_1^{\alpha}|\leq \frac{C}{\delta(x')},\quad\alpha,\beta=1,2,~\alpha\neq\beta;
\end{align*}
while, in view of \eqref{v11-p},
$$|({\bf f}_{1}^{\alpha})^{(3)}|=\Big|\mu\partial_{x_1x_1}({\bf v}^{\alpha}_{1})^{(3)}+\mu\partial_{x_2x_2}({\bf v}^{\alpha}_{1})^{(3)}\Big|\leq \frac{C|x'|}{\delta(x')},\quad\alpha=1,2.$$
The above estimates yield
\begin{align}\label{fLinfty}
|{\bf f}_{1}^{\alpha}|=\left|\mu\Delta{\bf v}_{1}^{\alpha}-\nabla\bar{p}_1^{\alpha}\right|\leq \frac{C}{\delta(x')}.
\end{align}

Let us denote the total energy in $\Omega$ by
\begin{align}\label{wialpha}
E_{i}^{\alpha}:=\int_{\Omega}|\nabla{\bf w}_{i}^{\alpha}|^{2}\mathrm{d}x.
\end{align}
The above properties of ${\bf v}_{1}^{\alpha}$ and $\bar{p}_1^{\alpha}$ enables us to obtain the boundedness of $E_{i}^{\alpha}$, $\alpha=1,2$, which is a very important step to employ our iteration approach.

For reader's convenience, we present the following lemma in \cite{LX2}.
\begin{lemma}\label{lemmaenergy}
Let $({\bf w},q)$ be the solution to \eqref{w1}. Then if ${\bf v}\in C^{2}(\Omega;\mathbb R^d)$ and $\bar{p}\in C^{1}(\Omega)$ satisfy
\begin{equation*}%\label{estv113D3}
\|{\bf v}\|_{C^{2}(\Omega\setminus\Omega_R)}\leq\,C,\quad\|\bar{p}\|_{C^1(\Omega\setminus\Omega_{R})}\leq C,
\end{equation*}
and 
\begin{align}\label{int-fw}
\Big| \int_{\Omega_{R}}\sum_{j=1}^{d}{\bf f}^{(j)}{\bf w}^{(j)}\mathrm{d}x\Big|\leq\,C\left(\int_{\Omega}|\nabla {\bf w}|^2\mathrm{d}x\right)^{1/2},
\end{align}
then
\begin{align*}%\label{estw11case2}
\int_{\Omega}|\nabla {\bf w}|^{2}\mathrm{d}x\leq C.
\end{align*}
\end{lemma}

\begin{lemma}\label{lem3.0}
Let $({\bf w}_i^\alpha,q_i^\alpha)$ be the solution to \eqref{w1}, for $\alpha=1,2$. Then
\begin{align}\label{w1alpha}
E_{i}^{\alpha}=\int_{\Omega}|\nabla{\bf w}_{i}^{\alpha}|^{2}\mathrm{d}x\leq\,C,\quad i=1,2.
\end{align}
\end{lemma}

\begin{proof}
Here and throughout this paper, we only prove the case of $i=1$  for instance, since the case $i=2$ is the same. From Lemma  \ref{lemmaenergy},  the proof of inequality \eqref{w1alpha} is reduced to proving condition \eqref{int-fw}, i.e.
\begin{align*}%\label{int-fw}
\Big| \int_{\Omega_{R}}\sum_{j=1}^{3}{\bf f}^{(j)}{\bf w}^{(j)}\mathrm{d}x\Big|\leq\,C\left(\int_{\Omega}|\nabla {\bf w}|^2\mathrm{d}x\right)^{1/2}.
\end{align*}
For this, first as in \cite{BLL}, by mean value theorem and Poincar\'e inequality, there exists $r_{0}\in(R,\frac{3}{2}R)$ such that
\begin{align}\label{w1Dw1}
\int_{\substack{|x'|=r_{0},\\-\epsilon/2-h_{2}(x')<x_{3}<\epsilon/2+h_{1}(x')}}|{\bf w}_1^\alpha|\mathrm{d}x_{3}
&\leq C\left(\int_{\Omega}|\nabla {\bf w}_1^\alpha|^2\mathrm{d}x\right)^{1/2}.
\end{align}

For $j=1,2$, applying integration by parts with respect to $x'$, using \eqref{estv1alphaxx0} and \eqref{v11-p},
\begin{align*}%\label{w11p11}
\left|\int_{\Omega_{r_{0}}} {\bf f}^{(j)}({\bf w}_1^\alpha)^{(j)}\mathrm{d}x\right|
&=\left|\int_{\Omega_{r_{0}}} ({\bf w}_1^\alpha)^{(j)}(\mu\Delta_{x'}({\bf v}_1^\alpha)^{(j)}-\partial_{x_j}\bar{p}_1^\alpha)\mathrm{d}x\right|\nonumber\\
&\leq\,C\int_{\Omega_{r_{0}}}|\nabla_{x'}({\bf w}_1^\alpha)^{(j)}||\nabla_{x'}({\bf v}_1^\alpha)^{(j)}|+|\partial_{x_j}({\bf w}_1^\alpha)^{(j)}||\bar{p}_1^\alpha|\mathrm{d}x\nonumber\\
&\quad+C\,\int_{\substack{|x'|=r_{0},\\-\epsilon/2-h_{2}(x')<x_{3}<\epsilon/2+h_{1}(x')}}|({\bf w}_1^\alpha)^{(j)}|\mathrm{d}x_{3}=:\mathrm{I}_1+C\,\mathrm{I}_2^{(j)};
\end{align*}
and for $j=3$,
\begin{align*}%\label{estw13vp13}
&\left|\int_{\Omega_{r_{0}}} {\bf f}^{(3)}({\bf w}_1^\alpha)^{(3)}\mathrm{d}x\right|=\left|\mu\int_{\Omega_{r_{0}}} ({\bf w}_1^\alpha)^{(3)}\Delta_{x'}({\bf v}_1^\alpha)^{(3)}\mathrm{d}x\right|\nonumber\\
\leq&\,C\int_{\Omega_{r_{0}}}|\nabla_{x'}({\bf w}_1^\alpha)^{(3)}||\nabla_{x'}({\bf v}_1^\alpha)^{(3)}|\mathrm{d}x+C\int_{\substack{|x'|=r_{0},\\-\epsilon/2-h_{2}(x')<x_{3}<\epsilon/2+h_{1}(x')}}|({\bf w}_1^\alpha)^{(3)}|\mathrm{d}x_{3}\nonumber\\
=&\,:\mathrm{I}_2+C\,\mathrm{I}_2^{(3)}.
\end{align*}
In view of  \eqref{v11-0}, \eqref{defp113D} and \eqref{v11-1}, 
\begin{equation}\label{estv113D}
|\nabla_{x'}({\bf v}_1^\alpha)^{(j)}|,|\bar{p}_1^\alpha|\leq\,\frac{C|x'|}{\delta(x')},~\,j=1,2,\quad\mbox{and}~|\nabla_{x'}({\bf v}_1^\alpha)^{(3)}|\leq\,C,\quad
\mbox{ in} ~\Omega_{2R}.
\end{equation}
Applying H\"{o}lder's inequality and using \eqref{estv113D}, 
\begin{align*}
|\mathrm{I}_1|&\leq C\left(\int_{\Omega_{r_{0}}}|\nabla_{x'}({\bf v}_1^\alpha)^{(j)}|^{2}+|\bar{p}_1^\alpha|^2\mathrm{d}x\right)^{1/2}
\left(\int_{\Omega} |\nabla_{x'}({\bf w}_1^\alpha)^{(j)}|^2\mathrm{d}x\right)^{1/2}\nonumber\\
&\leq C \left(\int_{\Omega} |\nabla {\bf w}_1^\alpha|^2\mathrm{d}x\right)^{1/2},
\end{align*}
and
\begin{align*}
|\mathrm{I}_2|&\leq C\left(\int_{\Omega_{r_{0}}}|\nabla_{x'}({\bf v}_1^\alpha)^{(3)}|^{2}\mathrm{d}x\right)^{1/2}
\left(\int_{\Omega} |\nabla_{x'}({\bf w}_1^\alpha)^{(3)}|^2\mathrm{d}x\right)^{1/2}\leq C \left(\int_{\Omega} |\nabla {\bf w}_1^\alpha|^2\mathrm{d}x\right)^{1/2}.
\end{align*}
Together with \eqref{w1Dw1}, \eqref{int-fw} is proved. Then, thanks to Lemma  \ref{lemmaenergy}, the proof of Lemma \ref{lem3.0} is finished.
\end{proof}

\begin{remark}
With \eqref{w1alpha}, as mentioned before, we can directly apply classical elliptic estimates, see \cite{ADN1959,ADN1964,Solonni1966}, to obtain $\|\nabla{\bf w}_{i}^{\alpha}\|_{L^{\infty}(\Omega\setminus\Omega_{R})}\leq\,C$. So we only focus on the estimates in $\Omega_{R}$.
\end{remark}

Integrating \eqref{fLinfty}, we have
\begin{align}\label{f1alpha}
\int_{\Omega_{s}(z')}|{\bf f}_{1}^{\alpha}|^{2}\mathrm{d}x\leq\frac{Cs^2}{\delta(z')}.
\end{align}
This and \eqref{w1alpha} are good enough to employ the adapted iteration technique to obtain the following local energy estimates.

\begin{lemma}\label{lem3.1}
Let $({\bf w}_i^\alpha,q_i^\alpha)$ be the solution to \eqref{w1}, for $\alpha=1,2$. Then
\begin{align}\label{estw11narrow3D}
\int_{\Omega_{\delta}(z')}|\nabla {\bf w}_i^\alpha|^{2}\mathrm{d}x\leq C\delta^3(z'),\quad\alpha=1,2,~i=1,2.
\end{align}
\end{lemma}

\begin{proof} The iteration scheme we use in the proof is in
spirit similar  to that used in \cite{LLBY,BLL2}. Let us set
$$E(t):=\int_{\Omega_{t}(z')}|\nabla {\bf w}_1^\alpha|^{2}\mathrm{d}x.$$
Recalling Caccioppoli-type inequality in \cite[Lemma 3.10]{LX2}: for $0<t<s\leq R$,
\begin{align}\label{iterating1}
\int_{\Omega_{t}(z')}|\nabla {\bf w}|^{2}\mathrm{d}x\leq\,&
\frac{C\delta^2(z')}{(s-t)^{2}}\int_{\Omega_{s}(z')}|\nabla {\bf w}|^2\mathrm{d}x+C\left((s-t)^{2}+\delta^{2}(z')\right)\int_{\Omega_{s}(z')}|{\bf f}|^{2}\mathrm{d}x.
\end{align}
Then combining with \eqref{f1alpha}, we have 
\begin{align}\label{iteration3D}
E(t)\leq \left(\frac{c_{0}\delta(z')}{s-t}\right)^2 E(s)+C\left((s-t)^{2}+\delta^2(z')\right)\frac{s^2}{\delta(z')},
\end{align}
where $c_{0}$ is a constant and we fix it now. Let $k_{0}=\left[\frac{1}{4c_{0}\sqrt{\delta(z')}}\right]$ and $t_{i}=\delta(z')+2c_{0}i\delta(z'), i=0,1,2,\dots,k_{0}$. 
So, applying \eqref{iteration3D} with $s=t_{i+1}$ and $t=t_{i}$, we have the following iteration formula:
\begin{align*}%\label{18}
E(t_{i})\leq \frac{1}{4}E(t_{i+1})+C(i+1)^2\delta^3(z').
\end{align*}
After $k_{0}$ iterations, and by virtue of \eqref{w1alpha}, we obtain
\begin{align*}
E(t_0)&\leq \left(\frac{1}{4}\right)^{k}E(t_{k})
+C\delta^3(z')\sum\limits_{l=0}^{k-1}\left(\frac{1}{4}\right)^{l}(l+1)^2\nonumber\\
&\leq \left(\frac{1}{4}\right)^{k}E_{1}^{\alpha}+C\delta^3(z')\sum\limits_{l=0}^{k-1}\left(\frac{1}{4}\right)^{l}(l+1)^2\leq C\delta^3(z'),
\end{align*}
for sufficiently small $\varepsilon$ and $|z'|$.  As a consequence,  \eqref{estw11narrow3D} is proved.
\end{proof}

We are now in a position to prove Proposition \ref{propu113D}.
\begin{proof}[Proof of Proposition \ref{propu113D}.]
Recalling \cite[Proposition 3.6]{LX2}, we have
\begin{align}\label{W2pstokes}
\|\nabla {\bf w}_1^\alpha\|_{L^{\infty}(\Omega_{\delta/2}(z'))}
\leq
\,C\left(\delta^{-\frac{3}{2}}\|\nabla {\bf w}_1^\alpha\|_{L^{2}(\Omega_{\delta}(z'))}+\delta\|{\bf f}_1^\alpha\|_{L^{\infty}(\Omega_{\delta}(z'))} \right)
\end{align}
and 
\begin{align}\label{Wmpstokes}
&\|\nabla^2 {\bf w}_1^\alpha\|_{L^{\infty}(\Omega_{\delta/2}(z'))}+\|\nabla q_1^\alpha\|_{L^{\infty}(\Omega_{\delta/2}(z'))}\nonumber\\
&\leq
C\Big(\delta^{-\frac{5}{2}}\|\nabla {\bf w}_1^\alpha\|_{L^{2}(\Omega_{\delta}(z'))}+\|{\bf f}_1^\alpha\|_{L^{\infty}(\Omega_{\delta}(z'))}+\delta\|\nabla{\bf f}_1^\alpha\|_{L^{\infty}(\Omega_{\delta}(z'))}\Big).
\end{align}
Substituting \eqref{estw11narrow3D} and \eqref{fLinfty} into \eqref{W2pstokes}, we have
\begin{align*}%\label{W2pstokes}
\|\nabla {\bf w}_1^\alpha\|_{L^{\infty}(\Omega_{\delta/2}(z'))}
\leq C.
\end{align*}
This together with ${\bf w}_1^\alpha={\bf u}_1^\alpha-{\bf v}_1^\alpha$ and \eqref{v11-1}--\eqref{v11-3} gives
\begin{align*}
|\nabla {\bf u}_1^\alpha(z)|&\leq|\nabla {\bf v}_1^\alpha(z)|+|\nabla {\bf w}_1^\alpha(z)|\leq\frac{C}{\delta(z')},\quad z\in\Omega_{R},\\
|\nabla {\bf u}_1^\alpha(0',z_{3})|&\geq|\nabla {\bf v}_1^\alpha(0',z_{3})|-C\geq\frac{1}{C\delta(z')},\quad |z_{3}|\leq\frac{\varepsilon}{2}.
\end{align*}
By \eqref{fLinfty} and using a direct calculation, we have for $z\in\Omega_R$,
\begin{equation*}
|{\bf f}_1^\alpha(z)|+\delta(z')|\nabla {\bf f}_1^\alpha(z)|\leq\frac{C}{\delta(z')}.
\end{equation*}
Then combining with \eqref{estw11narrow3D}, and applying \eqref{Wmpstokes}, we obtain
\begin{align}\label{nablaq}
\|\nabla^2 {\bf w}_1^\alpha\|_{L^{\infty}(\Omega_{\delta/2}(z'))}+\|\nabla q_1^\alpha\|_{L^{\infty}(\Omega_{\delta/2}(z'))}\leq\frac{C}{\delta(z')}.
\end{align}

Recalling
$q_{1}^{\alpha}=p_{1}^{\alpha}-\bar{p}_{1}^{\alpha}$, $p_{1}^{\alpha}$ and $\bar{p}_{1}^{\alpha}$ are of class $C^{1}(\overline{\Omega})$, so $p_{1}^{\alpha}(x',x_{3})|_{|x'|=R,x\in\Omega}$ and $\bar{p}_{1}^{\alpha}(x',x_{3})|_{|x'|=R,x\in\Omega}$ are bounded. Hence, by using \eqref{nablaq}, we know that $(q_{1}^{\alpha})_{R}$ is bounded, independent of $\varepsilon$. 
Applying the mean value theorem and \eqref{nablaq} yields that, for any $z\in\Omega_R$,
\begin{align}\label{est-q}
|q_1^{\alpha}(z)-(q_{1}^{\alpha})_{R}|&=\Big|\frac{1}{\Omega_R\setminus\Omega_{R/2}}\int_{\Omega_R\setminus\Omega_{R/2}}(q_1^{\alpha}(z)-q_1^{\alpha}(y))\ dy\Big|\nonumber\\
&\leq C\int_{\Omega_R}|q_1^{\alpha}(z)-q_1^{\alpha}(y)|\ dy\leq\,C\|\nabla q_1^\alpha\|_{L^{\infty}(\Omega_R)}\leq\frac{C}{\varepsilon}.
\end{align}
In virtue of \eqref{defp113D} and \eqref{est-q}, we deduce
\begin{align*}
|p_{1}^{\alpha}(z)-(q_{1}^\alpha)_{\Omega_R}|\leq |q_1^{\alpha}(z)-(q_{1}^{\alpha})_{\Omega_R}|+|\bar{p}_1^{\alpha}(z)|\leq\frac{C}{\varepsilon}+\frac{C|z'|}{\delta(z')}\leq\frac{C}{\varepsilon}.
\end{align*}
Using  ${\bf w}_1^\alpha={\bf u}_1^\alpha-{\bf v}_1^\alpha$ and \eqref{v1alpha}, the estimate of $|\nabla^2 {\bf u}_1^\alpha|$ follows.
Therefore, Proposition \ref{propu113D} is proved. 
\end{proof}

\begin{remark}
From the above calculations, our method works well for general $h_{1}(x')$ and $h_{2}(x')$. For example, if $h_{1}=h_{2}$ satisfies \eqref{h1-h2}--\eqref{h1h14}, then we take
$${\bf v}_{1}^{\alpha}=\boldsymbol\psi_{\alpha}\Big(k(x)+\frac{1}{2}\Big)+\frac{\partial_{x_{\alpha}}(h_{1}+h_{2})(x')}{2}\boldsymbol\psi_{3}\Big(k^2(x)-\frac{1}{4}\Big),\quad\alpha=1,2,\quad\hbox{in}\ \Omega_{2R},
$$
which satisfies $\nabla\cdot{\bf v}_{1}^{\alpha}=0$ in $\Omega_{2R}$. 
\end{remark}

\subsection{Estimates of $|\nabla{\bf u}_1^3|$ and $|p_1^3|$}\label{subsec3.2}
Recalling the definition of ${\bf v}_{1}^{3}$ in $\Omega_{2R}$, \eqref{v133D}, by direct calculations, we deduce
\begin{align}
\partial_{x_{j}}({\bf v}_{1}^{3})^{(j)}&=\left(\frac{3}{\delta(x')}-\frac{18x_{j}^{2}}{\delta^{2}(x')}\right)k^{2}(x)-\frac{1}{4}\left(\frac{3}{\delta(x')}-\frac{6x_{j}^{2}}{\delta^{2}(x')}\right),\quad\,j=1,2;\label{estv131}\\
\partial_{x_{l}}({\bf v}_{1}^{3})^{(j)}&=\frac{18x_{1}x_{2}}{\delta^{2}(x')}k^{2}(x)+\frac{3}{2}\frac{x_{1}x_{2}}{\delta^{2}(x')},~ \partial_{x_3}({\bf v}_{1}^{3})^{(j)}=\frac{6x_{j}}{\delta^2(x')}k(x),~j,l=1,2,j\neq\,l;\label{estv132}
\end{align}
while, for $j=3$,
\begin{equation}\label{estv133}
\partial_{x_l}({\bf v}_{1}^{3})^{(3)}:=A^{33}_{1}(x')x_{3}+A^{33}_{3}(x')x_{3}^{3},\quad\mbox{and}\quad|\partial_{x_l}({\bf v}_{1}^{3})^{(3)}|\leq\frac{C|x_l|}{\delta(x')},\quad l=1,2,
\end{equation}
where the coefficients of $x_{3}^{k}$, $A^{33}_{k}(x')$, are just some polynomials of $x'$, and
\begin{equation}\label{estv1333}
\partial_{x_3}({\bf v}_{1}^{3})^{(3)}=\left(\frac{18|x'|^{2}}{\delta^{2}(x')}-\frac{6}{\delta(x')}\right)k^{2}(x)-\frac{3}{2}\left(\frac{|x'|^{2}}{\delta^{2}(x')}-\frac{1}{\delta(x')}\right).
\end{equation}
Combining \eqref{estv131} and \eqref{estv1333} implies that such ${\bf v}_{1}^{3}$ is incompressible in $\Omega_{2R}$, that is,
\begin{equation*}%\label{divv13}
\nabla\cdot{\bf v}_{1}^{3}=0,\quad\hbox{in}\ \Omega_{2R}.
\end{equation*}
Consequently, it follows from \eqref{estv131}--\eqref{estv1333} that
\begin{equation*}%\label{v13upper}
|\nabla{\bf v}_{1}^{3}|\leq\,C\left(\frac{1}{\delta(x')}+\frac{|x'|}{\delta^{2}(x')}\right),\quad|\nabla{\bf v}_{1}^{3}(0',x_{3})|\geq\,\frac{1}{C}\frac{1}{\delta(x')}.
\end{equation*}
Notice that
\begin{equation}\label{v13-0}
\int_{\Omega}|\nabla {\bf v}_{1}^{3}|^{2}\mathrm{d}x\leq C\int_{\Omega}\left(\frac{1}{\delta(x')}+\frac{|x'|}{\delta^{2}(x')}\right)^{2}\mathrm{d}x\leq\frac{C}{\varepsilon},
\end{equation}
it is impossible prove the boundedness of the global energy of ${\bf w}_{1}^{3}$ by directly applying Lemma \ref{lem3.0}. So we must improve our technique later. 

In order to control ${\bf f}_{1}^3=\mu\Delta{\bf v}_{1}^3-\nabla\bar{p}_1^3$, we need to check the second order derivatives of ${\bf v}_{1}^3$. It is easy to see from \eqref{estv1333} that $|\partial_{x_3x_3} ({\bf v}_{1}^3)^{(3)}|\leq\frac{C}{\delta^2(x')}$, which is the biggest term in $\partial_{x_{l}}^{2}({\bf v}_{1}^3)^{(j)}$. So we construct a $\bar{p}_1^3\in C^{1}(\Omega)$, defined in \eqref{p133D}, 
$$\bar{p}_1^3=-\frac{3}{2}\frac{\mu}{\delta^2(x')}+\frac{18\mu}{\delta(x')}\left(\frac{|x'|^{2}}{\delta(x')}-\frac{1}{3}\right)k^{2}(x),\quad\mbox{in}~\Omega_{2R}.$$
We notice that at this moment we can not directly get $\mu\partial_{x_3}({\bf v}_{1}^{3})^{(3)}-\bar{p}_1^3=0$ as we dealt with ${\bf v}_{1}^{1}$ in Subsection \ref{subsec3.1}. Indeed,
\begin{equation}\label{eqv13-p}
\mu\partial_{x_3}({\bf v}_{1}^{3})^{(3)}-\bar{p}_1^3=\frac{3}{2}\frac{\mu}{\delta^2(x')}-\frac{3}{2}\mu\left(\frac{|x'|^{2}}{\delta^{2}(x')}-\frac{1}{\delta(x')}\right).
\end{equation}
However, by observation, we find that the right hand side of \eqref{eqv13-p} is independent of $x_{3}$. After differentiating it with respect to $x_{3}$, we still have
\begin{align}\label{estv133p13}
\mu\partial_{x_3x_3} ({\bf v}_{1}^3)^{(3)}-\partial_{x_3}\bar{p}_1^3=0.
\end{align}
Thus, the biggest term $\partial_{x_3x_3} ({\bf v}_{1}^3)^{(3)}$ vanishes.

Meanwhile, 
\begin{equation*}
\partial_{x_j}\bar{p}_1^3=\frac{6\mu x_j}{\delta^3(x')}+\frac{36\mu }{\delta^2(x')}\left(x_j+x_l-\frac{4x_{j}|x'|^{2}}{\delta(x')}\right)k^{2}(x), \quad j,l=1,2,~j\neq l,
\end{equation*}
can also cancel the other big terms
$$\partial_{x_3x_{3}}({\bf v}_{1}^{3})^{(j)}=\frac{6x_{j}}{\delta^3(x')},$$
which implies, for $j,l=1,2$ and $j\neq l$,
\begin{align}\label{estv131p13}
\Big|\mu\partial_{x_3x_3} ({\bf v}_{1}^3)^{(j)}-\partial_{x_j}\bar{p}_1^3\Big|=\left|-\frac{36\mu}{\delta^2(x')}\left(x_j+x_l-\frac{4x_{j}|x'|^{2}}{\delta(x')}\right)k^{2}(x)\right|\leq\frac{C|x'|}{\delta^2(x')}.
\end{align}
For other terms,
\begin{align}\label{estv131p132}
|\partial_{x_1x_1}({\bf v}^{3}_{1})^{(j)}|,|\partial_{x_2x_2}({\bf v}^{3}_{1})^{(j)}|\leq\frac{C|x_j|}{\delta^2(x')},\quad |\partial_{x_1x_1}({\bf v}^{3}_{1})^{(3)}|,|\partial_{x_2x_2}({\bf v}^{3}_{1})^{(3)}|\leq\frac{C}{\delta(x')}.
\end{align}
To summarise, we have
\begin{align}\label{estf13}
|{\bf f}_{1}^3|=\left|\mu\Delta{\bf v}_{1}^3-\nabla\bar{p}_1^3\right|\leq\frac{C|x'|}{\delta^{2}(x')},\quad \mbox{in}~\Omega_{2R}.
\end{align}
As a consequence, we obtain
\begin{align}\label{L2_f13}
\int_{\Omega_{s}(z')}|{\bf f}_{1}^{3}|^{2}\mathrm{d}x\leq\frac{Cs^2}{\delta^{3}(z')}(s^2+|z'|^2),
\end{align}
which is to prepare for applying \eqref{iterating1} to prove Proposition \ref{propu133D}.

\begin{lemma}\label{lem_energyw13}
Let $({\bf w}_i^3,q_i^3)$ be the solution to \eqref{w1}, for $\alpha=3$. Then
\begin{align}\label{energyw13}
\int_{\Omega}|\nabla {\bf w}_{i}^{3}|^{2}\mathrm{d}x\leq C,\quad\,i=1,2.
\end{align}
\end{lemma}

\begin{proof}
Similarly as in Lemma \ref{lem3.0}, by virtue of Lemma \ref{lemmaenergy}, it suffice to proving 
\begin{align}\label{energyDw13}
\Big| \int_{\Omega_{r_{0}}}\sum_{j=1}^{3}({\bf f}_{1}^{3})^{(j)}({\bf w}_{1}^{3})^{(j)}\mathrm{d}x\Big|\leq\,C\left(\int_{\Omega}|\nabla {\bf w}_{1}^{3}|^2\mathrm{d}x\right)^{1/2},
\end{align}
where $r_{0}\in(R,\frac{3}{2}R)$ is fixed by \eqref{w1Dw1}.

First, for $j=1,2$, in view of \eqref{estv131p13}, 
\begin{align}\label{fw12}
\int_{\Omega_{r_{0}}} ({\bf f}_{1}^{3})^{(j)}({\bf w}_{1}^{3})^{(j)}\mathrm{d}x
=\int_{\Omega_{r_{0}}} ({\bf w}_1^3)^{(j)}(\mu\Delta({\bf v}^{3}_{1})^{(j)}-\partial_{x_{j}}\bar{p}_1^3)\mathrm{d}x.
\end{align}
Note that here we can not use the integration by parts with the respect to $x'$ any more as in Lemma \ref{lem3.0}, since after that the terms $|\partial_{x_1}({\bf v}^{3}_{1})^{(j)}|\leq\frac{C}{\delta(x')}$ are still too large to prove \eqref{energyDw13}, see for example \eqref{v13-0}. Observing from \eqref{estv131}, \eqref{estv132} and \eqref{estv131p132}, we can write $({\bf f}_{1}^{3})^{(j)}$ in the polynomial form
$$\mu\Delta({\bf v}^{3}_{1})^{(j)}-\partial_{x_{j}}\bar{p}_1^3:=A^{3j}_{0}(x')+A^{3j}_{2}(x')x_{3}^{2},$$
where $A^{3j}_{l}(x')$ are rational functions of $x_{1}$ and $x_{2}$, and in view of \eqref{estv131p13} and \eqref{estv131p132},  
$$\Big|A^{3j}_{0}(x')+A^{3j}_{2}(x')x_{3}^{2}\Big|\leq\frac{C|x'|}{\delta^{2}(x')}.$$
However, by simple differentiation, we know
$$\partial_{x_{3}}\Big(A^{3j}_{0}(x')x_{3}+\frac{1}{3}A^{3j}_{2}(x')x_{3}^{3}\Big)=A^{3j}_{0}(x')+A^{3j}_{2}(x')x_{3}^{2},$$
while at this moment,  
$$\int_{\Omega_{r_{0}}}\Big|A^{3j}_{0}(x')x_{3}+\frac{1}{3}A^{3j}_{2}(x')x_{3}^{3}\Big|^{2}\leq\,C\int_{\Omega_{r_{0}}}\frac{|x'|^{2}}{\delta^{2}(x')}\leq\,C.$$
So, we rewrite \eqref{fw12} as follows:
\begin{align*}%\label{fw122}
\int_{\Omega_{r_{0}}} ({\bf f}_{1}^{3})^{(j)}({\bf w}_{1}^{3})^{(j)}\mathrm{d}x
=\int_{\Omega_{r_{0}}} ({\bf w}_{1}^{3})^{(j)}\partial_{x_{3}}\Big(A^{3j}_{0}(x')x_{3}+\frac{1}{3}A^{3j}_{2}(x')x_{3}^{3}\Big)\mathrm{d}x.
\end{align*}
Since ${\bf w}=0$ on $\partial D_{1}\cup \partial D_{2}$, we apply the integration by parts with respect to $x_{3}$, instead of $x'$ used in Lemma \ref{lem3.0}. By H\"older inequality, we deduce 
\begin{align*}
\left|\int_{\Omega_{r_{0}}} ({\bf f}_{1}^{3})^{(j)}({\bf w}_{1}^{3})^{(j)}\mathrm{d}x\right|&\leq\int_{\Omega_{r_{0}}}|\partial_{x_{3}}({\bf w}_{1}^{3})^{(j)}|\Big|A^{3j}_{0}(x')x_{3}+\frac{1}{3}A^{3j}_{2}(x')x_{3}^{3}\Big|\mathrm{d}x\nonumber\\
&\leq\,C\left(\int_{\Omega}|\nabla {\bf w}_{1}^{3}|^2\mathrm{d}x\right)^{1/2}.
\end{align*}
Hence,
\begin{equation}\label{estf13w13}
\left|\sum_{j=1}^2\int_{\Omega_{r_{0}}} ({\bf f}_{1}^{3})^{(j)}({\bf w}_{1}^{3})^{(j)}\mathrm{d}x\right|\leq C\left(\int_{\Omega}|\nabla {\bf w}_{1}^{3}|^2\mathrm{d}x\right)^{1/2}.
\end{equation}

For $j=3$, in view of \eqref{estv133p13}, and by means of \eqref{estv133}, 
$$\int_{\Omega_{r_{0}}}|\nabla_{x'}({\bf v}^{3}_{1})^{(3)}|^{2}\mathrm{d}x\leq\,C\int_{\Omega_{r_{0}}}\frac{|x'|^{2}}{\delta^{2}(x')}\leq\,C,$$
so we can still use the integration by parts with respect to $x'$, similarly as in Lemma \ref{lemmaenergy}, and  use \eqref{w1Dw1} again to obtain 
\begin{align}\label{estf1333}
\left|\int_{\Omega_{r_{0}}} ({\bf f}_{1}^{3})^{(3)}({\bf w}_{1}^{3})^{(3)}\mathrm{d}x\right|
&=\left|\int_{\Omega_{r_{0}}} ({\bf w}_{1}^{3})^{(3)}(\mu\partial_{x_1x_1}({\bf v}^{3}_{1})^{(3)}+\mu\partial_{x_2x_2}({\bf v}^{3}_{1})^{(3)})\mathrm{d}x\right|\nonumber\\
&\leq\int_{\Omega_{r_{0}}}|\nabla_{x'}({\bf w}_{1}^{3})^{(3)}||\mu\partial_{x_1}({\bf v}^{3}_{1})^{(3)}+\mu\partial_{x_2}({\bf v}^{3}_{1})^{(3)}|\mathrm{d}x\nonumber\\
&\quad+\int_{\substack{|x'|=r_{0},\\-\epsilon/2-h_{2}(x')<x_{3}<\epsilon/2+h_{1}(x')}}|({\bf w}_{1}^{3})^{(3)}|\mathrm{d}x'\nonumber\\
&\leq C\left(\int_{\Omega}|\nabla ({\bf w}_{1}^{3})|^2\mathrm{d}x\right)^{1/2}.
\end{align}
Taking into account \eqref{estf13w13} and \eqref{estf1333}, we obtain \eqref{energyDw13}, so \eqref{int-fw}. Thanks to Lemma \ref{lemmaenergy}, the proof of Lemma \ref{lem_energyw13} is finished. 
\end{proof}

With \eqref{L2_f13} and \eqref{energyw13}, we may argue as before to complete the proof of Proposition \ref{propu133D}.

\begin{proof}[Proof of Proposition \ref{propu133D}.]
Substituting \eqref{L2_f13} into  \eqref{iterating1}, instead of \eqref{iteration3D}, 
\begin{align*}
\int_{\Omega_{t}(z')}|\nabla {\bf w}_{1}^{3}|^{2}\mathrm{d}x\leq\,&
\frac{C\delta^2(z')}{(s-t)^{2}}\int_{\Omega_{s}(z')}|\nabla {\bf w}_{1}^{3}|^2\mathrm{d}x+C\left((s-t)^{2}+\delta^{2}(z')\right)\frac{s^2}{\delta^3(z')}(s+|z'|^2).
\end{align*}
By the same iteration process used in Lemma \ref{lem3.1}, we deduce 
\begin{align*}%\label{estw13narrow3D}
\int_{\Omega_{\delta}(z')}|\nabla {\bf w}_1^3|^{2}\mathrm{d}x\leq C\delta^2(z').
\end{align*}
Substituting this, together with \eqref{estf13}, into \eqref{W2pstokes} and  \eqref{W2pstokes}  yields
\begin{align*}
\|\nabla {\bf w}_1^3\|_{L^{\infty}(\Omega_{\delta/2}(z'))}
&\leq
\,C\left(\delta^{-\frac{3}{2}}\|\nabla {\bf w}_1^3\|_{L^{2}(\Omega_{\delta}(z'))}+\delta\|{\bf f}_1^3\|_{L^{\infty}(\Omega_{\delta}(z'))} \right)\nonumber\\
&\leq\,C\left(\delta^{1-\frac{3}{2}}(z')+\frac{|z'|}{\delta(z')}\right)
\leq\,\frac{C}{\sqrt{\delta(z')}},
\end{align*}
and 
\begin{align*}
&\|\nabla^2 {\bf w}_1^3\|_{L^{\infty}(\Omega_{\delta/2}(z'))}+\|\nabla q_1^3\|_{L^{\infty}(\Omega_{\delta/2}(z'))}\nonumber\\
&\leq
\,C\Big(\delta^{-\frac{5}{2}}\|\nabla {\bf w}_1^3\|_{L^{2}(\Omega_{\delta}(z'))}+\|{\bf f}_1^3\|_{L^{\infty}(\Omega_{\delta}(z'))}+\delta\|\nabla{\bf f}_1^3\|_{L^{\infty}(\Omega_{\delta}(z'))}\Big)\leq\,\frac{C}{\delta^{3/2}(z')},
\end{align*}
Recalling  ${\bf w}_1^3={\bf u}_1^3-{\bf v}_1^3$, $q_1^3=p_1^3-\bar p_1^3$, and using \eqref{v13upper1}--\eqref{p133D}, Proposition \ref{propu133D} is proved. 
\end{proof}

\subsection{Estimates of $|\nabla{\bf u}_1^\alpha|$ and $|p_1^\alpha|$, $\alpha=4,5,6$.}\label{subsec4.3}

For ${\bf u}_{1}^{4}$, because ${\boldsymbol\psi}_{4}$ does not include the component $x_{3}$, it is easier to construct ${\bf v}_{1}^{4}\in C^{2}(\Omega;\mathbb R^3)$. Set  
\begin{align*}%\label{v12}
{\bf v}_{1}^{4}=
{\boldsymbol\psi}_{4}\left(k(x)+\frac{1}{2}\right),\quad \mbox{in}~\Omega_{2R},
\end{align*}
and $\|{\bf v}_{1}^{4}\|_{C^{2}(\Omega\setminus\Omega_{R})}\leq\,C$. Then, a simple calculation shows that
\begin{align*}%\label{estv14}
\begin{split}
\partial_{x_1}({\bf v}_{1}^{4})^{(1)}&=-\frac{2x_1x_2}{\delta(x')}k(x),\quad |\partial_{x_2}({\bf v}_{1}^{4})^{(1)}|\leq C,\quad\partial_{x_3}({\bf v}_{1}^{4})^{(1)}=\frac{x_2}{\delta(x')};\\
|\partial_{x_1}({\bf v}_{1}^{4})^{(2)}|&\leq C,\quad\partial_{x_2}({\bf v}_{1}^{4})^{(2)}=\frac{2x_1x_2}{\delta(x')}k(x),\quad\partial_{x_3}({\bf v}_{1}^{4})^{(2)}=-\frac{x_1}{\delta(x')},
\end{split}
\end{align*}
which directly gives
\begin{equation*}%\label{divv14}
\nabla\cdot{\bf v}_1^4=0,\quad\hbox{in}\ \Omega_{2R}.
\end{equation*}
Further, for the second order derivatives
\begin{equation*}
|\partial_{x_1x_1}({\bf v}_{1}^{4})^{(\alpha)}|,|\partial_{x_2x_2}({\bf v}_{1}^{4})^{(\alpha)}|\leq \frac{C|x'|}{\delta(x')},\quad\alpha=1,2,
\end{equation*}
so that
\begin{equation*}%\label{Deltav14}
\int_{\Omega_{R}}|\mu\Delta{\bf v}_1^4|^{2}\leq\,C \int_{\Omega_{R}}\frac{|x'|^{2}}{\delta^{2}(x')}\leq\,C.
\end{equation*}

\begin{proof}[Proof of Proposition \ref{propu4563D} for $\alpha=4$]
Since the above bounds for ${\bf v}_{1}^{4}$ are much better than the analogues of the case $\alpha=1,2$, and satisfy the conditions in Lemma \ref{lemmaenergy} with taking $\bar{p}_1^{4}=0$, the rest of the Proof of Proposition \ref{propu4563D}  for $\alpha=4$ is the 
same as that for Proposition \ref{propu113D}. We omit the details here.
\end{proof}

For ${\bf u}_{1}^{5}$ and ${\bf u}_{1}^{6}$, we construct  ${\bf v}_{1}^{5}\in C^{2}(\Omega;\mathbb R^3)$ in \eqref{v15}. By a careful calculation and using \eqref{partial_kx}, we have
\begin{align*}
\partial_{x_1}({\bf v}_{1}^{5})^{(1)}=-2x_{1}k^{2}(x)-&\frac{3}{5}\Big(k^2(x)-\frac{1}{4}\Big)\Big(\frac{8x_{1}}{\delta(x')}-\frac{8x_{1}^{3}}{\delta^{2}(x')}-\frac{50x_{1}}{3}k^{2}(x)\Big)\\
-&\frac{3}{5}k^2(x)\Big(\frac{4x_{1}}{\delta(x')}-\frac{16x_{1}^{3}}{\delta^{2}(x')}-\frac{100x_{1}}{3}k^{2}(x)\Big);\\
\partial_{x_2}({\bf v}_{1}^{5})^{(2)}=-\frac{3}{5}\Big(k^2(x)-&\frac{1}{4}\Big)\Big(\frac{4x_{1}}{\delta(x')}-\frac{8x_{1}x_{2}^{2}}{\delta^{2}(x')}\Big)+\frac{3}{5}\frac{16x_1x_2^2}{\delta^2(x')}k^{2}(x);
\end{align*}
and
\begin{align*}
\partial_{x_3}({\bf v}_{1}^{5})^{(3)}=-\frac{x_1}{\delta(x')}&+\frac{3}{5}\Big(k^2(x)-\frac{1}{4}\Big)\left(\frac{16x_{1}}{3\delta(x')}-\frac{8x_{1}|x'|^{2}}{\delta^{2}(x')}-30x_{1}k^{2}(x)\right)\\
&+\frac{3}{5}k^2(x)\left(\frac{32x_{1}}{3\delta(x')}-\frac{16x_{1}|x'|^{2}}{\delta^{2}(x')}-20x_{1}k^{2}(x)\right).
\end{align*}
The above calculation shows that
\begin{equation*}%\label{divv15}
\nabla\cdot{\bf v}_1^5=\partial_{x_1}({\bf v}_{1}^{5})^{(1)}+\partial_{x_2}({\bf v}_{1}^{5})^{(2)}+\partial_{x_3}({\bf v}_{1}^{5})^{(3)}=0,\quad\hbox{in}\ \Omega_{2R}.
\end{equation*}
For the other terms,
\begin{equation}\label{v1511}
|\partial_{x_2}({\bf v}_{1}^{5})^{(1)}|\leq \frac{C|x_2|}{\delta(x')},~~
\partial_{x_3}({\bf v}_{1}^{5})^{(1)}=\frac{9}{2}k(x)+\frac{6}{5}\left(\frac{1}{\delta(x')}-\frac{4x_{1}^2}{\delta^{2}(x')}\right)k(x)-20k^{3}(x);
\end{equation}
\begin{equation}\label{v1512}
|\partial_{x_1}({\bf v}_{1}^{5})^{(2)}|\leq \frac{C|x_2|}{\delta(x')},\quad
\partial_{x_3}({\bf v}_{1}^{5})^{(2)}=-\frac{24}{5}\frac{x_{1}x_2}{\delta^2(x')}k(x);
\end{equation}
and
\begin{equation}\label{v1513}
|\partial_{x_j}({\bf v}_{1}^{5})^{(j)}|\leq\frac{C|x'|}{\delta(x')},\quad|\partial_{x_1}({\bf v}_{1}^{5})^{(3)}|,|\partial_{x_2}({\bf v}_{1}^{5})^{(3)}|\leq C.
\end{equation}

For the second derivatives,
\begin{align*}
&\partial_{x_3x_3}({\bf v}_{1}^{5})^{(1)}=\frac{9}{2\delta(x')}+\frac{6}{5}\left(\frac{1}{\delta^{2}(x')}-\frac{4x_{1}^2}{\delta^{3}(x')}\right)-\frac{60}{\delta(x')}k^{2}(x);\\
&\partial_{x_3x_3}({\bf v}_{1}^{5})^{(2)}=-\frac{24}{5}\frac{x_{1}x_2}{\delta^3(x')};
\end{align*}
and
\begin{align*}
\partial_{x_3x_3}({\bf v}_{1}^{5})^{(3)}=\frac{48}{5}\left(\frac{2x_{1}}{\delta^{2}(x')}-\frac{3x_{1}|x'|^{2}}{\delta^{3}(x')}\right)k(x)-\frac{120x_{1}}{\delta(x')}k^{3}(x)
+\frac{9x_{1}}{\delta(x')}k(x).
\end{align*}
Note that all of these three terms above are large,
\begin{align}\label{laplacev15}
|\partial_{x_3x_3}({\bf v}_{1}^{5})^{(j)}|\leq\frac{C}{\delta^{2}(x')},~j=1,2,
\quad|\partial_{x_3x_3}({\bf v}_{1}^{5})^{(3)}|\leq\frac{C|x'|}{\delta^{2}(x')}.
\end{align}

Here we choose $\bar{p}_1^5\in C^{1}(\Omega)$ such that
\begin{equation*}
\bar{p}_1^5=\frac{6\mu}{5}\frac{ x_1}{\delta^2(x')}+\frac{72\mu}{5}\frac{ x_1}{\delta(x')}\left(\frac{2}{3}-\frac{|x'|^{2}}{\delta(x')}\right)k^{2}(x),\quad\mbox{in}~\Omega_{2R},
\end{equation*}
and $
\|\bar{p}_1^5\|_{C^{1}(\Omega\setminus\Omega_{R})}\leq C$.
A straightforward calculation gives
$$\partial_{x_1}\bar{p}_1^5=\frac{6\mu}{5}\left(\frac{1}{\delta^{2}(x')}-\frac{4x_{1}^2}{\delta^{3}(x')}\right)+\frac{72\mu}{5}\left(\frac{2}{3\delta(x')}-\frac{6x_{1}^{2}}{\delta^{2}(x')}-\frac{|x'|^{2}}{\delta^{2}(x')}+\frac{8x_{1}^{2}|x'|^{2}}{\delta^{3}(x')}\right)k^{2}(x);$$
$$\partial_{x_2}\bar{p}_1^5=-\frac{24\mu}{5}\frac{x_{1}x_{2}}{\delta^{3}(x')}+\frac{72\mu}{5}\left(-\frac{4}{3}\frac{x_{1}x_{2}}{\delta(x')}+\frac{2x_{1}x_{2}|x'|^{2}}{\delta^{2}(x')}-\frac{2x_{1}x_{2}}{\delta^{2}(x')}+\frac{2x_{1}x_{2}|x'|^{2}}{\delta^{3}(x')}\right)k^{2}(x);$$
and
$$\partial_{x_3}\bar{p}_1^5=\mu\left(\frac{96}{5}\frac{x_{1}}{\delta^{2}(x')}-\frac{144}{5}\frac{x_{1}|x'|^{2}}{\delta^{3}(x')}\right)k(x).$$

Consequently,
\begin{align}
&\mu\partial_{x_3x_3} ({\bf v}_{1}^5)^{(1)}-\partial_{x_1}\bar{p}_1^5=\frac{9\mu}{2\delta(x')}-\frac{72\mu}{5}\left(\frac{29}{6\delta(x')}-\frac{(6x_{1}^{2}+|x'|^{2})}{\delta^{2}(x')}+\frac{8x_{1}^{2}|x'|^{2}}{\delta^{3}(x')}\right)k^{2}(x);\label{estv151-p151}\\
&\mu\partial_{x_3x_3} ({\bf v}_{1}^5)^{(2)}-\partial_{x_2}\bar{p}_1^5=-\frac{72\mu x_{1}x_{2}}{5}\left(-\frac{4}{3\delta(x')}+\frac{2(|x'|^{2}-1)}{\delta^{2}(x')}+\frac{2|x'|^{2}}{\delta^{3}(x')}\right)k^{2}(x);\label{estv152-p152}
\end{align}
and
\begin{align}\label{estv153-p153}
\mu\partial_{x_3x_3} ({\bf v}_{1}^5)^{(3)}-\partial_{x_3}\bar{p}_1^5=\frac{9\mu x_{1}}{\delta(x')}k(x)-\frac{120\mu x_{1}}{\delta(x')}k^{3}(x).
\end{align}
Obviously, they all do not vanish, but they make $|{\bf f}_{1}^5|$ become smaller. 

In summary,
\begin{align*}%\label{Deltav15}
|{\bf f}_{1}^5|=\left|\mu\Delta{\bf v}_{1}^5-\nabla\bar{p}_1^5\right|\leq\frac{C}{\delta(x')},\quad\mbox{in}~\Omega_{2R}.
\end{align*}
This is much better than \eqref{laplacev15}. Moreover,
\begin{align}\label{L2_f15}
\int_{\Omega_{s}(z')}|{\bf f}_{1}^{5}|^{2}\mathrm{d}x\leq\frac{Cs^2}{\delta(z')}.
\end{align}
Notice that this bound is the same as in \eqref{f1alpha} for ${\bf f}_{1}^{1}$. So, in order to apply \eqref{iteration3D} and the iteration process as in Lemma \ref{lem3.1}, it remains to prove the boundedness of $E_{1}^{\alpha}$, $\alpha=5,6$.

\begin{lemma}\label{lem3.4}
Let $({\bf w}_i^\alpha,q_i^\alpha)$ be the solution to \eqref{w1}, for $\alpha=5,6$. Then
\begin{align}\label{energy_w15}
E_{1}^{\alpha}=\int_{\Omega}|\nabla{\bf w}_{1}^{\alpha}|^{2}\mathrm{d}x\leq\,C,\quad\,\alpha=5,6.
\end{align}
\end{lemma}

\begin{proof}

Similarly as in the proof of Lemma \ref{lem_energyw13}, we rewrite \eqref{estv151-p151}--\eqref{estv153-p153} into the form of polynomials as follows:
\begin{align*}
\mu\partial_{x_3x_3} ({\bf v}_{1}^5)^{(1)}-\partial_{x_1}\bar{p}_1^5&:=A^{51}_{0}(x')+A^{51}_{2}(x')x_{3}^{2}=\partial_{x_{3}}\Big(A^{51}_{0}(x')x_{3}+\frac{1}{3}A^{51}_{2}(x')x_{3}^{3}\Big),\\
\mu\partial_{x_3x_3} ({\bf v}_{1}^5)^{(2)}-\partial_{x_2}\bar{p}_1^5&:=A^{52}_{2}(x')x_{3}^{2}=\partial_{x_{3}}\Big(\frac{1}{3}A^{52}_{2}(x')x_{3}^{3}\Big),\\
\mu\partial_{x_3x_3} ({\bf v}_{1}^5)^{(3)}-\partial_{x_3}\bar{p}_1^5&:=A^{53}_{1}x_{3}+A^{53}_{3}x_3^{3}=\partial_{x_{3}}\Big(\frac{1}{2}A^{53}_{1}x_{3}^{2}+\frac{1}{4}A^{53}_{3}x_3^{4}\Big),
\end{align*}
where $A^{5j}_{k}(x')$ are polynomials of $x'$. Moreover, 
\begin{align*}
\Big|A^{51}_{0}(x')x_{3}+\frac{1}{3}A^{51}_{2}(x')x_{3}^{3}\Big|,\Big|\frac{1}{3}A^{52}_{2}(x')x_{3}^{3}\Big|,\Big|\frac{1}{2}A^{53}_{1}x_{3}^{2}+\frac{1}{4}A^{53}_{3}x_3^{4}\Big|\leq\,C.
\end{align*}
Applying the integration by parts with respect to $x_{3}$, and using ${\bf w}=0$ on $\partial D_{1}\cup \partial D_{2}$ and H\"older inequality, gives
\begin{align*}
&\left|\int_{\Omega_{r_{0}}} (\mu\partial_{x_3x_3} ({\bf v}_{1}^5)^{(j)}-\partial_{x_j}\bar{p}_1^5)({\bf w}_{1}^{5})^{(j)}\mathrm{d}x\right|\\
\leq&\,C\int_{\Omega_{r_{0}}}|\partial_{x_{3}}({\bf w}_{1}^{5})^{(j)}|\mathrm{d}x
\leq\,C\left(\int_{\Omega}|\nabla {\bf w}_{1}^{5}|^2\mathrm{d}x\right)^{1/2},\quad\,j=1,2,3.
\end{align*}
For other terms $\Delta_{x'}({\bf v}_{1}^5)^{(j)}$, since \eqref{v1511}--\eqref{v1513}, we still use the integration by parts with respect to $x'$ as  \eqref{estf1333} in Lemma \ref{lem_energyw13}.
\begin{align*}
\left|\int_{\Omega_{r_{0}}} ({\bf w}_{1}^{5})^{(j)}(\mu\partial_{x_1x_1}({\bf v}^{5}_{1})^{(j)}+\mu\partial_{x_2x_2}({\bf v}^{5}_{1})^{(j)})\mathrm{d}x\right|\leq C\left(\int_{\Omega}|\nabla ({\bf w}_{1}^{5})|^2\mathrm{d}x\right)^{1/2}.
\end{align*}
Thus,
$$\Big| \int_{\Omega_{r_{0}}}\sum_{j=1}^{3}({\bf f}_{1}^{5})^{(j)}({\bf w}_{1}^{5})^{(j)}\mathrm{d}x\Big|\leq\,C\left(\int_{\Omega}|\nabla {\bf w}_{1}^{5}|^2\mathrm{d}x\right)^{1/2}.
$$
Thanks to Lemma \ref{lemmaenergy}, we obtain \eqref{energy_w15} for $\alpha=5$.

While,  ${\bf v}_{1}^{6}\in C^{2}(\Omega;\mathbb R^3)$ is very similar to ${\bf v}_{1}^{5}$, namely, in $\Omega_{2R}$,
\begin{align*}%\label{v12}
{\bf v}_{1}^{6}={\boldsymbol\psi}_{6}\left(k(x)+\frac{1}{2}\right)
+\frac{3}{5}\Big(k^2(x)-\frac{1}{4}\Big)(G(x),F_2(x),H_2(x))^{\mathrm T},
\end{align*}
and $\|{\bf v}_{1}^{6}\|_{C^{2}(\Omega\setminus\Omega_{R})}\leq\,C$, where $F_2(x),G(x)$, and $H_2(x)$ are defined in \eqref{defFGH}. Construct $\bar{p}_1^6\in C^{1}(\Omega)$ such that
\begin{equation*}
\bar{p}_1^6=\frac{6\mu}{5}\frac{ x_2}{\delta^2(x')}+\frac{72\mu}{5}\frac{ x_2}{\delta(x')}\left(\frac{2}{3}-\frac{|x'|^{2}}{\delta(x')}\right)k^{2}(x),\quad\mbox{in}~\Omega_{2R},
\end{equation*}
and $
\|\bar{p}_1^6\|_{C^{1}(\Omega\setminus\Omega_{R})}\leq C$.
Thus, the desired estimates on ${\bf v}_{1}^{6}$ and $\bar{p}_1^6$ are all the same with those of ${\bf v}_{1}^{5}$ and $\bar{p}_1^5$. We omit the details and finish the proof of Lemma \ref{lem3.4}. 
\end{proof}

\begin{proof}[Proof of Proposition \ref{propu4563D} for $\alpha=5,6$]
Notice that the above bounds for ${\bf v}_{1}^{\alpha}$ and $\bar{p}_1^{\alpha}$, for example, \eqref{L2_f15} and \eqref{energy_w15}, are the same as those of the case $\alpha=1,2$. Following exactly the same proof as for Proposition \ref{propu113D}, 
we have Proposition \ref{propu4563D} holds.
\end{proof}

 \subsection{Estimates of $C_i^\alpha$ in 3D}\label{subsecCi}
Define
\begin{align}\label{aijbj}
a_{ij}^{\alpha\beta}:=-
\int_{\partial D_j}{\boldsymbol\psi}_\beta\cdot\sigma[{\bf u}_{i}^\alpha,p_{i}^{\alpha}]\nu,\quad
b_{j}^{\beta}:=
\int_{\partial D_j}{\boldsymbol\psi}_\beta\cdot\sigma[{\bf u}_{0},p_{0}]\nu.
\end{align}
Then using the integration by parts, \eqref{equ_v1}, and \eqref{equ_v3}, we have 
\begin{align}\label{defaij}
a_{ij}^{\alpha\beta}=\int_{\Omega} \left(2\mu e({\bf u}_{i}^{\alpha}), e({\bf u}_{j}^{\beta})\right)\mathrm{d}x,\quad
b_{j}^{\beta}=-\int_{\Omega} \left(2\mu e({\bf u}_{0}),e({\bf u}_{j}^\beta)\right)\mathrm{d}x,
\end{align}
and thus, \eqref{equ-decompositon} can be rewritten as
\begin{align}\label{systemC}
\begin{cases}
\sum\limits_{\alpha=1}^{6}C_{1}^{\alpha}a_{11}^{\alpha\beta}
+\sum\limits_{\alpha=1}^{6}C_{2}^{\alpha}a_{21}^{\alpha\beta}
-b_{1}^{\beta}=0,&\\\\
\sum\limits_{\alpha=1}^{6}C_{1}^{\alpha}a_{12}^{\alpha\beta}
+\sum\limits_{\alpha=1}^{6}C_{2}^{\alpha}a_{22}^{\alpha\beta}
-b_{2}^{\beta}=0.
\end{cases}\quad \beta=1,\dots,6.
\end{align}
To prove Proposition \ref{lemCialpha3D}, we need to prove the estimates of  $a_{11}^{\alpha\beta}$ and $b_1^\beta$, $\alpha,\beta=1,\dots,6$.
 
 First, for the diagonal elements $a_{11}^{\alpha\alpha}$, $\alpha=1,\dots,6$.
 \begin{lemma}\label{lema113D}
 We have
 \begin{align}
\frac{1}{C}|\ln\varepsilon|&\leq a_{11}^{\alpha\alpha}\leq C|\ln\varepsilon|,\quad\alpha=1,2,5,6;\label{esta11113D}\\
 \frac{1}{C\varepsilon}&\leq a_{11}^{33}\leq \frac{C}{\varepsilon},
 \quad\mbox{and}\quad\frac{1}{C}\leq a_{11}^{44}\leq C.\label{esta11333D}
 \end{align}
 \end{lemma}
 
\begin{proof}
By means of Proposition \ref{propu113D} and \eqref{defaij}, 
\begin{align*}
a_{11}^{11}\leq C\int_{\Omega}|\nabla {\bf u}_1^1|^2\mathrm{d}x\leq\,C\int_{\Omega_R}\frac{\mathrm{d}x}{\delta^2(x')}\ +C\leq C|\ln\varepsilon|.
\end{align*}
On the other hand, thanks to \eqref{v11-1} and Proposition \ref{propu113D},
\begin{align*}
a_{11}^{11}\geq &\,\frac{1}{C}\int_{\Omega_{R}}|e({\bf u}_1^1)|^2\mathrm{d}x\geq \frac{1}{C}\int_{\Omega_{R}}|\partial_{x_{3}} ({\bf v}_1^1)^{(1)}|^2\mathrm{d}x-C
\geq\, \frac{1}{C}\int_{\Omega_{R}} \frac{\mathrm{d}x}{\delta^{2}(x')} -C\geq\frac{|\ln\varepsilon|}{C}.
\end{align*}
For $\alpha=2,5,6$, the proof is the same and so is omitted. Thus, \eqref{esta11113D} is proved. 

For $a_{11}^{33}$, from Proposition \ref{propu133D}, we have
\begin{align*}
a_{11}^{33}\leq C\int_{\Omega}|\nabla {\bf u}_1^3|^2\mathrm{d}x\leq C\int_{\Omega_R}\left(\frac{1}{\delta(x')}+\frac{|x'|}{\delta^2(x')}\right)^2\mathrm{d}x+C\leq \frac{C}{\varepsilon}.
\end{align*}
For the lower bound, 
\begin{align*}
a_{11}^{33}&\geq\frac{1}{C}\int_{\Omega_R}|e({\bf u}_1^3)|^2\mathrm{d}x\geq\frac{1}{C}\int_{\Omega_R}\sum_{j=1}^{2}|\partial_{x_{3}}({\bf v}_1^3)^{(j)}+\partial_{x_{3}}({\bf w}_1^3)^{(j)})|^2\mathrm{d}x\\
&\geq\frac{1}{C}\int_{\Omega_R}\sum_{j=1}^{2}|\partial_{x_{3}}({\bf v}_1^3)^{(j)}|^2\mathrm{d}x-\frac{2}{C}\int_{\Omega_R}\sum_{j=1}^{2}|\partial_{x_{3}}({\bf v}_1^3)^{(j)}||\partial_{x_{3}}({\bf w}_1^3)^{(j)}|\mathrm{d}x.
\end{align*}
With the aid of \eqref{estv132},
\begin{align*}
&\int_{\Omega_R}\sum_{j=1}^{2}|\partial_{x_{3}}({\bf v}_1^3)^{(j)}|^2 \mathrm{d}x=\int_{\Omega_R}\frac{36|x'|^2x_3^2}{\delta^6(x')}\mathrm{d}x\geq \frac{1}{C}\int_{|x'|<R}\frac{|x'|^2}{\delta^3(x')}\mathrm{d}x'\\
&\geq\frac{1}{C}\int_{0}^R\frac{r^3}{(\varepsilon+r^2)^3}\ dr\geq\frac{1}{C\varepsilon}\int_0^{\frac{R}{\sqrt{\varepsilon}}}\frac{r^3}{(1+r^2)^3}\ dr\geq\frac{1}{C\varepsilon},
\end{align*}
and by Proposition \ref{propu133D},
\begin{equation*}
\int_{\Omega_R}\sum_{j=1}^{2}|\partial_{x_{3}}({\bf v}_1^3)^{(j)}||\partial_{x_{3}}({\bf w}_1^3)^{(j)}|\mathrm{d}x\leq\int_{\Omega_R}\frac{|x'|}{\delta^{5/2}(x')}\mathrm{d}x\leq C.
\end{equation*}
Consequently, $a_{11}^{33}\geq\frac{1}{C\varepsilon}$. Estimate \eqref{esta11333D} for $a_{11}^{33}$ is proved.

On account of Proposition \ref{propu4563D}, we readily have
\begin{align*}
a_{11}^{44}\leq C\int_{\Omega}|\nabla {\bf u}_1^4|^2\mathrm{d}x\leq C\int_{\Omega_R}\frac{|x'|^2}{\delta^2(x')}\, \mathrm{d}x+C\leq C,
\end{align*}
while
\begin{align*}
a_{11}^{44}\geq \int_{\Omega_R\setminus\Omega_{R/2}}|\partial_{x_3}({\bf v}_{1}^{4})^{(2)}|^{2}\geq\int_{\Omega_R\setminus\Omega_{R/2}}\frac{|x'|^{2}}{\delta^{2}(x')}\mathrm{d}x\geq \frac{1}{C}.
\end{align*}
So that estimate \eqref{esta11333D} is proved. 
\end{proof}

For the other terms of $(a_{11}^{\alpha\beta})_{6\times6}$, and $b_1^\beta$, we have
 \begin{lemma}\label{lema114563D}
We have
\begin{align*}
&|a_{11}^{\alpha\beta}|=|a_{11}^{\beta\alpha}|\leq C,~\alpha,\beta=1,\dots,6,~\alpha\neq\beta;\\%\label{esta11123D}\\
&|a_{11}^{\alpha\beta}+a_{21}^{\beta\alpha}|\leq C,\quad\alpha,\beta=1,\dots,6;%\label{estsuma113D}
\end{align*}
and
\begin{equation*}%\label{estb13D}
|b_1^\beta|\leq C,\quad\beta=1,\dots,6.
\end{equation*}
\end{lemma}

\begin{proof}
We shall estimate $a_{11}^{12}$, $a_{11}^{13}$, $a_{11}^{15}$, and $a_{11}^{35}$, for examples. It follows from Proposition \ref{propu113D} and \eqref{defaij} that 
\begin{align*}
a_{11}^{12}=\int_{\Omega_R}(2\mu e({\bf v}_{1}^{1}), e({\bf v}_{1}^{2}))\mathrm{d}x+C.
\end{align*}
From \eqref{v11-1}--\eqref{v11-3}, one can see that 
\begin{align*}
\left|\int_{\Omega_R}(2\mu e({\bf v}_{1}^{1}), e({\bf v}_{1}^{2}))\mathrm{d}x\right|\leq \int_{\Omega_R}\frac{C|x'|}{\delta(x')}\mathrm{d}x\leq C.
\end{align*}
Thus, we obtain $|a_{11}^{12}|\leq C$. 

For $a_{11}^{13}$, by using \eqref{defaij}, Propositions \ref{propu113D} and \ref{propu133D}, we have 
\begin{align*}
a_{11}^{13}=\int_{\Omega_R}(2\mu e({\bf v}_{1}^{1}), e({\bf v}_{1}^{3}))\mathrm{d}x+C.
\end{align*}
In view of \eqref{v11-1}--\eqref{v11-3} and \eqref{estv131}--\eqref{estv133}, we find that the biggest term in $(2\mu e({\bf v}_{1}^{1}), e({\bf v}_{1}^{3}))$ is 
\begin{equation*}
\partial_{x_3}({\bf v}_{1}^{1})^{(1)}\cdot\partial_{x_3}({\bf v}_{1}^{3})^{(1)}=\frac{6x_{1}}{\delta^3(x')}k(x),
\end{equation*}
which is an odd function with respect to $x_1$, and thus the  integral is $0$. The integral of the rest terms is bounded by a constant. Hence, $|a_{11}^{13}|\leq C$. Similarly, from \eqref{estv131}--\eqref{estv133} and \eqref{v1511}--\eqref{v1513}, it follows that the biggest term in $(2\mu e({\bf v}_{1}^{3}), e({\bf v}_{1}^{5}))$ is 
\begin{equation*}
\partial_{x_3}({\bf v}_{1}^{3})^{(1)}\cdot\partial_{x_3}({\bf v}_{1}^{5})^{(1)}=\frac{6x_{1}}{\delta^2(x')}k(x)\Big(\frac{9}{2}k(x)+\frac{6}{5}\left(\frac{1}{\delta(x')}-\frac{4x_{1}^2}{\delta^{2}(x')}\right)k(x)-20k^{3}(x)\Big),
\end{equation*}
which is also an odd function with respect to $x_1$. Combining Propositions \ref{propu133D} and \ref{propu4563D}, the rest terms is bounded by a constant and thus, $|a_{11}^{35}|\leq C$. 

Finally, for $a_{11}^{15}$,  by Propositions \ref{propu113D} and \ref{propu4563D}, we have 
\begin{align*}
a_{11}^{15}=\int_{\Omega_R}(2\mu e({\bf v}_{1}^{1}), e({\bf v}_{1}^{5}))\mathrm{d}x+C.
\end{align*}
Note that the biggest term in $(2\mu e({\bf v}_{1}^{1}), e({\bf v}_{1}^{5}))$ is 
\begin{equation*}
\partial_{x_3}({\bf v}_{1}^{1})^{(1)}\cdot\partial_{x_3}({\bf v}_{1}^{5})^{(1)}=\frac{1}{\delta(x')}\Big(\frac{9}{2}k(x)+\frac{6}{5}\left(\frac{1}{\delta(x')}-\frac{4x_{1}^2}{\delta^{2}(x')}\right)k(x)-20k^{3}(x)\Big).
\end{equation*}
A direct calculation yields 
\begin{align*}
&\left|\int_{\Omega_R}\partial_{x_3}({\bf v}_{1}^{1})^{(1)}\cdot\partial_{x_3}({\bf v}_{1}^{5})^{(1)}\mathrm{d}x\right|\\
&\leq\frac{6}{5}\int_{\Omega_R}\frac{1}{\delta(x')}\left(\frac{1}{\delta(x')}-\frac{4x_{1}^2}{\delta^{2}(x')}\right)k(x)\mathrm{d}x+C\\
&=\frac{6}{5}\int_{|x'|\leq R}\frac{1}{\delta^2(x')}\left(\frac{1}{\delta(x')}-\frac{4x_{1}^2}{\delta^{2}(x')}\right)\mathrm{d}x'\int_{-\frac{\varepsilon+x_1^2}{2}}^{\frac{\varepsilon+x_1^2}{2}}x_3\mathrm{d}x_3+C\leq C.
\end{align*}
Hence, $|a_{11}^{15}|\leq C$. The rest terms can be dealt with by the same way, so we omit the details here. 
\end{proof}

Now let us prove Proposition \ref{lemCialpha3D}.

\begin{proof}[Proof of Proposition \ref{lemCialpha3D}.] By using the trace theorem, we can prove the boundedness of $C_i^\alpha$, similarly as in \cite{BLL2}. To solve $|C_1^\alpha-C_2^\alpha|$, we use the first six equations in \eqref{systemC}:
\begin{align*}
a_{11}(C_{1}-C_{2})=f:=b_1-(a_{11}+a_{21})C_{2},
\end{align*}
where $a_{11}:=(a_{11}^{\alpha\beta})_{\alpha,\beta=1}^6$, $C_{i}:=(C_{i}^{1}, \dots, C_i^6)^{\mathrm{T}}$, and  $b_1:=(b_1^1,\dots, b_1^6)^{\mathrm{T}}$. From \cite{BLL2}, the matrix $a_{11}$  is positive definite. Denote $\mathrm{cof}(A)_{\alpha\beta}$ by the cofactor of $A:=(a_{11}^{\alpha\beta})_{6\times6}$ for simplicity of notation. By  Lemma \ref{lema113D} and Lemma \ref{lema114563D}, we have 
\begin{equation*}
\frac{1}{C\varepsilon}|\ln\varepsilon|^4\leq\det{A}\leq\frac{C}{\varepsilon}|\ln\varepsilon|^4;
\end{equation*}
and
\begin{align*}
\frac{1}{C\varepsilon}|\ln\varepsilon|^3&\leq\mathrm{cof}(A)_{\alpha\alpha}\leq \frac{C}{\varepsilon}|\ln\varepsilon|^3,\quad\alpha=1,2,5,6,\\
\frac{1}{C}|\ln\varepsilon|^4&\leq\mathrm{cof}(A)_{33}\leq C|\ln\varepsilon|^4,\quad\frac{1}{C\varepsilon}|\ln\varepsilon|^4\leq\mathrm{cof}(A)_{44}\leq\frac{C}{\varepsilon}|\ln\varepsilon|^4;\\
&|\mathrm{cof}(A)_{\alpha\beta}|\leq \frac{C}{\varepsilon}|\ln\varepsilon|^2,\quad\alpha,\beta=1,2,5,6,~\alpha\neq\beta,\\
&|\mathrm{cof}(A)_{\alpha 3}|, |\mathrm{cof}(A)_{3\alpha}|\leq C|\ln\varepsilon|^3,\quad\alpha=1,2,5,6,\\
&|\mathrm{cof}(A)_{43}|, |\mathrm{cof}(A)_{34}|\leq C|\ln\varepsilon|^4,\\
&|\mathrm{cof}(A)_{\alpha 4}|, |\mathrm{cof}(A)_{4\alpha}|\leq \frac{C}{\varepsilon}|\ln\varepsilon|^3,\quad\alpha=1,2,5,6.
\end{align*}
Thus, by Cramer's rule, we readily obtain
\begin{equation*}
|C_1^\alpha-C_2^\alpha|\leq\frac{C\,\mathrm{cof}(A)_{\alpha\alpha}}{\det{A}}\leq\frac{C}{|\ln\varepsilon|},\quad\alpha=1,2,5,6,
\end{equation*}
and
\begin{equation*}
|C_1^3-C_2^3|\leq\frac{C}{\det{A}}|\ln\varepsilon|^4\leq C\varepsilon,\quad |C_1^4-C_2^4|\leq\frac{C\,\mathrm{cof}(A)_{44}}{\det{A}}\leq C.
\end{equation*}
Proposition \ref{lemCialpha3D} is proved.
\end{proof}

\section{Proof of Theorem \ref{mainthm2} and the estimates of Cauchy stress tensor.}\label{sec5}

In this section, we shall prove the lower bounds of $|\nabla{\bf u}|$ on the segment $\overline{P_{1}P_{2}}$ in dimension three, and then show the proof of the estimates of Cauchy stress tensor $\sigma[{\bf u},p]$. 

We first decompose
$${\bf u}^{*}=\sum_{\alpha=1}^{6}C_{*}^{\alpha}{\bf u}^{*\alpha}+{\bf u}_{0}^{*},\quad\mbox{and}~ p^{*}=\sum_{\alpha=1}^{6}C_{*}^{\alpha}p^{*\alpha}+p_{0}^{*},$$
where $({\bf u}^{*}, p^{*})$ verifies \eqref{maineqn touch}, $({\bf u}^{*\alpha},p^{*\alpha})$ satisfies
\begin{equation}\label{def valpha*}
\begin{cases}
\nabla\cdot\sigma[{\bf u}^{*\alpha},p^{*\alpha}]=0,\quad\nabla\cdot {\bf u}^{*\alpha}=0,&\mathrm{in}~\Omega^{0},\\
{\bf u}^{*\alpha}={\boldsymbol\psi}_{\alpha},&\mathrm{on}~\partial{D}_{1}^{0}\cup\partial{D}_{2}^{0},\\
{\bf u}^{*\alpha}=0,&\mathrm{on}~\partial{D},
\end{cases}
\end{equation}
and $({\bf u}_{0}^{*},p_{0}^{*})$ satisfies
\begin{equation}\label{equ_v3*}
\begin{cases}
\nabla\cdot\sigma[{\bf u}_0^{*},p_0^{*}]=0,\quad\nabla\cdot {\bf u}_0^{*}=0,&\mathrm{in}~\Omega^{0},\\
{\bf u}_{0}^{*}=0,&\mathrm{on}~\partial{D}_{1}^{0}\cup\partial{D_{2}^{0}},\\
{\bf u}_{0}^{*}={\boldsymbol\varphi},&\mathrm{on}~\partial{D}.
\end{cases}
\end{equation}
Then we define
\begin{align}\label{blowupfactor1}
\tilde b_{j}^{\beta}:=
\int_{\partial D_j}{\boldsymbol\psi}_\beta\cdot\sigma[{\bf u}_{b},p_{b}]\nu,
\end{align}
where 
\begin{equation*}%\label{def u_b}
{\bf u}_{b}:=\sum_{\alpha=1}^{6}C_{2}^{\alpha}{\bf u}^{\alpha}+{\bf u}_{0},\quad p_{b}:=\sum_{\alpha=1}^{6}C_{2}^{\alpha}p^{\alpha}+p_{0},
\end{equation*}
and
$${\bf u}^{\alpha}:={\bf u}_{1}^{\alpha}+{\bf u}_{2}^{\alpha},\quad p^{\alpha}:=p_{1}^{\alpha}+p_{2}^{\alpha}$$ satisfies
\begin{equation}\label{def valpha}
\begin{cases}
\nabla\cdot\sigma[{\bf u}^{\alpha},p^{\alpha}]=0,\quad\nabla\cdot {\bf u}^{\alpha}=0,&\mathrm{in}~\Omega,\\
{\bf u}^{\alpha}={\boldsymbol\psi}_{\alpha},&\mathrm{on}~\partial{D}_{1}\cup\partial D_{2},\\
{\bf u}^{\alpha}=0,&\mathrm{on}~\partial{D},
\end{cases}
\end{equation}
and $({\bf u}_{0},p_{0})$ satisfies
\begin{equation*}%\label{equ_v300*}
\begin{cases}
\nabla\cdot\sigma[{\bf u}_0,p_0]=0,\quad\nabla\cdot {\bf u}_0=0,&\mathrm{in}~\Omega,\\
{\bf u}_{0}=0,&\mathrm{on}~\partial{D}_{1}\cup\partial{D}_{2},\\
{\bf u}_{0}={\boldsymbol\varphi},&\mathrm{on}~\partial{D}.
\end{cases}
\end{equation*}

\begin{prop}\label{protildeb}
For $\beta=1,\dots,6$, we have 
\begin{equation*}%\label{convtilbe}
|\tilde b_{1}^{\beta}[{\boldsymbol\varphi}]-\tilde b_{1}^{*\beta}[{\boldsymbol\varphi}]|\leq\frac{C}{|\ln\varepsilon|},
\end{equation*}
where $\tilde b_{1}^{\beta}[{\boldsymbol\varphi}]$ and $\tilde b_{1}^{*\beta}[{\boldsymbol\varphi}]$ are defined in \eqref{blowupfactor1} and \eqref{blowupfactor}, respectively. 
\end{prop}

By using \eqref{blowupfactor} and \eqref{blowupfactor1}, we have 
\begin{align}\label{difference bi beta}
\tilde b_{1}^{\beta}[{\boldsymbol\varphi}]-\tilde b_{1}^{*\beta}[{\boldsymbol\varphi}]&=\int_{\partial D_{1}}{\boldsymbol\psi}_\beta\cdot\sigma[{\bf u}_{0},p_{0}]\nu-\int_{\partial D_{1}^{0}}{\boldsymbol\psi}_\beta\cdot\sigma[{\bf u}_{0}^*,p_{0}^*]\nu\nonumber\\
&\quad+\sum_{\alpha=1}^{6}C_{2}^{\alpha}\Bigg(\int_{\partial D_{1}}{\boldsymbol\psi}_\beta\cdot\sigma[{\bf u}^{\alpha},p^{\alpha}]\nu-\int_{\partial D_{1}^{0}}{\boldsymbol\psi}_\beta\cdot\sigma[{\bf u}^{*\alpha},p^{*\alpha}]\nu\Bigg)\nonumber\\
&\quad+\sum_{\alpha=1}^{6}\Big(C_{2}^{\alpha}-C_{*}^{\alpha}\Big)\int_{\partial D_{1}^{0}}\frac{\partial v^{*\alpha}}{\partial \nu}\Big|_{+}\cdot{\boldsymbol\psi}_\beta.
\end{align}
Thus, one can see that  to prove Proposition \ref{protildeb}, it suffices to get the convergence of ${\bf u}^\alpha$, $p^\alpha$, ${\bf u}_0$, $p_0$, and $C_{2}^\alpha$. We shall deal with them one by one in the following subsections.

\subsection{Convergence of ${\bf u}_1^\alpha$ and $p_1^\alpha$} 

For $\alpha=1,\dots,6$, let us define $({\bf u}_{1}^{*\alpha},p_1^{*\alpha})$ satisfying
\begin{equation}\label{equ_ui*alpha}
\begin{cases}
\nabla\cdot\sigma[{\bf u}_{1}^{*\alpha},p_{1}^{*\alpha}]=0,\quad\nabla\cdot {\bf u}_{1}^{*\alpha}=0,&\mathrm{in}~\Omega^0,\\
{\bf u}_{1}^{*\alpha}={\boldsymbol\psi}_{\alpha},&\mathrm{on}~\partial{D}_{1}^0\setminus\{0\},\\
{\bf u}_{1}^{*\alpha}=0,&\mathrm{on}~\partial{{D}_{2}^0}\cup\partial{D}.
\end{cases}
\end{equation}	
For each $\alpha$, we will prove that ${\bf u}_{1}^{\alpha}\rightarrow {\bf u}_{1}^{*\alpha}$ when $\varepsilon\rightarrow 0$, with proper convergence rates.
Define  the auxiliary function $k^{*}(x)$ as the limit of $k(x)$ as $\varepsilon\rightarrow0$. Namely, $k^{*}(x)=\frac{1}{2}$ on $\partial{D}_{1}^{0}$, $k^{*}(x)=-\frac{1}{2}$ on $\partial{D}_{2}^{0}\cup\partial{D}$ and
\begin{equation*}%\label{def u bar*}
k^{*}(x)=\frac{x_{3}}{|x'|^2},\quad\hbox{in}\ \Omega_{2R}^{0},\quad \|k^{*}(x)\|_{C^{2}(\Omega^{0}\setminus\Omega_{R}^{0})}\leq\,C,
\end{equation*}
where $\Omega_r^{0}:=\left\{(x',x_{3})\in \mathbb{R}^{3}~\big|~-\frac{|x'|^2}{2}<x_{3}<\frac{|x'|^2}{2},~|x'|<r\right\}$, $r<R$.
Denote
$$\rho_{3}^{\alpha}(\varepsilon):=\begin{cases}
\varepsilon^{1/2},\quad\alpha=1,2,5,6,\\
\varepsilon^{1/4},\quad\alpha=3,\\
\varepsilon^{2/3},\quad\alpha=4.
\end{cases}$$ 
Then we have the following lemma.
\begin{lemma}\label{lem difference v11}
	Let $({\bf u}_{1}^{\alpha},p_1^{\alpha})$ and $({\bf u}_{1}^{*\alpha},p_1^{*\alpha})$ satisfy  \eqref{equ_v1} and \eqref{equ_ui*alpha}, respectively. Then the following boundary estimates hold
\begin{align}\label{boundaryup}
	|\nabla({\bf u}_{1}^{\alpha}-{\bf u}_{1}^{*\alpha})(x)|_{x\in\partial D}+|(p_{1}^{\alpha}-p_{1}^{*\alpha})(x)|_{x\in\partial D}&\leq
	\,C\rho_{3}^{\alpha}(\varepsilon). 
\end{align}
\end{lemma}

\begin{proof}
	Notice that $({\bf u}_{1}^{\alpha}-{\bf u}_{1}^{*\alpha},p_1^{\alpha}-p_1^{*\alpha})$ satisfies
	\begin{equation*}
	\begin{cases}
	\nabla\cdot\sigma[{\bf u}_{1}^{\alpha}-{\bf u}_{1}^{*\alpha},p_{1}^{\alpha}-p_{1}^{*\alpha}]=0,&\mathrm{in}~V:=D\setminus\overline{D_{1}\cup D_{2}\cup D_{1}^{0}\cup D_{2}^{0}},\\
	\nabla\cdot ({\bf u}_{1}^{\alpha}-{\bf u}_{1}^{*\alpha})=0,&\mathrm{in}~V,\\
	{\bf u}_{1}^{\alpha}-{\bf u}_{1}^{*\alpha}={\boldsymbol\psi}_{\alpha}-{\bf u}_{1}^{*\alpha},&\mathrm{on}~\partial{D}_{1}\setminus D_{1}^{0},\\
	{\bf u}_{1}^{\alpha}-{\bf u}_{1}^{*\alpha}=-{\bf u}_{1}^{*\alpha},&\mathrm{on}~\partial{D}_{2}\setminus D_{2}^{0},\\
	{\bf u}_{1}^{\alpha}-{\bf u}_{1}^{*\alpha}={\bf u}_{1}^{\alpha}-{\boldsymbol\psi}_{\alpha},&\mathrm{on}~\partial{D}_{1}^{0}\setminus(D_{1}\cup\{0\}),\\
	{\bf u}_{1}^{\alpha}-{\bf u}_{1}^{*\alpha}={\bf u}_{1}^{\alpha},&\mathrm{on}~\partial{D}_{2}^{0}\setminus D_{2},\\
	{\bf u}_{1}^{\alpha}-{\bf u}_{1}^{*\alpha}=0,&\mathrm{on}~\partial{D}.
	\end{cases}
	\end{equation*}
We will prove the case that $\alpha=1$  for instance, since the other cases are similar. 

If $x\in\partial{D}_{1}\setminus D_{1}^{0}\subset\Omega^{0}\setminus\Omega_{R}^{0}$, then we notice that the point $(x',x_{3}-\varepsilon/2)\in\Omega^{0}\setminus\Omega_{R}^{0}$. By using mean value theorem, 
\begin{align}\label{partial D11}
|({\bf u}_{1}^{1}-{\bf u}_{1}^{*1})(x',x_{3})|&=|(\boldsymbol\psi_{1}-{\bf u}_{1}^{*1})(x',x_{3})|=|{\bf u}_{1}^{*1}(x',x_{3}-\varepsilon/2)-{\bf u}_{1}^{*1}(x',x_{3})|\leq C\varepsilon.
\end{align}
Let
$$\mathcal{C}_{r}:=\left\{x\in\mathbb R^{3}\big| |x'|<r,~-\frac{\varepsilon}{2}-r^2\leq x_{3}\leq\frac{\varepsilon}{2}+r^2\right\},\quad r<R.$$
If $x\in\partial{D}_{1}^{0}\setminus(D_{1}\cup\mathcal{C}_{\varepsilon^{\theta}})$, where $0<\theta<1$ is some constant to be determined later, then by mean value theorem  and Proposition \ref{propu113D}, we have
	\begin{align}\label{partial D11*}
	|({\bf u}_{1}^{1}-{\bf u}_{1}^{*1})(x',x_{3})|&=|({\bf u}_{1}^{1}-\boldsymbol\psi_{1})(x',x_{3})|=|{\bf u}_{1}^{1}(x',x_{3})-{\bf u}_{1}^{1}(x',x_{3}+\varepsilon/2)|\nonumber\\
	&\leq\frac{C\varepsilon}{\varepsilon+|x'|^{2}}\leq C\varepsilon^{1-2\theta}.
	\end{align}
Similarly, for $x\in\partial{D}_{2}\setminus D_{2}^{0}$, 
\begin{align}\label{partial D21}
|({\bf u}_{1}^{1}-{\bf u}_{1}^{*1})(x',x_{3})|\leq C\varepsilon,
\end{align}
and for $x\in\partial{D}_{2}^{0}\setminus(D_{2}\cup\mathcal{C}_{\varepsilon^{\theta}})$, we have
\begin{align}\label{partial D21*}
|({\bf u}_{1}^{1}-{\bf u}_{1}^{*1})(x',x_{3})|\leq\,C\varepsilon^{1-2\theta}.
\end{align}

Define another auxillary function
\begin{align*}%\label{u11*}
{\bf v}_{1}^{*1}=\boldsymbol\psi_{1}\Big(k^*(x)+\frac{1}{2}\Big)+\boldsymbol\psi_{3}x_{1}\Big((k^*(x))^2-\frac{1}{4}\Big),\quad\hbox{in}\ \Omega_{2R}^*,
	\end{align*}
and $\|{\bf v}_{1}^{*1}\|_{C^{2}(\Omega^0\setminus\Omega_{R}^{0})}\leq\,C$. Recalling \eqref{v1alpha}, we have 
	\begin{equation*}
	\left|\partial_{x_{3}}({\bf v}_{1}^{1}-{\bf v}_{1}^{*1})\right|\leq \frac{C\varepsilon}{|x'|^{2}(\varepsilon+|x'|^{2})}.
	\end{equation*}
	For $x\in\Omega_{R}^{0}$ with $|x'|=\varepsilon^{\theta}$, it follows from Proposition \ref{propu113D} that
	\begin{align}\label{partial x2 u11}
	\left|\partial_{x_{3}}({\bf u}_{1}^{1}-{\bf u}_{1}^{*1})(x',x_{3})\right|
	&=\left|\partial_{x_{3}}({\bf u}_{1}^{1}-{\bf v}_{1}^{1})+\partial_{x_{3}}({\bf v}_{1}^{1}-{\bf v}_{1}^{*1})
	+\partial_{x_{3}}({\bf v}_{1}^{*1}-{\bf u}_{1}^{*1})\right|(x',x_{3})\nonumber\\
	&\leq  C\left(1+\frac{1}{\varepsilon^{4\theta-1}}\right).
	\end{align}
	Thus, for $x\in\Omega_{R}^{0}$ with $|x'|=\varepsilon^{\theta}$, by using the triangle inequality, \eqref{partial D21*}, the mean value theorem, and \eqref{partial x2 u11}, 
\begin{align}\label{es v11*2D}
\left|({\bf u}_{1}^{1}-{\bf u}_{1}^{*1})(x',x_{3})\right|&\leq\left|({\bf u}_{1}^{1}-{\bf u}_{1}^{*1})(x',x_{3})-({\bf u}_{1}^{1}-{\bf u}_{1}^{*1})(x',-h_{2}(x')\right|
+C\varepsilon^{1-2\theta}\nonumber\\
&\leq\left|\partial_{x_{3}}({\bf u}_{1}^{1}-{\bf u}_{1}^{*1})\right|\cdot(h_{1}(x')+h_{2}(x'))\Big|_{|x'|=\varepsilon^{\theta}}+C\varepsilon^{1-2\theta}\nonumber\\
&\leq C\left(1+\frac{1}{\varepsilon^{4\theta-1}}\right)\varepsilon^{2\theta}+C\varepsilon^{1-2\theta}=C\left(\varepsilon^{2\theta}+\varepsilon^{1-2\theta}\right).
\end{align}	
By taking $2\theta=1-2\theta$, we get $\theta=\frac{1}{4}$. Substituting it into \eqref{partial D11*}, \eqref{partial D21*} and \eqref{es v11*2D}, and using \eqref{partial D11}, \eqref{partial D21}, and ${\bf u}_{1}^{1}-{\bf u}_{1}^{*1}=0$ on $\partial D$, we obtain
	\begin{align*}%\label{difference v11 boundary}
	|{\bf u}_{1}^{1}-{\bf u}_{1}^{*1}|\leq C\varepsilon^{1/2},\quad\mbox{on}~\partial{(V\setminus\mathcal{C}_{\varepsilon^{1/4}})}.
	\end{align*}
	Applying the maximum modulus for Stokes systems in $V\setminus\mathcal{C}_{\varepsilon^{1/4}}$ (see, for example, the book of Ladyzhenskaya \cite{Lady}), we obtain
	\begin{align*}
	|({\bf u}_{1}^{1}-{\bf u}_{1}^{*1})(x)|\leq C\varepsilon^{1/2},\quad\mbox{in}~ V\setminus \mathcal{C}_{\varepsilon^{1/4}}.
	\end{align*}
In view of the  standard boundary  estimates for Stokes systems with ${\bf u}_{1}^{\alpha}-{\bf u}_{1}^{*\alpha}=0$ on $\partial{D}$ (see, for example \cite{Kratz,Mazya}), we obtain \eqref{boundaryup} with $\alpha=1$. The proof is finished.
\end{proof}

\subsection{Convergence of $\frac{C_{1}^{\alpha}+C_{2}^{\alpha}}{2}-C_{*}^{\alpha}$}
Recalling \eqref{systemC}  and using a rearrangement, from the first line of \eqref{systemC}, we have
\begin{equation}\label{C1C2_1}
\sum_{\alpha=1}^{6}(C_{1}^{\alpha}+C_{2}^{\alpha})(a_{11}^{\alpha\beta}+a_{21}^{\alpha\beta})+\sum_{\alpha=1}^{6}(C_{1}^{\alpha}-C_{2}^{\alpha})(a_{11}^{\alpha\beta}-a_{21}^{\alpha\beta})-2b_{1}^{\beta}=0.
\end{equation}
Similarly, for the second equation of \eqref{systemC}, we have 
\begin{equation}\label{C1C2_22}
\sum_{\alpha=1}^{6}(C_{1}^{\alpha}+C_{2}^{\alpha})(a_{12}^{\alpha\beta}+a_{22}^{\alpha\beta})+\sum_{\alpha=1}^{6}(C_{1}^{\alpha}-C_{2}^{\alpha})(a_{12}^{\alpha\beta}-a_{22}^{\alpha\beta})-2b_{2}^{\beta}=0.
\end{equation}
Then adding \eqref{C1C2_1} and \eqref{C1C2_22} together, and dividing by two, we obtain
\begin{equation}\label{C1+C2_1}
\sum_{\alpha=1}^{6}\frac{C_{1}^{\alpha}+C_{2}^{\alpha}}{2}a^{\alpha\beta}
+\sum_{\alpha=1}^{6}\frac{C_{1}^{\alpha}-C_{2}^{\alpha}}{2}(a_{11}^{\alpha\beta}-a_{22}^{\alpha\beta}+a_{12}^{\alpha\beta}-a_{21}^{\alpha\beta})-(b_{1}^{\beta}+b_{2}^{\beta})=0,
\end{equation}
where 
\begin{equation*}%\label{def_a}
a^{\alpha\beta}=\sum_{i,j=1}^{2}a_{ij}^{\alpha\beta}=-\left(\int_{\partial{D}_{1}}{\boldsymbol\psi}_\beta\cdot\sigma[{\bf u}^\alpha,p^{\alpha}]\nu+\int_{\partial{D}_{2}}{\boldsymbol\psi}_\beta\cdot\sigma[{\bf u}^\alpha,p^{\alpha}]\nu\right),
\end{equation*}
and ${\bf u}^{\alpha}:={\bf u}_{1}^{\alpha}+{\bf u}_{2}^{\alpha}$, $p^{\alpha}:=p_{1}^{\alpha}+p_{2}^{\alpha}$ satisfy
\eqref{def valpha}.

From the third line of \eqref{maineqn touch}, it follows that
\begin{align}\label{equ_C*alpha}
&\sum_{\alpha=1}^{6}C_{*}^{\alpha}a_{*}^{\alpha\beta}-(b_{1}^{*\beta}+b_{2}^{*\beta})=0,\quad\beta=1,\dots,6,
\end{align}
where
\begin{align}\label{b*jbeta}
a_{*}^{\alpha\beta}=-\int_{\partial{D}_{1}^{0}\cup\partial{D}_{2}^{0}}{\boldsymbol\psi}_\beta\cdot\sigma[{\bf u}^{*\alpha},p^{*\alpha}]\nu,\quad b_{j}^{*\beta}=\int_{\partial{D}_{j}^{0}}{\boldsymbol\psi}_\beta\cdot\sigma[{\bf u}_0^*,p_0^{*}]\nu,\quad j=1,2.
\end{align}

\begin{lemma}\label{es b1 b1* beta=1}
Let $b_{1}^{\beta}$ and $b_{1}^{*\beta}$ be defined in \eqref{aijbj} and \eqref{b*jbeta}, respectively. Then we have 	
\begin{equation}\label{est v0-v0*}
\left|b_{1}^{\beta}[{\boldsymbol\varphi}]-b_{1}^{*\beta}[{\boldsymbol\varphi}]\right|\leq C\,\rho_{3}^{\beta}(\varepsilon)\|{\boldsymbol\varphi}\|_{L^{\infty}(\partial D)},
\end{equation}
where $b_{1}^{\beta}[{\boldsymbol\varphi}]$ is defined by \eqref{aijbj}, and $b_{1}^{*\beta}[{\boldsymbol\varphi}]$  by \eqref{b*jbeta}.
\end{lemma}

\begin{proof}
Applying  the integration by parts,  \eqref{defaij}, \eqref{equ_v1}, and \eqref{equ_v3}, we have 
\begin{align*}
b_{1}^{\beta}[{\boldsymbol\varphi}]=\int_{\partial D_{1}}{\boldsymbol\psi}_\beta\cdot\sigma[{\bf u}_0,p_0]\nu&=-\int_{\partial D}{\boldsymbol\varphi}\cdot\sigma[{\bf u}_1^\beta,p_1^{\beta}]\nu.
\end{align*}
Similarly, by means of \eqref{def valpha*}, \eqref{equ_v3*}, and \eqref{b*jbeta}, we obtain
\begin{align*}
b_{1}^{*\beta}[{\boldsymbol\varphi}]=-\int_{\partial D}{\boldsymbol\varphi}\cdot\sigma[{\bf u}_1^{*\beta},p_1^{*\beta}]\nu,
\end{align*}
and thus, 
\begin{equation*}%\label{difference v1}
b_{1}^{\beta}[{\boldsymbol\varphi}]-b_{1}^{*\beta}[{\boldsymbol\varphi}]=-\int_{\partial D}{\boldsymbol\varphi}\cdot\sigma[{\bf u}_1^\beta-{\bf u}_1^{*\beta},p_1^{\beta}-p_1^{*\beta}]\nu.
\end{equation*}
Together with Lemma \ref{lem difference v11}, we get \eqref{est v0-v0*}.
\end{proof}

\begin{lemma}\label{lem valpha1}
Let ${\bf u}^{\alpha}$ and ${\bf u}^{*\alpha}$ be defined by \eqref{def valpha} and \eqref{def valpha*}, respectively, $\alpha=1,\dots,6$. Then 
\begin{align}\label{est v11 v11*}
\left|\int_{\partial D_{1}}{\boldsymbol\psi}_\beta\cdot\sigma[{\bf u}^{\alpha},p^{\alpha}]\nu-\int_{\partial D_{1}^{0}}{\boldsymbol\psi}_\beta\cdot\sigma[{\bf u}^{*\alpha},p^{*\alpha}]\nu\right|\leq
C\,\rho_{3}^{\beta}(\varepsilon)\|{\boldsymbol\psi}_{\alpha}\|_{L^{\infty}(\partial D)}.
\end{align}
\end{lemma}

\begin{proof}
For $\alpha=1,\dots,6$, it follows from \eqref{def valpha} that
\begin{equation*}
\begin{cases}
\nabla\cdot\sigma[{\bf u}^{\alpha}-{\boldsymbol\psi}_{\alpha},p^{\alpha}]=0,\quad\nabla\cdot ({\bf u}^{\alpha}-{\boldsymbol\psi}_{\alpha})=0,&\mbox{in}~\Omega,\\
{\bf u}^{\alpha}-{\boldsymbol\psi}_{\alpha}=0,&\mbox{on}~\partial{D}_{1}\cup\partial{D}_{2},\\
{\bf u}^{\alpha}-{\boldsymbol\psi}_{\alpha}=-{\boldsymbol\psi}_{\alpha},&\mbox{on}~\partial D.
\end{cases}
\end{equation*}
Using the integration by parts, for $\alpha=1,\dots, 6$,
\begin{align*}
\int_{\partial D_{1}}{\boldsymbol\psi}_\beta\cdot\sigma[{\bf u}^{\alpha},p^{\alpha}]\nu=\int_{\partial D_{1}}{\boldsymbol\psi}_\beta\cdot\sigma[{\bf u}^{\alpha}-{\boldsymbol\psi}_{\alpha},p^{\alpha}]\nu=\int_{\partial D}{\boldsymbol\psi}_\alpha\cdot\sigma[{\bf u}_1^{\beta},p_1^{\beta}]\nu.
\end{align*}
Similarly,
\begin{align*}
\int_{\partial D_{1}^{0}}{\boldsymbol\psi}_\beta\cdot\sigma[{\bf u}^{*\alpha},p^{*\alpha}]\nu=\int_{\partial D}{\boldsymbol\psi}_\alpha\cdot\sigma[{\bf u}_1^{*\beta},p_1^{*\beta}]\nu.
\end{align*}
Applying Lemma \ref{lem difference v11}, we obtain  \eqref{est v11 v11*}. 
\end{proof}

\begin{prop}\label{propC2D}
Assume that $D_{1}\cup D_{2}$ and $D$ satisfies $({\rm S_{H}})$, and ${\boldsymbol\varphi}$ satisfies the symmetric condition $({\rm S_{{\boldsymbol\varphi}}})$. Let $C_{1}^{\alpha}, C_{2}^{\alpha}$, and $C_{*}^{\alpha}$ be defined in \eqref{udecom} and \eqref{maineqn touch}, respectively. Then
\begin{equation*}%\label{conv C C*}
C_1^\alpha+C_2^\alpha=C_*^\alpha\equiv0,~\alpha=1,2,3,\quad\mbox{and}~\left|\frac{C_1^\alpha+C_2^\alpha}{2}-C_*^\alpha\right|\leq\frac{C}{|\ln\varepsilon|},~\alpha=4,5,6.
\end{equation*}	
\end{prop}

\begin{proof} 
Firstly, by using the symmetry of the domain with respect to the origin, 
\begin{align}
a_{11}^{\alpha\beta}&=a_{22}^{\alpha\beta},\quad a_{12}^{\alpha\beta}=a_{21}^{\alpha\beta},\quad~\alpha,\beta=1,2,3,\label{origin_sym0}\\
a_{11}^{\alpha\beta}&=-a_{22}^{\alpha \beta},\quad a_{12}^{\alpha\beta}=-a_{21}^{\alpha\beta},\quad\alpha=1,2,3,~\beta=4,5,6,\label{origin_sym}
\end{align}
and 
\begin{align}\label{origin_sym1}
a_{11}^{\alpha\beta}=a_{22}^{\alpha \beta},\quad a_{12}^{\alpha\beta}=a_{21}^{\alpha\beta},\quad \alpha,\beta=4,5,6.
\end{align}

Secondly, due to the symmetry of the domain with respect to $\{x_{3}=0\}$, set
\begin{align*}
{\bf u}_{2}^{\alpha}(x',x_{3})&=\big(
({\bf u}_{2}^{\alpha})^{(1)}(x',x_{3}),
({\bf u}_{2}^{\alpha})^{(2)}(x',x_{3}),
({\bf u}_{2}^{\alpha})^{(3)}(x',x_{3})\big)^{\mathrm T}\\
&=\big(
({\bf u}_{1}^{\alpha})^{(1)}(x',-x_{3}),
({\bf u}_{1}^{\alpha})^{(2)}(x',-x_{3}),
-({\bf u}_{1}^{\alpha})^{(3)}(x',-x_{3})\big)^{\mathrm T},\quad\alpha=1,2,4,
\end{align*}
and
\begin{align*}
{\bf u}_{2}^{\alpha}(x',x_{3})&=\big(
({\bf u}_{2}^{\alpha})^{1}(x',x_{3}),
({\bf u}_{2}^{\alpha})^{(2)}(x',x_{3}),
({\bf u}_{2}^{\alpha})^{(3)}(x',x_{3})\big)^{\mathrm T}\\
&=\big(
-({\bf u}_{1}^{\alpha})^{(1)}(x',-x_{3}),
-({\bf u}_{2}^{\alpha})^{(2)}(x',-x_{3}),
({\bf u}_{1}^{\alpha})^{(3)}(x',-x_{3})\big)^{\mathrm T},\quad\alpha=3,5,6.
\end{align*}
Since
\begin{align*}
&\left(2\mu e({\bf u}_{2}^{1}), e({\bf u}_{2}^{2})\right)\\
=&\,\mu\Big(2\partial_{x_1}({\bf u}_{2}^{1})^{(1)}\partial_{x_1}({\bf u}_{2}^{2})^{(1)}+2\partial_{x_2}({\bf u}_{2}^{1})^{(2)}\partial_{x_2}({\bf u}_{2}^{2})^{(2)}+2\partial_{x_3}({\bf u}_{2}^{1})^{(3)}\partial_{x_3}({\bf u}_{2}^{2})^{(3)}\\
&\quad+\big(\partial_{x_1}({\bf u}_{2}^{1})^{(2)}+\partial_{x_2}({\bf u}_{2}^{1})^{(1)}\big)\cdot\big(\partial_{x_1}({\bf u}_{2}^{2})^{(2)}+\partial_{x_2}({\bf u}_{2}^{2})^{(1)}\big)\\
&\quad+\big(\partial_{x_1}({\bf u}_{2}^{1})^{(3)}+\partial_{x_3}({\bf u}_{2}^{1})^{(1)}\big)\cdot\big(\partial_{x_1}({\bf u}_{2}^{2})^{(3)}+\partial_{x_3}({\bf u}_{2}^{2})^{(1)}\big)\\
&\quad+\big(\partial_{x_2}({\bf u}_{2}^{1})^{(3)}+\partial_{x_3}({\bf u}_{2}^{1})^{(2)}\big)\cdot\big(\partial_{x_2}({\bf u}_{2}^{2})^{(3)}+\partial_{x_3}({\bf u}_{2}^{2})^{(2)}\big)\Big)=\left(2\mu e({\bf u}_{1}^{1}), e({\bf u}_{1}^{2})\right),
\end{align*}
we  have
\begin{equation}\label{12_sym}
a_{22}^{12}=a_{11}^{12},\quad a_{12}^{12}=a_{21}^{12}.
\end{equation}
Similarly, we obtain
\begin{equation}\label{124_sym}
a_{22}^{\alpha\beta}=a_{11}^{\alpha\beta},\quad a_{12}^{\alpha\beta}=a_{21}^{\alpha\beta},\quad \alpha,\beta=1,2,4,
\end{equation}
\begin{equation}\label{356_sym}
a_{22}^{\alpha\beta}=a_{11}^{\alpha\beta},\quad a_{12}^{\alpha\beta}=a_{21}^{\alpha\beta},\quad\alpha,\beta=3,5,6,
\end{equation}
and
\begin{equation}\label{124356_sym}
a_{22}^{\alpha\beta}=-a_{11}^{\alpha\beta},\quad a_{12}^{\alpha\beta}=-a_{21}^{\alpha\beta},\quad \alpha=1,2,4,~\beta=3,5,6.
\end{equation}

Further, using the symmetry of the domain with respect to $x_{1}$-axis, we set for $\alpha=2,3,4,5$,
\begin{align*}
{\bf u}_{2}^{\alpha}(x',x_{3})&=\big(
({\bf u}_{2}^{\alpha})^{(1)}(x',x_{3}),
({\bf u}_{2}^{\alpha})^{(2)}(x',x_{3}),
({\bf u}_{2}^{\alpha})^{(3)}(x',x_{3})\big)^{\mathrm T}\\
&=\big(
-({\bf u}_{1}^{\alpha})^{(1)}(x_1,-x_2,-x_{3}),
({\bf u}_{1}^{\alpha})^{(2)}(x_1,-x_2,-x_{3}),
({\bf u}_{1}^{\alpha})^{(3)}(x_1,-x_2,-x_{3})\big)^{\mathrm T},
\end{align*}
and for $\alpha=1,6$,
\begin{align*}
{\bf u}_{2}^{\alpha}(x',x_{3})&=\big(
({\bf u}_{2}^{\alpha})^{(1)}(x',x_{3}),
({\bf u}_{2}^{\alpha})^{(2)}(x',x_{3}),
({\bf u}_{2}^{\alpha})^{(3)}(x',x_{3})\big)^{\mathrm T}\\
&=\big(
({\bf u}_{1}^{\alpha})^{(1)}(x_1,-x_2,-x_{3}),
-({\bf u}_{1}^{\alpha})^{(2)}(x_1,-x_2,-x_{3}),
-({\bf u}_{1}^{\alpha})^{(3)}(x_1,-x_2,-x_{3})\big)^{\mathrm T}.
\end{align*}
Then
\begin{equation}\label{2345_sym}
a_{22}^{\alpha\beta}=a_{11}^{\alpha\beta},\quad a_{12}^{\alpha\beta}=a_{21}^{\alpha\beta},\quad \alpha,\beta=2,3,4,5,\quad a_{22}^{16}=a_{11}^{16},\quad a_{12}^{16}=a_{21}^{16},
\end{equation}
and
\begin{equation}\label{16_sym}
a_{22}^{\alpha\beta}=-a_{11}^{\alpha\beta},\quad a_{12}^{\alpha\beta}=-a_{21}^{\alpha\beta},\quad\alpha=1,6,~\beta=2,3,4,5.
\end{equation}

Besides, thanks to the symmetry of the domain with respect to $x_{2}$-axis, set for $\alpha=1,3,4,6$,
\begin{align*}
{\bf u}_{2}^{\alpha}(x',x_{3})&=\big(
({\bf u}_{2}^{\alpha})^{(1)}(x',x_{3}),
({\bf u}_{2}^{\alpha})^{(2)}(x',x_{3}),
({\bf u}_{2}^{\alpha})^{(3)}(x',x_{3})\big)^{\mathrm T}\\
&=\big(
({\bf u}_{1}^{\alpha})^{(1)}(-x_1,x_2,-x_{3}),
-({\bf u}_{1}^{\alpha})^{(2)}(-x_1,x_2,-x_{3}),
({\bf u}_{1}^{\alpha})^{(3)}(-x_1,x_2,-x_{3})\big)^{\mathrm T},
\end{align*}
and for $\alpha=2,5$,
\begin{align*}
{\bf u}_{2}^{\alpha}(x',x_{3})&=\big(
({\bf u}_{2}^{\alpha})^{(1)}(x',x_{3}),
({\bf u}_{2}^{\alpha})^{(2)}(x',x_{3}),
({\bf u}_{2}^{\alpha})^{(3)}(x',x_{3})\big)^{\mathrm T}\\
&=\big(
-({\bf u}_{1}^{\alpha})^{(1)}(-x_1,x_2,-x_{3}),
({\bf u}_{1}^{\alpha})^{(2)}(-x_1,x_2,-x_{3}),
-({\bf u}_{1}^{\alpha})^{(3)}(-x_1,x_2,-x_{3})\big)^{\mathrm T}.
\end{align*}
Then 
\begin{equation}\label{1346_sym}
a_{22}^{\alpha\beta}=a_{11}^{\alpha\beta},\quad a_{12}^{\alpha\beta}=a_{21}^{\alpha\beta},\quad \alpha,\beta=1,3,4,6,,\quad a_{22}^{25}=a_{11}^{25},\quad a_{12}^{25}=a_{21}^{25},
\end{equation}
and
\begin{equation}\label{25_sym}
a_{22}^{\alpha\beta}=-a_{11}^{\alpha\beta},\quad a_{12}^{\alpha\beta}=-a_{21}^{\alpha\beta},\quad\alpha=2,5,~\beta=1,3,4,6.
\end{equation}
	
In summary, combining \eqref{origin_sym0}--\eqref{25_sym}, we obtain
\begin{align}\label{syma11}
\begin{split}
&a_{11}^{\alpha\beta}=a_{22}^{\alpha\beta}=0,\quad a_{12}^{\alpha\beta}=a_{21}^{\alpha\beta}=0,\quad\alpha=1,~\beta=2,3,4,6,\\
&a_{11}^{\alpha\beta}=a_{22}^{\alpha\beta}=0,\quad a_{12}^{\alpha\beta}=a_{21}^{\alpha\beta}=0,\quad\alpha=2,~\beta=3,4,5,\\
&a_{11}^{\alpha\beta}=a_{22}^{\alpha\beta}=0,\quad a_{12}^{\alpha\beta}=a_{21}^{\alpha\beta}=0,\quad\alpha=3,~\beta=4,5,6,\\
&a_{11}^{\alpha\beta}=a_{22}^{\alpha\beta}=0,\quad a_{12}^{\alpha\beta}=a_{21}^{\alpha\beta}=0,\quad\alpha=4,~\beta=5,6,\\
&a_{11}^{56}=a_{22}^{56}=0,\quad a_{12}^{56}=a_{21}^{56}=0.
\end{split}
\end{align}
Thus,
\begin{equation}\label{symaalphabeta}
a^{\alpha\beta}=0,\quad\alpha,\beta=1,2,3,4,5,6,~\alpha\neq\beta.
\end{equation}
Similarly, 
\begin{equation}\label{symaalphabeta12}
a_*^{\alpha\beta}=0,\quad\alpha,\beta=1,2,3,4,5,6,~\alpha\neq\beta.
\end{equation}

Finally, by the symmetry $({\rm S_{{\boldsymbol\varphi}}})$, setting 
\begin{align*}
{\bf u}_0(x)=-{\bf u}_0(-x),
\end{align*}
it is easy to see that
\begin{equation}\label{|symtilb}
b_{1}^{\alpha}=-b_{2}^{\alpha},\quad \alpha=1,2,3,\quad b_{1}^{\alpha}=b_{2}^{\alpha},\quad\alpha=4,5,6.
\end{equation}
Similarly,
\begin{equation}\label{|symtilb00}
b_{1}^{*\alpha}=-b_{2}^{*\alpha},\quad \alpha=1,2,3,\quad b_{1}^{*\alpha}=b_{2}^{*\alpha},\quad\alpha=4,5,6.
\end{equation}

Now coming back to \eqref{systemC} and using the Cramer's rule,  \eqref{syma11} and \eqref{|symtilb}, we deduce
\begin{equation}\label{C1456}
C_1^\alpha=-C_2^\alpha,\quad \alpha=1,2,3,\quad C_1^\alpha=C_2^\alpha,\quad\alpha=4,5,6.
\end{equation}
By virtue of \eqref{symaalphabeta}, the equation \eqref{C1+C2_1} becomes
\begin{align}\label{C1+C2_matrix}
\mathcal{A}\mathcal{C}=\mathcal{B},
\end{align}
where 
$$\mathcal{A}=\mathrm{diag}\Big(a^{11},a^{22},a^{33},a^{44},a^{55},a^{66}\Big),\quad\mathcal{C}=\Big(0,0,0,\frac{C_1^4+C_2^4}{2},\frac{C_1^5+C_2^5}{2},\frac{C_1^6+C_2^6}{2}\Big)^{\mathrm T},$$
and
$$\mathcal{B}=\Big(0,0,0,2b_{1}^{4},2b_{1}^{5}-(C_1^1-C_2^1)(a_{11}^{15}+a_{12}^{15}),2b_{1}^{6}-(C_1^2-C_2^2)(a_{11}^{26}+a_{12}^{26})\Big)^{\mathrm T}.$$	

Similarly, substituting \eqref{symaalphabeta12} and  \eqref{|symtilb00} into \eqref{equ_C*alpha}, we obtain
\begin{align}\label{C1+C2_*}
\mathcal{A}^*\mathcal{C}^*=\mathcal{B}^*,
\end{align}
where
\begin{align*}
\mathcal{A}^*=\mathrm{diag}\Big(a_{*}^{11},a_{*}^{22},a_{*}^{33},a_{*}^{44},a_{*}^{55},a_{*}^{66}\Big),
\end{align*}
and
$$\mathcal{C}^*=\Big(0,0,0,C_{*}^{4},C_{*}^{5},C_{*}^{6})\Big)^{\mathrm T},\quad\,\mathcal{B}^*=\Big(0,0,0,2b_{1}^{*4},2b_{1}^{*5},2b_{1}^{*6})\Big)^{\mathrm T}.$$
Combining \eqref{C1+C2_matrix} and \eqref{C1+C2_*}, we have
\begin{equation*}
\mathcal{A}(\mathcal{C}-\mathcal{C}^*)=\mathcal{B}-\mathcal{B}^*-\mathcal{C}^*(\mathcal{A}-\mathcal{A}^*)=:\mathcal{R}.
\end{equation*}
From  Proposition  \ref{lemCialpha3D}, Lemma \ref{es b1 b1* beta=1} and Lemma \ref{lem valpha1}, one can find that
\begin{equation*}
\mathcal{R}^\alpha=O\left(\frac{1}{|\ln\varepsilon|}\right),\quad \alpha=1,\dots,6.
\end{equation*}
Then by Cramer's rule, 
\begin{equation*}
\left|\frac{C_1^\alpha+C_2^\alpha}{2}-C_*^\alpha\right|\leq\frac{C}{|\ln\varepsilon|}.
\end{equation*}
The proof of Proposition \ref{propC2D} is completed.
\end{proof}

\begin{proof}[Proof of Proposition \ref{protildeb}]
Recalling \eqref{difference bi beta} and by  Propositions \ref{lemCialpha3D}  and  \ref{propC2D}, we have 
\begin{align*}
|C_{2}^\alpha-C_*^\alpha|&=\left|\frac{C_1^\alpha+C_2^\alpha}{2}-C_*^\alpha-\frac{C_1^\alpha-C_2^\alpha}{2}\right|=\frac{|C_1^\alpha-C_2^\alpha|}{2}\leq C
\begin{cases}
\frac{1}{|\ln\varepsilon|},\quad\alpha=1,2,\\
\varepsilon,\quad\alpha=3,
\end{cases}
\end{align*}
and 
\begin{align*}
|C_{2}^\alpha-C_*^\alpha|&=\left|\frac{C_1^\alpha+C_2^\alpha}{2}-C_*^\alpha-\frac{C_1^\alpha-C_2^\alpha}{2}\right|=\left|\frac{C_1^\alpha+C_2^\alpha}{2}-C_*^\alpha\right|\leq
\frac{C}{|\ln\varepsilon|},~\alpha=4,5,6.
\end{align*}
This, together with Lemma \ref{es b1 b1* beta=1} and Lemma \ref{lem valpha1}, implies that
	\begin{equation*}
	|\tilde b_{1}^{\beta}[{\boldsymbol\varphi}]-\tilde b_{1}^{*\beta}[{\boldsymbol\varphi}]|\leq\frac{C}{|\ln\varepsilon|}.
	\end{equation*}
	The proof of Proposition \ref{protildeb} is finished.
\end{proof}

\subsection{The Completion of the Proof of Theorem \ref{mainthm2}}
\begin{proof}[Proof of Theorem \ref{mainthm2}]
By Propositions \ref{lemCialpha3D} and \ref{propu133D}, we obtain
\begin{align*}
|(C_1^3-C_2^3)\nabla{\bf u}_1^3(0',x_3)|\leq C.
\end{align*}
Then 
\begin{align}\label{nablau_dec}
|\nabla{{\bf u}}(0',x_3)|&=\left|\sum_{\alpha=1}^{3}\left(C_{1}^{\alpha}-C_{2}^{\alpha}\right)\nabla{\bf u}_{1}^{\alpha}(0',x_3)
+\nabla {\bf u}_{b}(0',x_3)\right|\nonumber\\
&\geq\left|\sum_{\alpha=1}^{2}\left(C_{1}^{\alpha}-C_{2}^{\alpha}\right)\nabla{\bf u}_{1}^{\alpha}(0',x_3)\right|-C\nonumber\\
&\geq\left|\sum_{\alpha=1}^{2}\left(C_{1}^{\alpha}-C_{2}^{\alpha}\right)\nabla{\bf v}_{1}^{\alpha}(0',x_3)\right|-C\nonumber\\
&\geq\left|\sum_{\alpha=1}^{2}\left(C_{1}^{\alpha}-C_{2}^{\alpha}\right)\partial_{x_3}{\bf v}_{1}^{\alpha}(0',x_3)\right|-C.
\end{align}
For $x=(0',x_3)\in\Omega_R$,  from the definition of ${\bf v}_1^\alpha$ in \eqref{v1alpha} and \eqref{estv113D}, we have 
\begin{equation}\label{lowerv11}
\partial_{x_3}{\bf v}_1^1(0',x_3)=\left(\varepsilon^{-1},0,0\right)^{\mathrm{T}}, 
\partial_{x_3}{\bf v}_1^2(0',x_3)=\left(0,\varepsilon^{-1},0\right)^{\mathrm{T}}.
\end{equation}
Thus, it suffices to prove the lower bounds of $C_1^\alpha-C_2^\alpha$, $\alpha=1,2$.

It follows from \eqref{equ-decompositon} and Proposition \ref{propC2D} that
\begin{equation*}%\label{eqC3D}
\sum_{\alpha=1}^{3}\left(C_{1}^{\alpha}-C_{2}^{\alpha}\right)a_{11}^{\alpha\beta}=\tilde b_1^\beta,\quad\beta=1,2,3.
\end{equation*}
By virtue of Lemma \ref{lema113D} and Lemma \ref{lema114563D}, and \eqref{syma11}, we deduce
\begin{equation*}
\frac{1}{C\varepsilon}|\ln\varepsilon|^2\leq\det\bar{A}\leq\frac{C}{\varepsilon}|\ln\varepsilon|^2,
\end{equation*}
where $\bar{A}=(a_{11}^{\alpha\beta})_{\alpha,\beta=1}^3$. Denote the cofactor of $a_{11}^{\alpha\beta}$ by $\mathrm{cof}(\bar{A})_{\alpha\beta}$. Then
\begin{equation*}
\frac{1}{C\varepsilon}|\ln\varepsilon|\leq \mathrm{cof}(\bar{A})_{\alpha\alpha}\leq\frac{C}{\varepsilon}|\ln\varepsilon|,\quad\alpha=1,2,\quad 	\frac{1}{C}|\ln\varepsilon|^2\leq \mathrm{cof}(\bar{A})_{33}\leq C|\ln\varepsilon|^2,
\end{equation*}
\begin{equation*} |\mathrm{cof}(\bar{A})_{\alpha\beta}|\leq\frac{C}{\varepsilon},\quad\alpha,\beta=1,2,~\alpha\neq\beta,
\end{equation*}
and
\begin{equation*} 
|\mathrm{cof}(\bar{A})_{3\alpha}|,|\mathrm{cof}(\bar{A})_{\alpha 3}|\leq C|\ln\varepsilon|,\quad\alpha=1,2.
\end{equation*}
Using Cramer's rule, if $\tilde b_1^\alpha[{\boldsymbol\varphi}]\neq0$, $\alpha=1,2$, then 
\begin{align}\label{diffC}
|C_1^\alpha-C_2^\alpha|=\left|\frac{1}{\det \bar{A}}\tilde b_1^\alpha[{\boldsymbol\varphi}] \mathrm{cof}(A)_{\alpha\alpha}+O\left(\frac{1}{|\ln\varepsilon|^2}\right)\right|\geq\frac{\Big|\tilde b_1^{\alpha}[{\boldsymbol\varphi}]\Big|}{C|\ln\varepsilon|},\quad\alpha=1,2.
\end{align}
Applying Proposition \ref{protildeb}, if $\tilde b_{1}^{*\alpha_0}[{\boldsymbol\varphi}]\neq0$ for some $\alpha_0\in\{1,2\}$, then there exists  a small enough constant $\varepsilon_0>0$ such that for $0<\varepsilon<\varepsilon_0$, 
\begin{equation*}%\label{tildeb3D}
\Big|\tilde b_{1}^{\alpha_0}[{\boldsymbol\varphi}]\Big|\geq\frac{1}{2}\Big|\tilde b_{1}^{*\alpha_0}[{\boldsymbol\varphi}]\Big|>0.
\end{equation*}
This, together with \eqref{nablau_dec}--\eqref{diffC}, gives
\begin{align*}
|\nabla{{\bf u}}(0',x_3)|\geq\frac{\Big|\tilde b_1^{*\alpha_0}[{\boldsymbol\varphi}]\Big|}{C\varepsilon|\ln\varepsilon|}.
\end{align*}
So the proof of Theorem \ref{mainthm2} is completed.
\end{proof}

\subsection{Cauchy stress estimates}
Finally, we have the following estimates for the Cauchy stress tensor $\sigma[{\bf u},p]$.

\begin{theorem}\label{mainthmsigma}(Cauchy  Stress estimates)
Under the assumptions of Theorem \ref{mainthm2}, let ${\bf u}\in H^1(D;\mathbb R^3)\cap C^1(\bar{\Omega};\mathbb R^3)$ and $p\in L^2(D)\cap C^0(\bar{\Omega})$ be a solution to \eqref{Stokessys} and \eqref{compatibility}. Then we have 
\begin{equation}\label{equ-sigma3D}
|\sigma[{\bf u},p-(q)_{R}]|\leq \frac{C}{\varepsilon}\|{\boldsymbol\varphi}\|_{C^{2,\alpha}(\partial D;\mathbb R^3)}\quad\mbox{in}~\Omega_R,
\end{equation}
where $(q)_{R}:=\sum_{\alpha=1}^{3}(C_{1}^{\alpha}-C_{2}^{\alpha})(q_{1}^\alpha)_{R}$, and $(q_{1}^\alpha)_{R}$ is defined in \eqref{defqialpha}. Moreover, at the segment $\overline{P_{1}P_{2}}$, if $\tilde b_{1}^{*3}[{\boldsymbol\varphi}]\neq0$, then
$$|\sigma[{\bf u},p-(q)_{R}]|(0',x_{3})\geq\frac{1}{C\varepsilon},\quad|x_{3}|\leq\varepsilon.$$ 
\end{theorem}

\begin{remark}
Here we show that the blow up rate of $|\sigma[{\bf u},p-(q)_{R}]|$, $\varepsilon^{-1}$  is optimal in dimension three. 
\end{remark}

\begin{proof}[Proof of Theorem \ref{mainthmsigma}]
From \eqref{C1456}, we have
$$C_1^\alpha=C_2^\alpha,\quad\alpha=4,5,6.$$
Instead of \eqref{upper-u3D} and \eqref{upper-p3D},
\begin{align*}
|\nabla{\bf u}(x)|&\leq\sum_{\alpha=1}^{3}\left|(C_{1}^{\alpha}-C_{2}^{\alpha}\right)\nabla{\bf u}_{1}^{\alpha}(x)|+C\nonumber\\
&\leq\sum_{\alpha=1,2}\left|(C_{1}^{\alpha}-C_{2}^{\alpha}\right)\nabla{\bf u}_{1}^{\alpha}(x)|+\left|(C_{1}^{3}-C_{2}^{3}\right)\nabla{\bf u}_{1}^{3}(x)|+C\nonumber\\
&\leq\frac{C}{|\ln\varepsilon|\delta(x')}+C\varepsilon\left(\frac{1}{\delta(x')}+\frac{|x'|}{\delta^2(x')}\right)+C\leq \frac{C}{|\ln\varepsilon|\delta(x')},
\end{align*}
and
\begin{align*}
|p(x)-(q)_{R}|&\leq\sum_{\alpha=1}^{3}\left|(C_{1}^{\alpha}-C_{2}^{\alpha}\right)(p_{1}^{\alpha}(x)-(q_{1}^\alpha)_{R})|+C\nonumber\\
&\leq\sum_{\alpha=1,2}\left|(C_{1}^{\alpha}-C_{2}^{\alpha}\right)(p_{1}^{\alpha}(x)-(q_{1}^\alpha)_{R})|\nonumber\\
&\quad+\left|(C_{1}^{3}-C_{2}^{3}\right)(p_{1}^{3}(x)-(q_{1}^3)_{R})|+C\nonumber\\
&\leq\frac{C}{|\ln\varepsilon|\varepsilon}+\frac{C}{\varepsilon}+C\leq \frac{C}{\varepsilon},
\end{align*}
which implies \eqref{equ-sigma3D}.

For the lower bound, similarly as \eqref{diffC}, and in view of Proposition \ref{protildeb},
\begin{align*}%\label{diffC}
|C_1^3-C_2^3|=\left|\frac{1}{\det \bar{A}}\tilde b_1^3[{\boldsymbol\varphi}] \mathrm{cof}(A)_{33}+O\left(\frac{\varepsilon}{|\ln\varepsilon|}\right)\right|\geq\frac{\varepsilon}{C}\Big|\tilde b_1^{*3}[{\boldsymbol\varphi}]\Big|.
\end{align*}
By using \eqref{v13upper1} and  Proposition \ref{propu133D}, we have 
\begin{equation*}
|\partial_{x_1}({\bf u}_1^3)^{(1)}(0',x_3)|\leq\frac{C}{\varepsilon}.
\end{equation*}
From \eqref{p133D} and Proposition  \ref{propu133D}, it follows that
\begin{equation*}
|(p_{1}^{3}-(q_1^3)_{R})|(0',x_3)\geq\frac{1}{C\varepsilon^2}.
\end{equation*}
Combining with Propositions \ref{propu113D} and \ref{propu133D}, if $\Big|\tilde b_1^{*3}[{\boldsymbol\varphi}]\Big|\neq0$, then 
\begin{align*}
&\Big|\sigma[{\bf u},p-(q)_{R}]\Big|(0',x_{3})=\Big|2\mu e({\bf u})-(p-(q)_{R})\mathbb{I}\Big|(0',x_{3})\\
\geq&\,\Big|\sum_{\alpha=1}^{3}(C_{1}^{\alpha}-C_{2}^{\alpha})\Big(2\mu e({\bf u}_{1}^{\alpha})-(p_{1}^{\alpha}-(q_1^\alpha)_{R})\mathbb{I}\Big)\Big|(0',x_{3})-C\\
\geq&\,\Big|(C_{1}^{3}-C_{2}^{3})\Big(2\mu \partial_{x_1}({\bf u}_1^3)^{(1)}-(p_{1}^{3}-(q_1^3)_{R})\Big)\Big|(0',x_{3})-\frac{C}{\varepsilon|\ln\varepsilon|}\geq\frac{|\tilde b_{1}^{*3}[{\boldsymbol\varphi}]|}{C\varepsilon}.
\end{align*}
The proof of Corollary \ref{mainthmsigma} is finished.
\end{proof}

\section{Proof of Theorem \ref{mainthmD4}}\label{prfthmD4}
In this section, we prove Theorem \ref{mainthmD4}. Similar to the decomposition \eqref{udecom}, we have 
\begin{align}\label{udecD4}
{\bf u}(x)=\sum_{i=1}^{2}\sum_{\alpha=1}^{\frac{d(d+1)}{2}}C_i^{\alpha}{\bf u}_{i}^{\alpha}(x)+{\bf u}_{0}(x)\quad~ \mbox{and}~
p(x)=\sum_{i=1}^{2}\sum_{\alpha=1}^{\frac{d(d+1)}{2}}C_i^{\alpha}p_{i}^{\alpha}(x)+p_{0}(x),
\end{align}
where ${\bf u}_{i}^{\alpha},{\bf u}_{0}\in{C}^{2}(\Omega;\mathbb R^d),~p_{i}^{\alpha}, p_0\in{C}^{1}(\Omega)$ satisfy \eqref{equ_v1} and \eqref{equ_v3}, respectively. The estimates of $|\nabla{\bf u}_0|$ and $|p_0|$ can be found in Proposition \ref{prop1.7}. To estimate $|\nabla{\bf u}_{i}^{\alpha}|$ and $|p_{i}^{\alpha}|$, we are going to construct auxiliary functions ${\bf v}_{i}^{\alpha}\in C^{2}(\Omega;\mathbb R^d)$ and $\bar p_i^\alpha\in C^1(\Omega)$ in the following. Since the case of $i=2$ is similar by replacing $\frac{x_d}{\delta(x')}$ with $-\frac{x_d}{\delta(x')}$ in $\Omega_{2R}$, we will only consider $i=1$ for instance.  

We seek ${\bf v}_{1}^{\alpha}\in C^{2}(\Omega;\mathbb R^d)$ such that ${\bf v}_{1}^{\alpha}={\bf u}_{1}^{\alpha}={\boldsymbol\psi}_{\alpha}$ on $\partial{D}_{1}$, ${\bf v}_{1}^{\alpha}={\bf u}_{1}^{\alpha}=0$ on $\partial{D}_{2}\cup\partial{D}$, $
\|{\bf v}_{1}^{\alpha}\|_{C^{2}(\Omega\setminus\Omega_{R})}\leq\,C$, and in $\Omega_{2R}$, they are constructed as follows:
\begin{align*}
{\bf v}_{1}^{\alpha}=\boldsymbol\psi_{\alpha}\Big(k(x)+\frac{1}{2}\Big)+\Big(k^2(x)-\frac{1}{4}\Big)
\begin{cases}
\boldsymbol\psi_{d}x_{\alpha},&\alpha=1,\dots,d-1,\\
{\bf F}_{d},&\alpha=d,\\
0,&\alpha=d+1,\dots,\frac{d(d-1)}{2}+1,\\
{\bf F}_{\alpha},&\alpha=\frac{d(d-1)}{2}+2,\dots,\frac{d(d+1)}{2},
\end{cases}
\end{align*}
where  $k(x)=\frac{x_d}{\delta(x')}$, ${\bf F}_{d}=\big(\frac{6x_1}{(d-1)\delta(x')},\dots,\frac{6x_{d-1}}{(d-1)\delta(x')},\frac{12|x'|^2x_d}{(d-1)\delta^2(x')}-2k(x)\big)^{\mathrm T}$, and for $\alpha=\frac{d(d-1)}{2}+2,\dots,\frac{d(d+1)}{2}$, 
\begin{align*}
{\bf F}_{\alpha}&=-\frac{12}{2d-1}\sum_{i=1}^{d-1}\frac{x_ix_{\alpha-\frac{d(d-1)}{2}-1}}{\delta(x')}\boldsymbol\psi_{i}+\Big(\frac{3}{2d-1}-5x_dk(x)\Big)\boldsymbol\psi_{\alpha-\frac{d(d-1)}{2}-1}\\
&\quad-2x_{\alpha-\frac{d(d-1)}{2}-1}k(x)\Big(-\frac{12|x'|^2}{(2d-1)\delta(x')}+\frac{2(d+1)}{2d-1}-3x_dk(x)\Big)\boldsymbol\psi_{d}.
\end{align*}
The corresponding $\bar p_1^\alpha\in C^1(\Omega)$ such that $
\|\bar{p}_1^{\alpha}\|_{C^{1}(\Omega\setminus\Omega_{R})}\leq C$, and in $\Omega_{2R}$,
\begin{equation*}
\bar p_1^\alpha=\begin{cases}
\frac{2\mu x_{\alpha}}{\delta(x')}k(x),&\alpha=1,\dots,d-1,\\
-\frac{3}{(d-1)}\frac{\mu}{\delta^2(x')}+\frac{6\mu}{\delta(x')}\left(\frac{6|x'|^{2}}{(d-1)\delta(x')}-1\right)k^{2}(x),&\alpha=d,\\
0,&\alpha=d+1,\dots,\frac{d(d-1)}{2}+1,\\
\frac{6\mu x_{\alpha-\frac{d(d-1)}{2}-1}}{(2d-1)\delta^2(x')}+\frac{12\mu x_{\alpha-\frac{d(d-1)}{2}-1}}{(2d-1)\delta(x')}\Big(d+1-\frac{6|x'|^2}{\delta(x')}\Big)k^2(x),&\alpha=\frac{d(d-1)}{2}+2,\dots,\frac{d(d+1)}{2}.
\end{cases}
\end{equation*}

Using the argument in the proof of Propositions \ref{propu113D}--\ref{propu4563D}, the following proposition holds. 

\begin{prop}\label{propuD4}
Let ${\bf u}_{i}^{\alpha}\in{C}^{2}(\Omega;\mathbb R^d),~p_{i}^{\alpha}\in{C}^{1}(\Omega)$ be the solution to \eqref{equ_v1}, $\alpha=1,2$. Then it holds that, for $x\in\Omega_{R}$,
\begin{align*}
\|\nabla({\bf u}_{i}^{\alpha}-{\bf v}_{i}^{\alpha})\|_{L^{\infty}(\Omega_{\delta/2}(x'))}\leq
\begin{cases}
C,&\alpha=1,\dots,\frac{d(d+1)}{2},~\alpha\neq d,\\
\frac{C}{\delta^{1/2}(x')},&\alpha=d,
\end{cases}
\end{align*}
and
\begin{align*}
\|\nabla^2({\bf u}_{i}^{\alpha}-{\bf v}_{i}^{\alpha})\|_{L^{\infty}(\Omega_{\delta/2}(x'))}&+\|\nabla q_i^\alpha\|_{L^{\infty}(\Omega_{\delta/2}(x'))}\\
&\leq
\begin{cases}
\frac{C}{\delta(x')},&~\alpha=1,\dots,d-1;\frac{d(d-1)}{2}+2,\dots,\frac{d(d+1)}{2},\\
\frac{C}{\delta^{3/2}(x')},&~\alpha=d,\\
\frac{C}{\delta^{1/2}(x')},&\alpha=d+1,\dots,\frac{d(d-1)}{2}+1.
\end{cases}
\end{align*}
\end{prop}

Consequently, for $x\in\Omega_R$, we have 
\begin{align*}
|\nabla {\bf u}_{i}^{\alpha}(x)|\leq
\begin{cases}
\frac{C}{\delta(x')},&\alpha=1,\dots,d-1;\frac{d(d-1)}{2}+2,\dots,\frac{d(d+1)}{2},\\
C\left(\frac{1}{\delta(x')}+\frac{|x'|}{\delta^2(x')}\right),&\alpha=d,\\
C\left(\frac{|x'|}{\delta(x')}+1\right),&\alpha=d+1,\dots,\frac{d(d-1)}{2}+1,
\end{cases}\\ 
|\nabla^2 {\bf u}_{i}^{\alpha}(x)|\leq
\begin{cases}
C\left(\frac{1}{\delta(x')}+\frac{|x'|}{\delta^2(x')}\right),&\alpha=1,\dots,d-1;\frac{d(d-1)}{2}+2,\dots,\frac{d(d+1)}{2},\\
C\left(\frac{1}{\delta^2(x')}+\frac{|x'|}{\delta^3(x')}\right),&\alpha=d,\\
\frac{C}{\delta^2(x')},&\alpha=d+1,\dots,\frac{d(d-1)}{2}+1,
\end{cases}\\ 
|p_{i}^{\alpha}(x)-(q_{i}^\alpha)_{R}|\leq
\begin{cases}
\frac{C}{\delta(x')},&\alpha=1,\dots,d-1,\\
\frac{C}{\delta^2(x')},&\alpha=d,\\
\frac{C}{\delta^{1/2}(x')},&\alpha=d+1,\dots,\frac{d(d-1)}{2}+1,\\
C\left(\frac{1}{\delta(x')}+\frac{|x'|}{\delta^2(x')}\right),&\alpha=\frac{d(d-1)}{2}+2,\dots,\frac{d(d+1)}{2},
\end{cases}
\end{align*}
and 
\begin{align*}
|\nabla p_{i}^{\alpha}(x)|\leq
\begin{cases}
C\left(\frac{1}{\delta(x')}+\frac{|x'|}{\delta^2(x')}\right),&\alpha=1,\dots,d-1,\\
C\left(\frac{1}{\delta^2(x')}+\frac{|x'|}{\delta^3(x')}\right),&\alpha=d,\\
\frac{C}{\delta^{1/2}(x')},&\alpha=d+1,\dots,\frac{d(d-1)}{2}+1,\\
\frac{C}{\delta^2(x')},&\alpha=\frac{d(d-1)}{2}+2,\dots,\frac{d(d+1)}{2},
\end{cases}
\end{align*}
where $(q_{i}^\alpha)_{R}$ is defined in \eqref{defqialpha}.

\begin{proof}[Proof of Theorem \ref{mainthmD4}]
By the argument in the proof of \cite[(4.1)]{BLL2} and using \cite[Lemma 6.1]{BLL2}, we still have \eqref{bddC} for $d\geq4$. Then in view of \eqref{udecD4}, Propositions \ref{propuD4} and \ref{prop1.7}, we obtain
\begin{align*}
|\nabla {\bf u}(x)|&=\Big|\sum_{i=1}^{2}\sum_{\alpha=1}^{\frac{d(d+1)}{2}}C_i^{\alpha}\nabla{\bf u}_{i}^{\alpha}(x)+\nabla{\bf u}_{0}(x)\Big|\\
&\leq 
C\left(\frac{1}{\delta(x')}+\frac{|x'|}{\delta^2(x')}\right)+\frac{C}{\delta(x')}\leq 
\frac{C|x'|}{\delta^2(x')}+\frac{C}{\delta(x')},
\end{align*}
\begin{align*}
|\nabla^2 {\bf u}(x)|=\Big|\sum_{i=1}^{2}\sum_{\alpha=1}^{\frac{d(d+1)}{2}}C_i^{\alpha}\nabla^2{\bf u}_{i}^{\alpha}(x)+\nabla^2{\bf u}_{0}(x)\Big|\leq 
C\left(\frac{1}{\delta^2(x')}+\frac{|x'|}{\delta^3(x')}\right),
\end{align*}
\begin{align*}
|p(x)-(q)_{R}|\leq \sum_{i=1}^{2}\sum_{\alpha=1}^{\frac{d(d+1)}{2}}|C_i^{\alpha}(p_{i}^{\alpha}(x)-(q_i^\alpha)_{R})|+C\leq\frac{C}{\varepsilon^2},
\end{align*}
and
\begin{align*}
|\nabla p(x)|= \Big|\sum_{i=1}^{2}\sum_{\alpha=1}^{\frac{d(d+1)}{2}}C_i^{\alpha}\nabla p_{i}^{\alpha}(x)+\nabla p_0(x)\Big|\leq C\left(\frac{1}{\delta^2(x')}+\frac{|x'|}{\delta^3(x')}\right),
\end{align*}
where $(q)_{R}:=\sum_{i=1}^{2}\sum_{\alpha=1}^{\frac{d(d+1)}{2}}C_{i}^{\alpha}(q_{i}^\alpha)_{R}$. This finishes the proof of Theorem \ref{mainthmD4}.
\end{proof}

\noindent{\bf{\large Acknowledgements.}}
H.G. Li was partially supported by NSF of China (11971061).

\end{document}